\documentclass[11pt]{amsart}

\usepackage[nomath,nothms,nomisc,nolayout,nohyperref,abbrs,nocolor]{mymacros}

\usepackage{amsmath,amssymb,bbm,mathtools}
\usepackage{xspace}
\usepackage{wrapfig}
\usepackage{xfrac}
\usepackage{verbatim}
\usepackage{thmtools, thm-restate}
\usepackage{subcaption}
\usepackage{fullpage}
\usepackage{enumerate}

\usepackage{soul}
\usepackage{microtype}

\usepackage[font=small,labelfont=bf]{caption}
\usepackage[bookmarks, colorlinks, breaklinks,
pdftitle={},pdfauthor={}]{hyperref}

\hypersetup{
  linkcolor=black,
  citecolor=black,
  filecolor=black,
  urlcolor=black}

\declaretheorem[numberwithin=section]{theorem}
\declaretheorem[sibling=theorem]{lemma}
\declaretheorem[sibling=theorem]{corollary}

\declaretheorem[sibling=theorem]{proposition}

\declaretheorem[sibling=theorem]{remark}
\declaretheorem[sibling=theorem]{definition}
\declaretheorem[sibling=theorem]{example}

\numberwithin{equation}{section}
\numberwithin{theorem}{section}

\usepackage{float}
\usepackage{cleveref}
\usepackage{tikz}

\usepackage{booktabs}

\usetikzlibrary{calc,decorations.markings}
\usetikzlibrary{decorations.pathmorphing}
\usetikzlibrary{decorations.pathreplacing}
\usetikzlibrary{patterns}
\usetikzlibrary{arrows}

\newcommand{\blank}[1]{}

\newcommand{\ghost}{\mathfrak{g}}

\newcommand{\conn}{\leftrightarrow}
\newcommand{\nconn}{\not\leftrightarrow}
\newcommand{\1}{{\bf 1}}

\newcommand{\Comp}{\operatorname{Comp}}

\newcommand{\obs}{\star}
\newcommand{\cVbulk}{\cV^\varnothing}
\newcommand{\cKbulk}{\cK^\varnothing}
\newcommand{\Vbulk}{V^\varnothing}
\newcommand{\Vobs}{V^\obs}
\newcommand{\uobs}{u^\obs}
\newcommand{\Kbulk}{K^\varnothing}
\newcommand{\Kobs}{K^\obs}
\newcommand{\consta}{{\sf c}_1}
\newcommand{\constb}{{\sf c}_2}
\newcommand{\consto}{{\sf c}}

\newcommand{\ellaj}{\ell_{a,j}}
\newcommand{\ellbj}{\ell_{b,j}}
\newcommand{\ellxj}{\ell_{x,j}}
\newcommand{\ellabj}{\ell_{ab,j}}
\newcommand{\ellaN}{\ell_{a,N}}

\newcommand{\ellxN}{\ell_{x,N}}
\newcommand{\ellabN}{\ell_{ab,N}}
\newcommand{\OL}{O_L}

\newcommand{\dplus}{(\frac{d-2}{2} \wedge 1)}
\newcommand{\clambda}{\lambda}
\newcommand{\cgamma}{\gamma}
\newcommand{\Loc}{\operatorname{Loc}}
\newcommand{\loc}{\operatorname{loc}}
\newcommand{\tVobs}{\tilde{V}^{\obs}}
\newcommand{\hVobs}{\hat V^\obs}

\title{Percolation transition for random forests in $d\geq 3$}
\date{}

\author{Roland Bauerschmidt}
\address{University of Cambridge, DPMMS; Courant Institute of Mathematical Sciences, NYU}
\email{bauerschmidt@cims.nyu.edu}

\author{Nicholas Crawford}
\address{The Technion, Department of Mathematics}
\email{nickc@technion.ac.il}

\author{Tyler Helmuth}
\address{Durham University, Department of Mathematical Sciences}
\email{tyler.helmuth@durham.ac.uk}

\begin{document}

\maketitle

\begin{abstract}
  The arboreal gas is the probability measure on (unrooted spanning)
  forests of a graph in which each forest is weighted by a factor
  $\beta>0$ per edge. It arises as the $q\to 0$ limit of
  the $q$-state random cluster model with $p=\beta q$.  We prove that in dimensions
  $d\geq 3$ the arboreal gas undergoes
  a percolation phase transition.  This contrasts
  with the case of $d=2$ where no percolation transition occurs.
  
  The starting point for our analysis is an exact relationship between
  the arboreal gas and a non-linear sigma model with target space the
  fermionic hyperbolic plane $\HH^{0|2}$.  This latter model can be
  thought of as the $0$-state Potts model, with the arboreal gas being
  its random cluster representation. Unlike the standard Potts
  models, the $\HH^{0|2}$ model has continuous symmetries. By
  combining a renormalisation group analysis with Ward identities we
  prove that this symmetry is spontaneously broken at low
  temperatures.  In terms of the arboreal gas, this symmetry breaking
  translates into the existence of infinite trees in the thermodynamic
  limit.  Our analysis also establishes massless free field
  correlations at low temperatures and the existence of a macroscopic
  tree on finite tori.
\end{abstract}

\setcounter{tocdepth}{2}
\tableofcontents

\section{Introduction}
\label{sec:intro}

This paper has two distinct motivations. The first is to study the
percolative properties of the \emph{arboreal gas}, and the second is
to understand \emph{spontaneously broken continuous symmetries}. We
first present our results from the percolation perspective, and then
turn to continuous symmetries.

\subsection{Main results for the arboreal gas}

The arboreal gas is the uniform measure on (unrooted spanning) forests
of a weighted graph.  More precisely, given an undirected graph
$G=(\Lambda,E)$, a forest $F=(\Lambda,E(F))$ is an acyclic
subgraph of $G$ having the same vertex set as $G$.  Given an edge weight
$\beta>0$ (inverse temperature) and a vertex weight $h\geq 0$
(external field), the probability of a forest $F$ under the arboreal
gas measure is
\begin{equation} \label{e:P-forest}
  \P_{\beta,h}^{G}[F] \bydef \frac{1}{Z_{\beta,h}^{G}} \beta^{|E(F)|} \prod_{T\in F} (1+h|V(T)|)
\end{equation}
where $T\in F$ denotes that $T$ is a tree in the forest, i.e., a
connected component of $F$, $|E(F)|$ is the number of edges in $F$,
and $|V(T)|$ is the number of vertices in $T$.  The arboreal gas is
also known as the (weighted) uniform forest model, as Bernoulli bond
percolation conditioned to be acyclic, and as the $q\to 0$ limit of
the $q$-state random cluster model with $p/q$ converging to $\beta$,
see \cite{MR2243761}.

We study the arboreal gas on a sequence of tori
$\Lambda_N = \Z^d/L^N\Z^d$ with $L$ fixed and $N\to\infty$.  To
simplify notation, we will use $\Lambda_N$ to denote both the graph
and its vertex set. From the percolation point of view, the most
fundamental question concerns whether a typical forest $F$ under the
law~\eqref{e:P-forest} contains a giant tree.  In all dimensions,
elementary arguments show that giant trees can exist only if $h=0$ and
if $\beta$ is large enough, in the sense that connection probabilities
decay exponentially whenever $h>0$ or $\beta$ is small; see
Appendix~\ref{sec:th0}.

The existence of a percolative phase for $h=0$ and $\beta$ large does
not, however, follow from standard techniques. The
subtlety of the existence of a percolative phase is perhaps best
evidenced by considering the case $d=2$: in this case giant trees do
not exist for any $\beta>0$~\cite{MR4218682}.  Our main result is that
for $d\geq 3$ giant trees do exist for $\beta$ large and $h=0$, and
that truncated correlations have massless free field decay.  To state
our result precisely, let $\{0 \conn x\}$ denote the event that $0$
and $x$ are connected, i.e., in the same tree.

\begin{theorem}  \label{thm:forest-macro}
  Let $d\geq 3$ and $L \geq L_0(d)$. Then there is
  $\beta_0 \in (0,\infty)$ such that
  for $\beta\geq \beta_0$
  there exist $\zeta_d(\beta) = 1-O(1/\beta)$,
  $\consto(\beta) = \consto + O(1/\beta)$ with $\consto>0$, and
  $\kappa>0$ such that 
  \begin{equation}
    \label{e:thm1}
    \P_{\beta,0}^{\Lambda_N}[0\conn x]
      = \zeta_d(\beta) + \frac{\consto(\beta)}{\beta|x|^{d-2}} +  O(\frac{1}{\beta|x|^{d-2+\kappa}}) + O(\frac{1}{\beta L^{\kappa N}}),
  \end{equation}
  where $|x|$ denotes the Euclidean norm.
\end{theorem}
Numerical evidence for this phase transition of the arboreal gas was
given in~\cite{PhysRevLett.98.030602}. More broadly our work was
inspired by
\cite{MR2110547,MR3622573,MR2567041,PhysRevLett.98.030602,MR2157859,MR2142209},
and we discuss further motivation later.

Although both the arboreal gas and Bernoulli bond percolation have
phase transitions for $d\geq 3$, the supercritical phases of these
models behave very differently: \eqref{e:thm1} shows that the arboreal
gas behaves like a critical model even in the supercritical phase, in
the sense that it has massless free field truncated correlation
decay. While this behaviour looks unusual when viewed through the lens
of supercritical percolation, it is natural from the viewpoint of
broken continuous symmetries.  We will return to this point in
Section~\ref{sec:spinintro}.

Theorem~\ref{thm:forest-macro} concerns the arboreal gas on large
finite tori in zero external field (i.e., $h=0$).  Another
limit to consider the arboreal gas in is the weak infinite volume
limit.  To this end, we consider the limit obtained by first taking
$N\to\infty$ with $h>0$ and then taking $h\downarrow 0$.  In a manner
similar to that for Bernoulli bond percolation in
\cite[Section~5]{MR378660} and \cite[Section~2.2]{MR874906}, the
external field is equivalent to considering the arboreal gas on an
extended graph
$G^{\ghost} = (\Lambda\cup\{\ghost\}, E\cup E^{\ghost})$ where
$E^{\ghost} = \Lambda\times\{\ghost\}$ and each edge in $E^{\ghost}$
has weight $h$. The additional vertex $\ghost$ is called the
\emph{ghost} vertex. The measure~\eqref{e:P-forest} is then obtained
by forgetting the connections to the ghost. This rephrases that the
product in~\eqref{e:P-forest} is equivalent to connecting a uniformly
chosen vertex in each tree $T$ to $\ghost$ with probability
$h|V(T)|/(1+h|V(T)|)$.  For vertices $x,y \in \Lambda$, we continue to
denote by $\{x\conn y\}$ the event that $x$ and $y$ are connected in
the random forest subgraph of $G$ with law \eqref{e:P-forest}, i.e.,
$\{x\conn y\}$ denotes the event that there is a path from $x$
to $y$ in the random subgraph, and that this (necessarily unique)
path does not contain $\ghost$. We write $\{x\conn \ghost\}$ to
denote the event that $x$ is connected to $\ghost$.

The event $\{0\conn\ghost\}$ is a finite volume proxy for the event
that the tree $T_0$ containing $0$ becomes infinite in the infinite
volume limit when $h\downarrow 0$. 
Indeed, let us define
\begin{equation}
  \label{eq:order}
  \theta_d(\beta)
  \bydef \lim_{h\downarrow 0}\lim_{N\to\infty} \P^{\Lambda_N}_{\beta,h}[0\conn \ghost],
\end{equation}
and let $\P_\beta^{\Z^d}$ be any (possibly subsequential) weak
infinite volume limit
$\lim_{h\downarrow 0}\lim_{N\to\infty}\P^{\Lambda_N}_{\beta,h}$.
Then 
\begin{equation}
\theta_d(\beta)=\P_\beta^{\Z^d}[|T_0|=\infty],
\end{equation}
see Proposition~\ref{prop:theta}.
By a stochastic domination argument it
is straightforward to show that 
\begin{equation}
  \label{eq:th0}
  \theta_d(\beta) =0 \qquad \text{for $0\leq \beta<p_c(d)/(1-p_c(d))<\infty$,}
\end{equation}
where $p_c(d)$ is the critical probability for Bernoulli bond
percolation on $\Z^d$, see Proposition~\ref{prop:theta0}. 
When $d=2$,
$\theta_2(\beta) = 0$ for all $\beta>0$ by
\cite[Section~4.2]{MR4218682}.  The next theorem shows that for
$d\geq 3$ the arboreal gas also has
a phase transition in this infinite volume limit. 

\begin{theorem}
  \label{thm:forests0}
  Let $d\geq 3$ and $L \geq L_0(d)$. Then there is
  $\beta_0 \in (0,\infty)$ such that
  for $\beta\geq \beta_0$
  the limit \eqref{eq:order} exists and 
  \begin{equation}
    \label{e:thm-forest-theta}
    \theta_{d}(\beta)^2 = \zeta_d(\beta)
    = 1-O(1/\beta),
  \end{equation}
  where $\zeta_d(\beta)$ is the finite volume density of the tree containing $0$
    from Theorem~\ref{thm:forest-macro}.
\end{theorem}

In fact, our proof shows that $\theta_d(\beta)\sim 1-c/\beta$
with $c=(-\Delta^{\Z^d})^{-1}(0,0)>0$ the expected time a simple random
walk spends at the origin.  This behaviour is different from that of
Bernoulli bond percolation and more generally that of the random
cluster model with $q>0$.  For these models the percolation
probability is governed by Peierls' contours and is $1-O((1-p)^{2d})$
by \cite[Remark~5.10]{MR2410872}.

That the arboreal gas behaves critically within its supercritical
phase can be further quantified in terms of the following truncated
two-point functions:
\begin{align}
  \label{e:tau-def}
  \tau_{\beta}(x) = \lim_{h\downarrow 0} \tau_{\beta,h}(x),
  &  \qquad
    \tau_{\beta,h}(x)=\lim_{N\to\infty}\P_{\beta,h}^{\Lambda_N}[0\conn x, 0 \nconn\ghost],\\
    \label{e:sigma-def}
  \sigma_{\beta}(x) = \lim_{h\downarrow 0} \sigma_{\beta,h}(x),
  & \qquad
    \sigma_{\beta,h}(x) = \lim_{N\to\infty}\pB{\P_{\beta,h}^{\Lambda_N}[0\nconn \ghost]^2 - \P_{\beta,h}^{\Lambda_N}[0\nconn x, 0\nconn \ghost,x\nconn\ghost]}.
\end{align}

\begin{theorem} \label{thm:forests}
  Under the assumptions of Theorem~\ref{thm:forests0}, for $\beta\geq \beta_0$,
  the limits \eqref{e:tau-def}--\eqref{e:sigma-def} exist and there exist
  constants ${\sf c}_i(\beta) = {\sf c}_i+O(1/\beta)$ and $\kappa>0$ such that
  \begin{align}
    \label{e:thm-forest-tau}
    \tau_{\beta}(x) &= \frac{\consta(\beta)}{\beta|x|^{d-2}} +
                      O(\frac{1}{\beta |x|^{d-2+\kappa}}), 
    \\
    \label{e:thm-forest-sigma}
    \sigma_{\beta}(x) &= \frac{\constb(\beta)}{\beta^2|x|^{2d-4}} + O(\frac{1}{\beta^2 |x|^{2d-4+\kappa}}).
  \end{align}
  The constants satisfy
  $(\constb(\beta)/\consta(\beta)^2)\theta_d(\beta)^2=1$  and
  $\consto(\beta) = 2\consta(\beta)$, $\consto(\beta)$ from
  Theorem~\ref{thm:forest-macro}.
\end{theorem}

Further results could be deduced from our analysis, but to maintain
focus we have not carried these out in detail.  We mention some of
them below in Section~\ref{sec:conn-open-direct} when discussing our
results and open problems.

\subsection{The $\HH^{0|2}$  model and its continuous symmetries}
\label{sec:spinintro}

In \cite{MR2110547,MR3622573}, the arboreal gas was related to a
fermionic field theory and a supersymmetric non-linear sigma model
with target space one half of the degenerate super-sphere
$\mathbb S^{0|2}$. In \cite{MR4218682} this was
reinterpreted as a non-linear sigma model with hyperbolic target space
$\HH^{0|2}$, which we refer to as the $\HH^{0|2}$
model for short.   The reinterpretation was essential in
\cite{MR4218682}; it is less essential for the present work, but
nevertheless we continue to use the $\HH^{0|2}$ formulation of the model.  

Briefly, the $\HH^{0|2}$
model is defined as follows, see
\cite[Section~2]{MR4218682} for further details.  For every vertex
$x\in \Lambda$, there are two (anticommuting) Grassmann variables
$\xi_x$ and $\eta_x$ and we then set
\begin{equation}
  z_x \bydef \sqrt{1-2\xi_x\eta_x} \bydef 1-\xi_x\eta_x.
\end{equation}
Thus the $z_x$ commute with each other and with the odd elements
$\xi_x$ and $\eta_x$.  The formal triples
$u_x\bydef(\xi_x,\eta_x,z_x)$ are supervectors with two odd components
$\xi_x,\eta_x$ and an even component $z_x$.  These supervectors
satisfy the sigma model constraint $u_x\cdot u_x=-1$ for the super
inner product
\begin{equation}
  \label{e:uxdotuy}
  u_x\cdot u_y \bydef -\xi_x\eta_y-\xi_y\eta_x - z_xz_y.
\end{equation}
In analogy with the tetrahedral representation of the $q$-state Potts
model, see~\cite[Section~2.2]{MR2581604}, the sigma model constraint
can be thought of as $u_x\cdot u_x = q-1$ with $q=0$. The constraint
is also reminiscent of the embedding of the hyperbolic space $\HH^2$
in $\R^3$ equipped with the standard quadratic form with Lorentzian
signature $(1, 1, -1)$. Indeed, $-\xi_x\eta_y-\xi_y\eta_x$ is the
fermionic analogue of the Euclidean inner product on $\R^2$.

Let $F$ be a (non-commutative) polynomial in the variables
$\{\xi_{x},\eta_{x}\}_{x\in \Lambda}$. The expectation of $F$ with
respect to the $\HH^{0|2}$ model is
\begin{equation}
  \label{e:h02-def}
  \avg{F}_{\beta,h} \bydef \frac{1}{Z_{\beta,h}}\int (\prod_{x\in\Lambda} \partial_{\eta_x}\partial_{\xi_x} \frac{1}{z_x}) e^{\frac{\beta}{2} (u,\Delta u) - h(1,z-1)} F
  .
\end{equation}
In this expression,
$\int \prod_{x\in\Lambda} \partial_{\eta_x}\partial_{\xi_x}$ denotes
the Grassmann integral (i.e., the coefficient of the top degree
monomial of the integrand), $Z_{\beta,h}$ is a normalising constant, and 
\begin{equation}
  \frac12
  (u,\Delta u) = -\frac12 \sum_{xy\in E(\Lambda) }
  (u_x-u_y)\cdot(u_x-u_y)
  = \sum_{xy\in E(\Lambda) }
  (u_x\cdot u_y+1),
  \qquad (1,z) = \sum_{x\in \Lambda} z_x,
\end{equation}
where $xy\in E(\Lambda)$ denotes that $x$ and $y$ are nearest
neighbours (counting every pair once), and the inner products are
given by \eqref{e:uxdotuy}.  The factors $1/z_x$ in \eqref{e:h02-def}
are the canonical fermionic volume form invariant under the symmetries
associated with \eqref{e:uxdotuy} as discussed further below.

As explained in \cite[Section~2.1]{MR4218682} (see also
\cite{MR2110547} where such relations were first observed) connection
and edge probabilities of the arboreal gas are equivalent to
correlation functions of the $\HH^{0|2}$ model.  The following
proposition summarises the relations we need, see
Appendix~\ref{app:h02forest} for the proof.

\begin{proposition}
  \label{prop:h02-forest}
  For any finite graph $G$,
  any $\beta \geq 0$ and $h \geq 0$,
  \begin{align} \label{e:zforest}
  \P_{\beta,h}[0\conn \ghost] &= \avg{z_0}_{\beta,h},
    \\
    \label{e:xietaforest}
    \P_{\beta,h}[0\conn x, 0 \nconn \ghost] &= \avg{\xi_0\eta_x}_{\beta,h},\\
  \label{e:u0uxforest}
  \P_{\beta,h}[0\conn x]+\P_{\beta,h}[0\nconn x, 0\conn \ghost, x\conn\ghost]  &= -\avg{u_0\cdot u_x}_{\beta,h}
  ,
  \end{align}
  and the normalising constants in \eqref{e:P-forest} and \eqref{e:h02-def} are equal.
  In particular,
  \begin{equation} \label{e:P0x4pt}
    \P_{\beta,0}[0\conn x] = -\avg{u_0\cdot u_x}_{\beta,0} = -\avg{z_0z_x}_{\beta,0} = \avg{\xi_0\eta_x}_{\beta,0} = 1-\avg{\xi_0\eta_0\xi_x\eta_x}_{\beta,0}.
  \end{equation}
\end{proposition}

These relations resemble those between the Potts model and the random
cluster model, giving further credence to our proposal that the
$\HH^{0|2}$ model may be interpreted as the $0$-state Potts model,
with the arboreal gas playing the role of the $0$-state random cluster
model.  Nevertheless, there are important differences from the
$q$-state Potts model with $q \geq 2$.  Chief amongst them is that the
$\HH^{0|2}$ model has continuous symmetries. To make this precise, let
\begin{equation} \label{e:Tdef}
  T=\sum_{x\in\Lambda} z_x \partial_{\xi_x}, \qquad
  \bar T=\sum_{x\in\Lambda} z_x \partial_{\eta_x}
  .
\end{equation}
One way to understand the significance of $T, \bar{T}$ is via the
identities $\avg{TF}_{\beta,0}=\avg{\bar T F}_{\beta,0}=0$ for any
polynomial $F$ in the variables $\xi$ and $\eta$; see
\cite[Section~2.2]{MR4218682}. For example,
$\avg{T\xi_{0}}_{\beta,0}=\avg{z_{0}}_{\beta,0}=0$. Identities derived
in this way are conventionally called Ward identities.

The maps $T$ and $\bar T$ are infinitesimal generators of two global
internal supersymmetries of the $\HH^{0|2}$ model.  These
supersymmetries are explicitly broken if $h \neq 0$.  They are
analogues of infinitesimal Lorentz boosts or infinitesimal
rotations. Together with a further internal symmetry corresponding to
rotations in the $\xi, \eta$ plane, these operators generate the
symmetry algebra $\mathfrak{osp}(1|2)$ of the $\HH^{0|2}$ model.  For
details and further explanations we again refer to
\cite[Section~2.2]{MR4218682}.  As generators of continuous
symmetries, $T$ and $\bar T$ imply Ward identities that are not
available for the Potts model with $q\geq 2$. These identities are
crucial for our analysis and will be discussed below.

The phase transition of the arboreal gas corresponds to a spontaneous
breaking of the above supersymmetries in the infinite volume
limit. By spontaneous symmetry breaking we mean that there is
an observable $F$ for which
$\lim_{N\to\infty}\lim_{h\downarrow 0}\avg{F}_{\beta,h}\neq \lim_{h\downarrow 0}\lim_{N\to\infty}\avg{F}_{\beta,h}$.
Indeed, this is shown in our next theorem for the $\HH^{0|2}$ model from which
Theorems~\ref{thm:forests0} and~\ref{thm:forests} follow immediately
by \eqref{e:zforest}--\eqref{e:u0uxforest} (except for the same
statements relating the constants, which we omitted here). A similar
reformulation applies to Theorem~\ref{thm:forest-macro}.

\begin{theorem}\label{thm:h02}
  Let $d\geq 3$ and $L\geq L_0(d)$. There exists  $\beta_0\in (0,\infty)$ and constants $\theta_d(\beta)=1+O(1/\beta)$
  and ${\sf c}_i(\beta) = {\sf c}_i+O(1/\beta)$ and $\kappa >0$ (all dependent on $d$)
  such that for $\beta \geq \beta_0$,
  \begin{align}
    \lim_{h\downarrow 0}\lim_{N\to\infty}\avg{z_0}_{\beta,h}
    &= \theta_d(\beta)
    \\
    \lim_{h\downarrow 0}\lim_{N\to\infty}
    \avg{\xi_0\eta_x}_{\beta,h}
    &= \frac{\consta(\beta)}{\beta |x|^{d-2}} + O(\frac{1}{\beta |x|^{d-2+\kappa}})
    \\
    \lim_{h\downarrow 0}\lim_{N\to\infty}\pB{\avg{z_0z_x}_{\beta,h}-\avg{z_0}_{\beta,h} \avg{z_x}_{\beta,h}}
    &= -\frac{\constb(\beta)}{\beta^2 |x|^{2d-4}} + O(\frac{1}{\beta^2|x|^{2d-4+\kappa}})
      .
  \end{align}
  In particular,
  \begin{equation}
    \lim_{h\downarrow 0}\lim_{N\to\infty}
    \avg{u_0\cdot u_x}_{\beta,h} = -\theta_d(\beta)^2 -\frac{2\consta(\beta)}{\beta |x|^{d-2}} + O(\frac{1}{\beta |x|^{d-2+\kappa}}).
  \end{equation}
\end{theorem}

In fact, the constants ${\sf c}_{i}(\beta)$ both satisfy ${\sf
  c}_{i}(\beta) = (c_{d})^{i} + O(1/\beta)$, where $c_{d}$ is the
leading constant in the asymptotics of the Green function of the
Laplacian $-\Delta^{\Z^{d}}$ on $\Z^{d}$:
\begin{equation}
  \label{eq:LapZd}
  (-\Delta^{\Z^d})^{-1}(0,x)
  = \frac{c_{d}}{|x|^{d-2}} + O(|x|^{-(d-2)-1}).
\end{equation}

Our proof of Theorem~\ref{thm:h02} is by a rigorous renormalisation
group analysis aided by Ward identities.  We set
$\psi=\sqrt{\beta}\eta$ and $\bar\psi = \sqrt{\beta}\xi$. 
Algebra then shows the fermionic density in
\eqref{e:h02-def} is equivalent to
\begin{equation}
  \label{e:h02-psi}
  \exp\qa{- (\psi,-\Delta\bar\psi)
    -\frac{1}{\beta}(1+h)  \sum_{x\in\Lambda}\psi_x\bar\psi_x - \frac{1}{2\beta} \sum_{x\in\Lambda} \psi_x\bar\psi_x \sum_{e\in \mathcal{E}_d} \psi_{x+e} \bar\psi_{x+e}
  },
\end{equation}
where the $1$ in the quadratic term arises from putting the $\bbH^{0|2}$ volume form (see   \eqref{e:h02-def})
into the exponential, i.e.,
\begin{equation}
   \prod_{x\in\Lambda}\frac{1}{z_x} = \prod_{x\in\Lambda} e^{+\xi_x\eta_x}=\prod_{x\in\Lambda}e^{-\eta_x\xi_x} = \exp\qa{-\frac{1}{\beta}\sum_{x\in\Lambda}\psi_x\bar\psi_x},
\end{equation}
and $\mathcal{E}_{d}=\{e_{1},\dots, e_{2d}\}$ are the standard unit
vectors (where $e_{d+j}=-e_{j}$). The reformulation~\eqref{e:h02-psi} looks
very much like a fermionic version of the $\varphi^4$ spin
model. However, the following differences are important:

(1) Due to the fermionic nature of the field, and because the
  fermionic field only has two components (different for example from
  the case of Dirac fermions with four components), the quartic term
actually has gradients in it: denoting the discrete gradient in
direction $e \in \cE_d$ by
$(\nabla_e\psi)_x = \psi_{x+e}-\psi_x$, the quartic term can be
written as
\begin{equation}
  \label{e:quartic}
  \frac12
  \psi_x\bar\psi_x \sum_{e\in \mathcal{E}_d} \psi_{x+e}\bar\psi_{x+e}
  =
   \frac12
  \psi_x\bar\psi_x \sum_{e\in \mathcal{E}_d} (\nabla_e \psi)_x(\nabla_e \bar\psi)_x
  \bydef
  \psi_x\bar\psi_x (\nabla \psi)_x(\nabla \bar\psi)_x,
\end{equation}
where we introduced the shorthand notation
$(\nabla \psi)_x(\nabla \bar\psi)_x= \frac12 \sum_{e\in \mathcal{E}_d}
(\nabla_e \psi)_x(\nabla_e \bar\psi)_x$.

(2) The coupling constants $\frac{1}{\beta}(1+h)$ of the
quadratic and $\frac{1}{\beta}$ of the quartic terms are
related, and they are equal in the case $h=0$ of ultimate
interest. This relation is due to the geometric origin of the
model as a non-linear sigma model and analogous relations are present
in intrinsic coordinates for other sigma models like the vector $O(n)$
model. We remark that if the coupling constant of the quartic
term was much smaller than that of the quadratic term (so $h\gg 0$)
the study of the model would reduce to an exercise in fermionic
cluster expansions (but see Appendix~\ref{sec:th0} for even simpler
arguments in this case).

To study the case of equal coupling constants,  we will
first consider their renormalisation group trajectories as a one parameter family among the set  of all renormalisation group trajectories 
obtained by allowing the initial quadratic and quartic couplings to vary
independent of one another.  We will then place the  equal-initial-coupling trajectories
on the critical manifold of the renormalisation group dynamical system using
the following Ward identity for the $\HH^{0|2}$ model:
\begin{equation} \label{e:ward}
  \avg{z_0}_{\beta,h} = \avg{T\xi_0}_{\beta,h} 
  = -\sum_{x\in\Lambda} h \avg{\xi_0 Tz_x}_{\beta,h}
  = h \sum_{x\in\Lambda} \avg{\xi_0 \eta_x}_{\beta,h},
\end{equation}
where $T$ is the symmetry generator \eqref{e:Tdef}; in the second
equality we have used~\cite[Lemma~2.3]{MR4218682}.

After taking into account the two points above (in particular
that the flow of the expanding quadratic term is constrained by the
Ward identity), power counting heuristics predict that the lower
critical dimension for spontaneous symmetry breaking with free field
low temperature fluctuations is two for the $\HH^{0|2}$ model. 
We expect that these considerations generalise to all non-linear
sigma models with continuous symmetry, in agreement with the
Goldstone mechanism.  In conjunction with~\cite{MR4218682}, our
results rigorously establish that the lower critical dimension is two
for the $\HH^{0|2}$ model.

\subsection{Background on non-linear sigma models and renormalisation}
The low temperature renormalisation group analysis of non-linear sigma
models with non-abelian continuous symmetry is a notorious problem
that was famously considered by Balaban for the case of $O(n)$
symmetry, see~\cite{MR1670021,MR1669693} and references therein.  Our
comparatively simple analysis of the $\HH^{0|2}$ model, which is a
non-linear sigma model with non-abelian continuous $OSp(1|2)$
symmetry, is made possible mainly by the fact that it does not suffer
a large field problem because it has a fermionic representation.  Our
approach to the $\HH^{0|2}$ model differs from Balaban's approach to
the $O(n)$ model on a conceptual level, in that it is based on
\emph{intrinsic} coordinates as opposed to \emph{extrinsic} ones.
In the extrinsic approach of Balaban the sphere $\bbS^{n-1}$ is
embedded into $\R^n$ and the renormalised action evolves as a
function of $\R^n$-valued fields, manifestly preserving $O(n)$
symmetry.  The intrinsic approach we use, which is similar to the
one taken in the physics literature (see \cite{MR1227790}), is based
on local coordinates for the target space $\bbH^{0|2}$. The
renormalised action does not have the $OSp(1|2)$ symmetry of the
model, and Ward identities are used \emph{a posteriori} to enforce
the essential constraints (relations between couplings) due to the
symmetry. It is unclear to us how to implement an extrinsic
approach in our situation of $OSp(1|2)$ symmetry, and more
generally for noncompact symmetries.

Somewhat remarkably, despite its
simplicity, the $\HH^{0|2}$ model has all of the main features present
in the non-abelian $O(n)$ models, including: absence of spontaneous
symmetry breaking in 2d (proven in \cite{MR4218682}); mass generation
in 2d (conjectured in~\cite{MR3622573}); and a spontaneous symmetry
breaking phase transition with massless low temperature fluctuations
in $d\geq 3$ (the main result of this work).

The $\HH^{0|2}$ model is a member of the family of hyperbolic sigma
models with target spaces $\HH^{n|2m}$, see~\cite{1912.05817} for a
discussion of some aspects of this.
By supersymmetric localisation the observables of the $\HH^{0|2}$
model considered in Theorem~\ref{thm:h02} are equivalent to the
analogous ones of the non-linear sigma model with target
$\HH^{2|4}$. While this relation does not play a role in this paper,
it leads to a more direct representation of the continuous symmetry
breaking observed here.  In brief, in the $\HH^{2|4}$ model each
vertex comes equipped with two real and four Grassmann fields.  By
expressing these fields in horospherical coordinates one of the real
fields and the four Grassmann fields can be integrated out. The
marginal distribution of the remaining real field, which is called the
$t$-field, may be viewed as a `$\nabla\phi$' random surface model,
albeit with a nonconvex and nonlocal Hamiltonian.  By this we mean
that the potential is invariant under the global translation
$t_x\mapsto t_x +r$ for $r\in \R$.  See~\cite{MR4218682} for more
details, where this perspective was used to prove the absence of
symmetry breaking in $d=2$.  The full $\HH^{n|2m}$ family has been
important for advancing our understanding of other aspects of these
models~\cite{MR4218682,1912.05817}. 
Of particular note, we mention that the $\HH^{2|2}$ model has received
substantial prior attention due to its exact connection to linearly
reinforced random walks and its motivation from random matrix theory,
see \cite{MR3420510,MR863830,MR1134935,MR2104878,MR2728731}.

For hyperbolic sigma models with target $\HH^n$, $n \geq 1$,
spontaneous symmetry breaking for all $\beta >0$ was shown in
\cite{MR2104878}, and with target $\HH^{2|2}$ for $\beta$ large in
\cite{MR2728731} (see also \cite{MR3366053}).  For motivation from
random matrix theory and the Anderson transition see
\cite{MR2953867,MR3204347}.  These proofs make essential use of the
horospherical coordinates mentioned above.  Moreover, the proof of
symmetry breaking for the $\HH^{2|2}$ model in \cite{MR2728731} relies
on an infinite number of Ward identifies resulting from supersymmetric
localisation.  These identities are absent in the $\HH^{0|2}$ model,
limiting the applicability of the methods of~\cite{MR2728731} to our
setting. At the same time, the $\HH^{2|2}$ model has no purely
fermionic representation, and so our methods do not apply there, at
least without significant further developments.

Introductions to fermionic renormalisation include
\cite{MR1380265,MR2414874,MR1658669}, see also
\cite{10.1007/JHEP01(2021)026}.  Recent probabilistic applications of
these approaches to fermionic renormalisation include the study of
interacting dimers \cite{MR3606736,MR4121614} and two-dimensional
finite range Ising models \cite{MR3116322,MR2905796,MR4396672,MR4538288}.  Our
organisation of the renormalisation group
is instead based on a finite range decomposition and
polymer coordinates, and follows \cite{MR2523458}
and its further developments in
\cite{MR3332938,MR3332939,MR3332940,MR3332941,MR3332942,MR3317791,MR2070102,MR3129804}. 
This approach has its origins in \cite{MR1048698}.  For an
introduction to this approach in a hierarchical bosonic context see
\cite{MR3969983}.  Previous applications of this approach include the study of
4d weakly self-avoiding walks \cite{MR3339164,MR3345374}; the
nearest-neighbour critical 4d $|\varphi|^4$ model
\cite{MR3269689,MR3459163} and long-range versions thereof
\cite{MR3772040,MR3723429}; the ultraviolet $\varphi^4_3$ problem
\cite{MR1346375,MR2004988}; analysis of the Kosterlitz--Thousless
transition of the 2d Coulomb gas \cite{MR1777310,MR2917175}; the
Cauchy--Born problem \cite{1910.13564}; and others.

While the construction of the bulk renormalisation group flow is
simpler for the intrinsic representation of the $\HH^{0|2}$ model than
in many of the previous references, a crucial
novelty of our present work is the combination of the finite range
renormalisation group approach with Ward identities, together with a
precise analysis of a nontrivial zero mode. 
This has enabled us to apply these methods to a non-linear sigma model in the
phase of broken symmetry. It would be extremely interesting to
understand this approach for bosonic non-linear sigma models where,
while
`large fields' cause serious complications, the formal
perturbative analysis is very much in parallel to the fermionic
version we study in this paper.
Ward identities of a
different type have previously been used in the renormalisation group
analyses in \cite{MR880416} and \cite{MR1947693} and many follow-up
works including \cite{MR3606736,MR4121614}.  Finally, we mention that
Theorem~\ref{thm:forest-macro} yields quantitative finite volume
statements. The proof implements a rigorous finite size analysis along
the lines of that proposed in \cite{10.1016/0550-3213(85)90379-7}. It
would be very interesting to extend this to even higher precision as
discussed in Section~\ref{sec:conn-open-direct} below.

\subsection{Future directions for the arboreal gas}
\label{sec:conn-open-direct}

In this section we discuss several interesting open directions,
including the geometric structure of the weak infinite volume limits
of the arboreal gas and its relation to the uniform spanning tree, and
a conjectural finite size universality similar to Wigner--Dyson
universality from random matrix theory.

\subsubsection*{Finite volume behaviour}

The detailed finite volume behaviour of the arboreal gas would be very
interesting to understand beyond the precision of
Theorem~\ref{thm:forest-macro}.  On the complete graph at
supercritical temperatures it is known that there is a unique macroscopic
cluster, and that there are an unbounded number of clusters whose
sizes are of order $|\Lambda|^{2/3}$~\cite{MR1167294}. The
fluctuations of the macroscopic cluster are non-Gaussian of scale
$|\Lambda|^{2/3}$ and the distribution of the ordered cluster sizes of
the mesoscopic clusters has been determined~\cite{MR1167294}. The
joint law of the mesoscopic clusters can be
characterised~\cite[Section~1.4.3]{MR3845513}. Intriguingly, $|\Lambda|^{2/3}$ is the size of the largest tree at criticality on
the complete graph, and the order statistics of the supercritical
mesoscopic clusters can be related to the order statistics at
the critical point~\cite[Section~1.4.3]{MR3845513}.

Going beyond the complete graph, is this distribution of ordered
cluster sizes universal, at least in sufficiently high dimensions?
This would be similar to the conjectured universality of Wigner--Dyson
statistics from random matrix theory~\cite{MR1843511} or the conjectured universality of the distribution of
  macroscopic loops in loop representations of $O(n)$ (and other) spin systems~\cite{NahumEtAl,MR2868048}. 
More generally it would be an instance of the
universality of low temperature fluctuations in finite volume
in models with continuous symmetries.

Finally, we mention that on expander graphs the existence of a phase transition for the arboreal gas is not difficult to show by using
a natural split--merge dynamics \cite{MR3451159}.
It would be interesting if this dynamical approach
could also be used to obtain information about the cluster size distribution.

\subsubsection*{Infinite volume behaviour and relation to the uniform spanning tree}

As mentioned previously, the arboreal gas is also known as the
\emph{uniform forest model} \cite{MR2243761}. We emphasise that the
arboreal gas is \emph{not} what is typically known as the
\emph{uniform spanning forest} (USF), which is in fact the weak limit
as $\Lambda_N \uparrow \Z^d$ of a uniform spanning tree (UST)
\cite{MR1127715}.  On a finite graph, the UST is the $\beta\to \infty$
limit of the arboreal gas. The correct scaling of the external field
for this limit is $h=\beta \kappa$ and we thus write
$\P_{\text{UST},\kappa} = \lim_{\beta\to\infty}
\P_{\beta,\beta\kappa}$ for the UST on a finite graph (plus ghost
vertex if $\kappa>0$).  For $\kappa>0$, this measure is also known as
the \emph{rooted} spanning forest, because disregarding the
connections to the ghost vertex disconnects the tree of the UST, with
vertices previously connected to the ghost becoming roots.  The
distributions of rooted and unrooted forests are not the same. To
help prevent confusion we will refer to the rooted spanning
forests as (a special case of) the UST.

It is trivial that $\P_{\text{UST},0}^{\Lambda_N}[0\conn x]=1$.
Nevertheless, the behaviour of the UST in the weak infinite volume limit depends
on the dimension $d$.
This limit can be defined as
$\P_{\rm UST}^{\Z^d} =\lim_{\kappa \downarrow 0}\lim_{N\to\infty}
\P_{\text{UST},\kappa}^{\Lambda_N}$
and is independent of
the finite volume boundary conditions (e.g. free, wired, or periodic as above)
imposed on $\Lambda_N$, see \cite{MR1127715}. Even though the function
$1_{0\conn x}$ is not continuous with respect to the topology of weak
convergence, it is still true that
\begin{equation}
  \P_{{\rm UST}}^{\Z^d}[0\conn x] = \lim_{\kappa \downarrow 0}\lim_{N\to\infty}
  \P_{{\rm UST},\kappa}^{\Lambda_N}[0 \conn x].
\end{equation} 
The order of limits here is essential.  In this infinite volume limit
the UST disconnects into infinitely many infinite trees if $d>4$, but
remains a single connected tree if $d\leq 4$, see \cite{MR1127715}.
Moreover,
\begin{equation}
  \P_{\rm UST}^{\Z^d}[0 \conn x] + \P_{\rm UST}^{\Z^d}[0\nconn x,
  |T_0|=\infty, |T_x|=\infty] = 1. 
\end{equation}
On the left-hand side, the second term vanishes if $d\leq 4$ whereas  the first term tends to $0$ as $|x|\to\infty$ if $d>4$.
Furthermore, the geometric structure of the trees under $\P_{\rm UST}^{\Z^d}$ is well understood.
In particular, all trees are one-ended, meaning that
removing one edge from a tree results in two trees, of which one is finite
\cite{MR1127715,MR1825141}.

For the arboreal gas, the existence and uniqueness of infinite volume
limits is an open question.  Nonetheless, subsequential limits exist,
and in such an infinite volume limit all trees are finite almost
surely when $\beta$ is small, while Theorem~\ref{thm:forests0} implies
the existence of an infinite tree for $\beta$ large.  Moreover, by
Theorem~\ref{thm:forests},
\begin{equation} \label{e:u0ux-asymp}
  \lim_{h\downarrow 0}\lim_{N\to\infty}\pbb{ \P_{\beta,h}^{\Lambda_N}[0 \conn x] + \P_{\beta,h}^{\Lambda_N}[0\nconn x, 0\conn\ghost, x\conn\ghost]}
  = \theta_d(\beta)^2 + \frac{2\consta(\beta)}{\beta |x|^{d-2}} + O(\frac{1}{\beta |x|^{d-2+\kappa}}).
\end{equation}
By analogy with the UST, we expect that only the first term on the
left-hand side contributes for $d\leq 4$ and that only the second term
contributes asymptotically as $|x|\to\infty$ for $d>4$.  The tempting
conjecture that the UST stochastically dominates the arboreal gas on
the torus is consistent with these expectations.  The analogue of the
left-hand side of \eqref{e:u0ux-asymp} plays an important role in the
proof of uniqueness of the infinite cluster in Bernoulli percolation
in \cite{MR901151}; this is related to the vanishing of the second
term.  As already mentioned, for the arboreal gas we only expect this
to be true in $d\leq 4$. Significant progress towards this
statement has been obtained in~\cite{2302.12224}, where it is shown that
translation-invariant infinite volume limits of the arboreal gas
have a unique infinite tree in $d=3,4$. More precisely, \cite{2302.12224}
makes use of the existence results of the
present paper and establishes uniqueness.

Beyond the questions above, it would be interesting to analyse
more detailed geometric aspects of the arboreal gas. For example, can one
construct scaling limits as has been done for some spanning tree models~\cite{MR1712629,MR1716768,MR3857861,MR4348685}?

Finally, we mention that a detailed analysis of the infinite volume
behaviour of the arboreal gas on regular trees with wired boundary
conditions has been carried out \cite{MR4402795,MR4460596}.  This infinite
volume behaviour is consistent with the finite volume behaviour of the
complete graph, e.g., at all supercritical temperatures the sizes of
finite clusters have the same distribution as those of critical
percolation.

\subsubsection*{Order of phase transition}

Our analysis could be extended to a detailed study of the approach
$h\downarrow 0$. To keep the length of this paper within bounds, we do
not carry this out, but here briefly comment on what we expect can be
shown by extensions of our analysis.
As discussed above, a natural object is the magnetisation 
\begin{equation}
  M(\beta,h) = \lim_{N\to\infty} M_{N}(\beta,h), \qquad M_{N}(\beta,h) = \P^{\Lambda_{N}}_{\beta,h}[0\conn\ghost],
\end{equation}
and the corresponding susceptibility (neglecting questions concerning
the order of limits) 
\begin{equation}
  \chi(\beta,h)
  = \ddp{}{h}M(\beta,h)
  = \sum_{x} \sigma_{\beta,h}(x).
\end{equation}
Thus for the arboreal gas, the susceptibility is not the sum over
$\tau_{\beta,h}(x)$ as is the case for Bernoulli bond percolation, but
the sum over $\sigma_{\beta,h}(x)$. In terms of the sigma model,
$\chi$ maybe viewed as the longitudinal susceptibility, often denoted $\chi_{||}$.
In this interpretation, the sum over $\tau_{\beta,h}(x)$ is the transversal susceptibility
$\chi_{\perp}$ and satisfies the Ward identity
$\chi_{\perp}(\beta,h)=\sum_{x} \tau_{\beta,h}(x) = h^{-1}M(\beta,h)$ 
which is crucial in our analysis. For the longitudinal
susceptibility, we expect that it would be possible to extend our
analysis to show
\begin{equation}
    \chi(\beta,h) \sim \begin{cases}
    C(\beta) h^{-1/2} & (d=3)
    \\
    C(\beta) |\log h|& (d=4)
    \\
    C(\beta) & (d>4).
  \end{cases}
\end{equation}
Defining the \emph{free energy}
$f(\beta,h)=\lim_{N\to\infty} |\Lambda_N|^{-1}\log
Z_{\beta,h}^{\Lambda_{N}}$, for $\beta \geq \beta_0$ the previous asymptotics suggest that
$h\mapsto f(\beta,h)$ is $C^2$ in $d>4$ but only $C^1$ for $d=3,4$.
In fact, extrapolating from our renormalisation group analysis we believe that
for $\beta \geq \beta_0$ the free energy is $C^n$ but not $C^{n+1}$ as
a function of $h\geq 0$ for $n=\floor{\frac{d-1}{2}}$.  It is unclear
how this is connected to the geometry of the component graph of the
UST, which also changes as the dimension is
varied~\cite{MR2123930,MR4009707}.

\subsubsection*{Critical behaviour}

The critical behaviour of the $\bbH^{0|2}$ model and its generalisations (the $\bbH^{0|2M}$ models)
were studied in \cite{MR3429411,MR4382903}, using $\epsilon$-expansions formally continued from the $O(n)$ models,
with the motivation of being candidates for the CFTs relevant for
a dS-CFT correspondence. Rigorous results about the critical
behaviour of the arboreal gas on $\Z^{d}$ for $d\geq 3$ would be
very interesting.

\subsection{Organisation and notation}

This paper is organised as follows.  In Section~\ref{sec:ward}, we
show how Theorem~\ref{thm:h02} is reduced to renormalisation group
results with the help of the Ward identity \eqref{e:ward}. 
The main renormalisation group input is Theorem~\ref{thm:psi4}
and~\ref{thm:psi4-full}. Sections~\ref{sec:bulk}--\ref{sec:proofs}
then prove these renormalisation group results.
Section~\ref{sec:bulk} is concerned with the construction of the bulk
renormalisation group flow, and Section~\ref{sec:suscept} uses this
analysis to compute the susceptibility. In
Section~\ref{sec:obs-flow} we extend this construction to
include observables. The renormalisation group flow for observables
is then used in Section~\ref{sec:obs} to compute pointwise correlation
functions. These computations involve a precise analysis of
the zero mode. The short Section~\ref{sec:proofs} then collects the
results and establishes Theorems~\ref{thm:psi4}
and~\ref{thm:psi4-full}. Finally, in Appendix~\ref{app:h02forest}
we collect relations between the arboreal gas and the $\HH^{0|2}$
model as well as basic percolation and high temperature properties
of the arboreal gas, and in Appendix~\ref{app:decomp} we
include some background material about the finite range
decomposition that we use.

Throughout we use $a_{n}\sim b_{n}$ to denote
$\lim_{n\to\infty}a_{n}/b_{n}=1$, $a_{n}\asymp b_{n}$ to denote the
existence of $c,C>0$ such that $c a_{n}\leq b_{n}\leq Ca_{n}$,
$a_n \lesssim b_n$ if $a_n \leq Cb_n$, and $a_{n}=O(b_{n})$ if
$|a_{n}|\lesssim |b_{n}|$. We write $a\wedge b = \min \{a,b\}$.
We consider the dimension $d \geq 3$
to be fixed, and hence allow implicit constants to depend on $d$.  In
Sections~1 and~2 we allow implicit constants to depend on $L$ as well,
as this dependence does not play a role.  In subsequent sections
$L$-dependence is made explicit, though uniformity in $L$ is only
crucial in the contractive estimate of Theorem~\ref{thm:step}.  Our
main theorems hypothesise $L=L(d)$ is large, and for geometric
convenience we will assume throughout that $L$ is at least $2^{d+2}$.

\section{Consequences of combining renormalisation and Ward identities}
\label{sec:ward}

In our renormalisation group analysis, which provides the foundation
for the proofs of the theorems stated in Section~\ref{sec:intro}, we
will not assume any relation between the coupling constants of
the quadratic and quartic terms in \eqref{e:h02-psi} (except that
they are small).  The equality of the quadratic and quartic
couplings is restored with the help of the Ward identity \eqref{e:ward}, i.e.,
\begin{equation} \label{e:ward-bis}
  \avg{z_0}_{\beta,h}
  = h \sum_{x\in\Lambda} \avg{\xi_0 \eta_x}_{\beta,h},
  \quad  \text{ and in particular }
  \avg{z_0}_{\beta,0}
  = 0.
\end{equation}
This application of the Ward identity is the subject of this section.

In our analysis we distinguish between two orders of limits.
We first analyse the `infinite volume' limit $\lim_{h\downarrow 0} \lim_{N\to\infty}$, and prove Theorem~\ref{thm:h02}
(and thus Theorems~\ref{thm:forests0}--\ref{thm:forests}).
Using results of this analysis
(and with several applications of the Ward identity),
we then also analyse
the much more delicate `finite volume' limit $\lim_{N\to\infty}\lim_{h\downarrow 0}$
in order to prove Theorem~\ref{thm:forest-macro}.

\subsection{Infinite volume correlation functions}

For $m^2 >0$ arbitrary and coupling constants $s_0,a_0,b_0$, which eventually
will be taken small, we consider the model with fermionic Gaussian reference measure
with covariance
\begin{equation} \label{e:C-Lapm}
  C= (-\Delta+m^2)^{-1}
\end{equation}
on $\Lambda_N$ and interaction 
\begin{equation}
  \label{e:V0}
  V_0
  = V_0(\Lambda_N)=\sum_{x\in\Lambda_N}\qB{s_0 (\nabla \psi)_x(\nabla\bar\psi)_x + a_0 \psi_x\bar\psi_x + b_0 \psi_x\bar\psi_x(\nabla \psi)_x(\nabla \bar\psi)_x},
\end{equation}
where we recall the squared gradient notation from \eqref{e:quartic}.
Thus the corresponding expectation is
\begin{equation} 
  \label{e:psi4}
  \avg{F}_{m^2,s_0,a_0,b_0} = 
  \frac{1}{Z_{m^2,s_0,a_0,b_0}}\frac{1}{\det (-\Delta+m^2)} \int \partial_\psi \partial_{\bar\psi}\, e^{-(\psi,(-\Delta+m^2)\bar\psi) - V_0} F,
\end{equation}
where $\int \partial_\psi \partial_{\bar\psi}$ denotes the Grassmann integral,
and $Z_{m^2,s_0,a_0,b_0}$ is defined such that
$\avg{1}_{m^2,s_0,a_0,b_0}=1$. We emphasise the connection to the arboreal
gas arises only if $m^2,s_0,a_0,b_0$ are chosen to agree
with~\eqref{e:h02-psi}, c.f.\ \eqref{e:betab0m}--\eqref{e:hb0m} below.

The following result states that for correctly chosen $a_0$ the
correlation functions of the fields $\psi$ and $\bar\psi\psi$ are to
leading order multiples of those of the free (Grassmann Gaussian)
case $b_0=0$ if first $N\to\infty$ and then $m^2\downarrow 0$. The
result resembles those in \cite{MR3339164,MR3345374,MR3459163} for
weakly self-avoiding walks in dimension~4. Compared to the latter
results, our analysis is substantially simplified since the
$\HH^{0|2}$ model can be studied in terms of only fermionic variables
with a quartic interaction that is irrelevant in dimensions $d >
2$. However, in Section~\ref{sec:finvol}, we state an improvement of
the following result that captures the full zero mode of the
low temperature phase and goes beyond the analysis
of~\cite{MR3339164,MR3345374,MR3459163}.

\begin{theorem} 
  \label{thm:psi4} 
  Let $d \geq 3$ and $L \geq L_0(d)$.
  For $b_0$ sufficiently small and $m^2 \geq 0$, there are
  $s_0 = s_0^c(b_0,m^2)$ and $a_0= a_0^c(b_0,m^2)$ independent of $N$
  so that the following hold: The functions $s_0^c$ and  $a_0^c$ are continuous in both
  variables, differentiable in $b_0$ with uniformly bounded
  $b_0$-derivatives, and satisfy the estimates
\begin{equation} \label{e:thm-psi4-y0a0}
  s_0^c(b_0,m^2) = O(b_0), \qquad a_0^c(b_0,m^2) = O(b_0)
\end{equation}
uniformly in $m^2\geq 0$. There exists $\kappa>0$ such that if the
torus sidelength satisfies $L^{-N} \leq m$,
\begin{equation} 
  \label{e:psi4-suscept}
  \sum_{x\in\Lambda_N} \avg{\bar\psi_0\psi_x}_{m^2,s_0,a_0,b_0} = \frac{1}{m^2}
  + \frac{O(b_0L^{-(2+\kappa)N})}{m^4}.
\end{equation}
Moreover, there are functions
\begin{equation}
  \clambda = \clambda(b_0,m^2) = 1+O(b_0), \qquad \cgamma = \cgamma(b_0,m^2) = (-\Delta^{\Z^d}+m^2)^{-1}(0,0) + O(b_0),
\end{equation}
having the same continuity
properties as $s_0^c$ and $a_0^c$ such that
\begin{align}
    \label{e:psi4-00}
  \avg{\bar\psi_0\psi_0}_{m^2,s_0,a_0,b_0}
  &= \cgamma + O(b_0L^{-\kappa N}),
  \\
  \label{e:psi4-0x}
  \avg{\bar\psi_0\psi_x}_{m^2,s_0,a_0,b_0}
  &= (-\Delta+m^2)^{-1}(0,x) + O(b_0|x|^{-(d-2)-\kappa})+ O(b_0L^{-\kappa N}),
  \\
      \label{e:psi4-00xx}
  \avg{\bar\psi_0\psi_0;\bar\psi_x\psi_x}_{m^2,s_0,a_0,b_0}
  &= - \clambda^{2}(-\Delta+m^2)^{-1}(0,x)^2 + O(b_0|x|^{-2(d-2)-\kappa})+ O(b_0L^{-\kappa N}).
\end{align}
Here $\avg{A;B}=\avg{AB}-\avg{A}\avg{B}$.
\end{theorem}

The proof of this theorem is given  in Sections~\ref{sec:bulk}--\ref{sec:proofs}. 
We now show how to derive Theorem~\ref{thm:h02} for the $\HH^{0|2}$ model from it together with the Ward identity
\eqref{e:ward}. 
To this end, assuming $s_0>-1$ we further rescale $\psi$ by $1/\sqrt{1+s_0}$
(and likewise for $\bar\psi$) in \eqref{e:h02-psi}, and thus set
\begin{equation} \label{e:xietapsi}
  \xi = \sqrt{\frac{1+s_0}{\beta}} \bar\psi, \qquad \eta=\sqrt{\frac{1+s_0}{\beta}} \psi.
\end{equation}
Up to a normalisation constant, the fermionic density
\eqref{e:h02-psi} becomes, see also \eqref{e:quartic},
\begin{equation}
  \label{e:quartic-rescale}
  \exp\qa{-\sum_{x\in\Lambda_N}\pa{(1+s_0) (\nabla\psi)_x(\nabla\bar\psi)_x
    +\frac{1+s_0}{\beta}(1+h) \psi_x\bar\psi_x + \frac{(1+s_0)^2}{\beta} \psi_x\bar\psi_x (\nabla \psi)_x(\nabla \bar\psi)_x}
  }.
\end{equation}
For any $m^2 \geq 0$ and $s_0>-1$, \eqref{e:quartic-rescale} is of the form \eqref{e:psi4} with
\begin{equation}
  a_0 = \frac{1+s_0}{\beta}(1+h) -m^2, \qquad b_0 = \frac{(1+s_0)^2}{\beta}.
\end{equation}
To use Theorem~\ref{thm:psi4} to study the arboreal gas we need to invert this implicit relation between $(\beta,h)$
and $(m^2,s_0,a_0,b_0)$.
This is achieved by the following corollary.
A key observation is that the Ward identity \eqref{e:ward} allows us to identify the critical point with $h=0$.
To make this precise, with $s_0^c$ and $a_0^c$ as in
Theorem~\ref{thm:psi4}, define the functions
\begin{align}
  \label{e:betab0m}
  \beta(b_0,m^2) &= \frac{(1+s_0^c(b_0,m^2))^2}{b_0},\\
    \label{e:hb0m}
  h(b_0,m^2) &= -1+ \frac{a_0^c(b_0,m^2)+m^2}{b_0} (1+s_0^c(b_0,m^2)).
\end{align}
By Theorem~\ref{thm:psi4}, both functions are continuous in $b_0>0$ small enough and $m^2 \geq 0$.

\begin{corollary} \label{cor:ward}
  (i) Assume $b_0>0$ is small enough. Then
  \begin{equation}
    h(b_0,m^2) = m^2 \beta(b_0,m^2)(1+O(b_0)).
  \end{equation}
  In particular, $h(b_0,0) = 0$ and $h(b_0,m^2)>0$ if $m^2>0$.

  (ii) For $\beta$ large enough and $h\geq 0$, there are functions $\tilde b_0(\beta,h)>0$ and $\tilde m^2(\beta,h)\geq 0$ such that
  $h(\tilde b_0,\tilde m^2)=h$ and $\beta(\tilde b_0,\tilde m^2)=\beta$.
  Both functions are right-continuous as $h\downarrow 0$ when $\beta$ is fixed.
\end{corollary}

\begin{proof}
  To prove (i), we use the Ward identity \eqref{e:ward-bis} with $(\beta,h)$ given by \eqref{e:betab0m}--\eqref{e:hb0m}.
  The left- and right-hand sides of \eqref{e:ward-bis} are, respectively,
  \begin{align}
        \label{e:ward-pf-1}
  \avg{z_0}_{\beta,h} &= 1-\frac{1+s_0^c(b_0,m^2)}{\beta} \avg{\bar\psi_0\psi_0}_{m^2,s_0,a_0,b_0} ,
    \\
    \label{e:ward-pf-2}
  h \sum_{x\in\Lambda_N} \avg{\xi_0 \eta_x}_{\beta,h}
                      &=  \frac{(1+s_0^c(b_0,m^2))h(b_0,m^2)}{\beta(b_0,m^2)} \sum_{x\in\Lambda_N} \avg{\bar\psi_0\psi_x}_{m^2,s_0,a_0,b_0} .
\end{align}
By Theorem~\ref{thm:psi4}, in the limit $N\to\infty$, we obtain
from \eqref{e:ward-bis} that if $m^{2}>0$, the identity
\begin{equation}
  1-\frac{1+s_0^c(b_0,m^2)}{\beta(b_0,m^2)}\cgamma(b_0,m^2) = \frac{(1+s_0^c(b_0,m^2))h(b_0,m^2)}{\beta(b_0,m^2) m^2}
\end{equation}
holds.  Solving for $h$, we have
\begin{equation}
  h(b_0,m^2)= m^2 \qa{ \frac{\beta(b_0,m^2)}{1+s_0^c(b_0,m^2)} - \cgamma(b_0,m^2) }.
\end{equation}
Since $s_0^c(b_0,m^2)=O(b_0)$, $\beta(b_{0},m^{2}) \asymp 1/b_0$, and $\cgamma(b_0,m^2)=O(1)$,
all uniformly in $m^2\geq 0$, we obtain
$h(b_0,m^2) = m^2 \beta(b_0,m^2)(1+O(b_0))$. In particular, $h(b_0,0)=0$.

Claim (ii)
follows from an implicit function theorem argument that uses that $s_0^c$ and $a_0^c$ are continuous in $m^2\geq 0$ and differentiable in $b_0$ if $m^2>0$ with  $b_0$-derivatives uniformly bounded in $m^2>0$.
This argument is the same as the proof of \cite[Proposition~4.2]{MR3339164}
(with our notation $s_0$ instead of $z_0$, $a_0$ instead of $\nu_0$, $b_0$ instead of $g_0$,
and with $1/\beta$ instead of $g$ and $h$ instead of $\nu$)  and is omitted here.
\end{proof}

Assuming Theorem~\ref{thm:psi4}, the proof of Theorem~\ref{thm:h02} is
immediate from the last corollary. The main statements of
Theorems~\ref{thm:forests0} and~\ref{thm:forests} then follow
immediately, except for the identifications
$\theta_{d}(\beta)^{2}=\zeta_{d}(\beta)$, 
$(\constb(\beta)/\consta(\beta)^{2})\theta_{d}(\beta)^{2}=1$, and
$\consto(\beta)=2\constb(\beta)$ which
we will obtain in Section~\ref{sec:finvol}.

\begin{proof}[Proof of Theorem~\ref{thm:h02}]
  Given $\beta \geq \beta_0$ and $h>0$ we choose $b_0>0$ and $m^2>0$ as in Corollary~\ref{cor:ward}~(ii).
  Since $z_x = 1-\xi_x\eta_x$ and using \eqref{e:xietapsi}
  we then have
  \begin{align}
    \avg{z_0}_{\beta,h}
    &= 1-\avg{\xi_0\eta_0}_{\beta,h}
      =
      1- \frac{1+s_0}{\beta} \avg{\bar\psi_0\psi_0}_{m^2,s_0,a_0,b_0},
    \\
    \avg{\xi_0\eta_x}_{\beta,h}
    &= \frac{1+s_0}{\beta} \avg{\bar\psi_0\psi_x}_{m^2,s_0,a_0,b_0},
    \\
    \avg{z_0z_x}_{\beta,h}-\avg{z_0}_{\beta,h}^2
    &= \avg{\xi_0\eta_0\xi_x\eta_x}_{\beta,h} - \avg{\xi_0\eta_0}_{\beta,h}^2
    = \frac{(1+s_0)^2}{\beta^2} \avg{\bar\psi_0\psi_0;\bar\psi_x\psi_x}_{m^2,s_0,a_0,b_0} 
      .
  \end{align}
  Taking $N\to\infty$ and then $h\downarrow 0$, the results follow
  from Corollary~\ref{cor:ward} (i) and Theorem~\ref{thm:psi4} with
  \begin{equation} \label{e:constants}
    \theta_d(\beta) = 1- \frac{b_0\cgamma}{1+s_0^c},
    \qquad
    \consta(\beta) = (1+s_0^c)c_d,
    \qquad
    \constb(\beta) = \clambda^{2} (1+s_0^c)^2 c_d^2,
  \end{equation}
  where the functions $\lambda$ and $\cgamma$ are evaluated at $m^2=0$
  and $b_0$ given as above, $c_d$ is the constant in the
  asymptotics of the free Green's function on $\Z^{d}$,
  see~\eqref{eq:LapZd}, and we have used the simplification of the error
  terms
  $O(|x|^{-(d-2)-1})+O(b_0 |x|^{-(d-2+\kappa)}) = O(|x|^{-(d-2+\kappa)})$
  and
  $O(|x|^{-2(d-2)-1})+O(b_0 |x|^{-2(d-2)-\kappa)}) = O(|x|^{-2(d-2)-\kappa)})$.
\end{proof}

\subsection{Finite volume limit}
\label{sec:finvol}

The next theorem extends Theorem~\ref{thm:psi4} by more precise
estimates valid in the limit $m^2\downarrow 0$ with $\Lambda_N$ fixed.
In these estimates $t_N \in (0,1/m^{2})$ is a continuous function of $m^2>0$ that
satisfies
\begin{gather}
  \label{e:tn}
  t_{N}-\frac{1}{m^{2}} = O(L^{2N})
  \qquad \text{and}
  \\
  \label{e:greenfLambdaZd}
  \lim_{m\downarrow 0}\qa{ (-\Delta+m^2)^{-1}(0,x) - \frac{t_N}{|\Lambda_N|}}
  = (-\Delta^{\Z^d})^{-1}(0,x) + O(L^{-(d-2)N}),
\end{gather}
where on the right-hand side $\Delta^{\Z^d}$ is the Laplacian on
$\Z^d$, on the left-hand side $\Delta$ is the Laplacian on
$\Lambda_N$, and $|\Lambda_N| = L^{dN}$ denotes the volume of the
torus $\Lambda_N$.  We define
\begin{equation}
  \label{e:WN-def}
  W_N(x) = W_{N,m^2}(x) = (-\Delta+m^2)^{-1}(0,x) - \frac{t_N}{|\Lambda_N|},
\end{equation}
so that $W_N(x)$ is essentially the torus Green's function $(-\Delta+m^2)^{-1}$ with the zero mode omitted.

In the following theorem, and throughout this section, $\Lambda_{N}$ is fixed and the parameters
$(\beta,h)$ are related to $(m^2,s_0,a_0,b_0)$ as in
Corollary~\ref{cor:ward}.  We will write
$\avg{\cdot}_{m^2,b_0} =
\avg{\cdot}_{m^2,s_0^c(b_0,m^2),a_0^c(b_0,m^2),b_0}$ for the
corresponding expectation and similarly for the partition function
$Z_{m^2,b_0}$.

\begin{theorem} \label{thm:psi4-full}
  Under the conditions of Theorem~\ref{thm:psi4} except that we no longer restrict $L^{-N}\leq m$, in addition to the functions
  $a^{c}_{0}$, $s^{c}_{0}$, $\lambda$, and $\gamma$, there are functions 
  $\tilde a_{N,N}^c = \tilde a_{N,N}^c(b_0,m^2)$ and $u_N^c = u_N^c(b_0,m^2)$,
  both continuous in $b_0$ small and $m^2\geq 0$, as well as
  \begin{equation}
    \label{e:tildeu}
    \tilde u_{N,N}^c = \tilde u_{N,N}^c(b_0,m^2)= t_N\tilde a_{N,N}^c(b_0,m^2)  + O(b_0L^{-\kappa N}),
  \end{equation}
  continuous in $b_0$ small and $m^2> 0$, such that, for $x\in\Lambda_N$,
  \begin{align} \label{e:suscept-atilde-thm}
    \sum_{x\in\Lambda_N} \avg{\bar\psi_0\psi_x}_{m^2, b_0}
    &= \frac{1}{m^2}
    - \frac{1}{m^4} 
    \frac{\tilde a_{N,N}^c}{1+\tilde u_{N,N}^c}, \\
    \label{e:1pt-b-thm}
    \avg{\bar\psi_0\psi_0}_{m^2,b_0}
    &=
    \cgamma + \frac{\clambda t_N |\Lambda_N|^{-1}}{1+\tilde u_{N,N}^c}  +E_{00}
      ,
      \\
  \label{e:2pt-b-thm}
    \avg{\bar\psi_0\psi_0\bar\psi_x\psi_x}_{m^2,b_0}
    &=
      -\clambda^2 
      W_N(x)^2
      +\gamma^2
      +
      \frac{
        -2\lambda^2 
        W_N(x)
      + 2\clambda\gamma 
      }{1+\tilde u_{N,N}^c} t_N|\Lambda_N|^{-1}
    +E_{00xx}
      ,
\end{align}
and 
\begin{equation}
    \label{e:ZuNNaNN}
    Z_{m^{2},b_{0}} = e^{-u_{N}^{c}|\Lambda_N|}(1+\tilde u_{N,N}^{c}). 
\end{equation}
The remainder terms satisfy
\begin{align}
  E_{00}
  &=
    \frac{O(b_0L^{-(d-2+\kappa)N}+b_0L^{-\kappa N}(m^2|\Lambda_N|)^{-1})}{1+\tilde u_{N,N}^c},
  \\
  E_{00xx}
   &= O(b_0|x|^{-2(d-2)-\kappa}) + O(b_0L^{-(d-2+\kappa)N}) 
    \nnb
  &\qquad
    +
    (O(b_0|x|^{-(d-2+\kappa)})+ O(b_0L^{-\kappa N})) \frac{(m^2|\Lambda_N|)^{-1}}{1+\tilde u_{N,N}^c}.
\end{align}
\end{theorem}

The proof of this theorem is again given in Sections~\ref{sec:bulk}--\ref{sec:proofs}.
This proof also gives a bound on $\tilde a_{N,N}^c$ of order $b_0L^{-(2+\kappa)N}$ for a small $\kappa>0$.
We did not state this bound above because (by using the Ward identity  \eqref{e:ward-bis}) 
the existence of
$\tilde a_{N,N}^c$ with its relation to the correlation functions as stated in the theorem
is, in fact, sufficient to determine its precise asymptotic value of order $b_0/|\Lambda_N| = b_0L^{-dN} \ll b_0L^{-(2+\kappa)N}$,
see Lemmas~\ref{cor:macro}--\ref{cor:anm0} below.
Using this precise asymptotic information on $\tilde a_{N,N}^c$,
Theorem~\ref{thm:forest-macro} then follows from Theorem~\ref{thm:psi4-full}.
The key computation occurs in Lemma~\ref{cor:conn4pt}, where the asymptotic value of $\tilde a_{N,N}^c$
is used to exhibit important cancelations between the
terms on the right-hand side of \eqref{e:2pt-b-thm}.

\begin{lemma}
  \label{cor:macro}
  Under the conditions of Theorem~\ref{thm:psi4-full},
  \begin{equation}
  \label{E:t_0}
    \E^{\Lambda_{N}}_{\beta,0}|T_0|
    =
    \frac{b_0}{1+s_0^c(b_0,0)} \frac{1+O(b_0L^{-\kappa N})}{\tilde a_{N,N}^{c}(b_{0},0)} + O(b_0L^{2N})
    .
  \end{equation}
  In particular, if $b_0>0$ this implies $1/\tilde a_{N,N}^{c}(b_{0},0) = O(|\Lambda_{N}|/b_{0})$
  and $\tilde a_{N,N}^{c}(b_{0},0)>0$.
\end{lemma}

\begin{proof}
From \eqref{e:P0x4pt}, we have that
\begin{equation}
 \E^{\Lambda_{N}}_{\beta,0}|T_0|=\sum_{x\in \Lambda_N}  \P_{\beta,0}[0\conn x] = \sum_{x\in \Lambda_N} \avg{\xi_0\eta_x}_{\beta,0}=\lim_{h\rightarrow 0}  \sum_{x\in \Lambda_N} \avg{\xi_0\eta_x}_{\beta,h}.
\end{equation}
Changing variables,
\begin{equation}
  \sum_{x\in\Lambda_N} \avg{\xi_0\eta_x}_{\beta,h}
  = \frac{b_0}{1+s_0^c(b_0,m^2)} \sum_{x\in\Lambda_N} \avg{\bar\psi_0\psi_x}_{m^2,b_0},
\end{equation}
where $(\beta, h)$ and $(b_0, m^2)$ are related as in \eqref{e:betab0m} and \eqref{e:hb0m}.
To evaluate the right-hand side we use \eqref{e:suscept-atilde-thm}. Note that
\begin{align}
  \frac{1}{m^2}
  - \frac{1}{m^4} 
  \frac{\tilde a_{N,N}^c}{1+\tilde u_{N,N}^c}
  &= \frac{1}{m^2}\frac{1+\tilde u_{N,N}^c - \tilde a_{N,N}^cm^{-2}}{1+\tilde u_{N,N}^c}
    \nnb
  &= \frac{1+\tilde a_{N,N}^c (t_N-m^{-2})+O(b_0L^{-\kappa N})}{m^2+\tilde a_{N,N}^ct_Nm^2 + O(b_0m^2L^{-\kappa N})}
    \nnb
  &= \frac{1+\tilde a_{N,N}^c O(L^{2N})+O(b_0L^{-\kappa N})}{m^2+\tilde a_{N,N}^c(1+O(m^2L^{2N})) + O(b_0m^2L^{-\kappa N})},
\end{align}
where the second equality is due to \eqref{e:tildeu} and the third follows from \eqref{e:tn}. As $m^2\downarrow 0$,
the right-hand side of the third equality behaves asymptotically as
\begin{align}
  \frac{1+\tilde a_{N,N}^c O(L^{2N})+O(b_0L^{-\kappa N})}{\tilde a_{N,N}^c}
  =
  \frac{1+O(b_0L^{-\kappa N})}{\tilde a_{N,N}^c}
  +
  O(L^{2N})
  .
\end{align}
Since $s_{0}^{c}(b_{0},0)=O(b_{0})$ by Theorem~\ref{thm:psi4} we therefore obtain the first claim:
\begin{equation}
  \E_{\beta,0}^{\Lambda_N}|T_0|=
  \frac{b_0}{1+s_0^c(b_0,0)}\lim_{m^2\downarrow 0} \sum_{x\in\Lambda_N} \avg{\bar\psi_0\psi_x}_{m^2,b_0}
  =
  \frac{b_0}{1+s_0^c(b_0,0)} \frac{1+O(b_0L^{-\kappa N})}{\tilde a_{N,N}^{c}(b_{0},0)} + O(b_0L^{2N})
  .
\end{equation}

For the second claim,
let us observe that, on the one hand,
\begin{equation} \label{e:ZbetahuNN}
  Z_{\beta,h}
  = \pa{\frac{\beta}{1+s_0^c}}^{|\Lambda_N|} 
    (\det (-\Delta+m^2)) Z_{m^2, b_0}
  = \pa{\frac{\beta e^{-u_{N}^{c}}}{1+s_0^c}}^{|\Lambda_N|} 
    (\det (-\Delta+m^2))(1+\tilde u_{N,N}^{c}),
\end{equation}
where the first equality is by Proposition~\ref{prop:h02-forest} and \eqref{e:psi4}, \eqref{e:xietapsi}, and \eqref{e:quartic-rescale},
and the second equality is \eqref{e:ZuNNaNN}. 
On the other hand, by \eqref{e:P-forest},
\begin{equation}
\lim_{h\rightarrow 0} Z_{\beta,h}= Z_{\beta,0}>0.
\end{equation}
Since, by Theorem \ref{thm:psi4-full}, $u_{N}^{c}$ and $s_0^c$ remain bounded as
$m^2\downarrow 0$ with $\Lambda_N$ fixed (and thus also $\beta$ which is given by \eqref{e:betab0m}),
from $\det (-\Delta+m^2) \downarrow 0$,
we conclude that $1+\tilde u_{N,N}^c$ diverges as $m^2\downarrow 0$.
By \eqref{e:tildeu},
this implies $\tilde a_{N,N}^{c}(b_{0},0)>0$. The upper bound on
$1/\tilde a_{N,N}^c(b_0,0)$ follows by re-arranging \eqref{E:t_0} and
using the trivial bound $|T_0| \leq |\Lambda_N|$.
\end{proof}

Using that $\tilde a_{N,N}^c$ is at least of order $b_0/|\Lambda_N|$ as established in the previous lemma,
the following lemma gives an asymptotic representation of $\tilde a_{N,N}^c$ of order $b_0/|\Lambda_N|$
in terms of $\gamma$ from Theorem~\ref{thm:psi4-full}.

\begin{lemma}
  \label{cor:anm0}
  Under the conditions of
  Theorem~\ref{thm:psi4-full} and if $b_{0}>0$,
  \begin{equation}
    \label{e:anm0}
  1
  =
  \frac{b_0}{1+s_0^c(b_0,0)}
  \qa{
    \cgamma(b_0,0) + \frac{\lambda(b_0,0)}{|\Lambda_N|\tilde
      a_{N,N}^c(b_0,0)}
    (1+O(b_0L^{-\kappa N}))
  }
  .
\end{equation}
\end{lemma}

\begin{proof}
  The Ward identity $\avg{z_0}_{\beta,0}=0$ implies
  \begin{align}
    0=\avg{z_0}_{\beta,0}
   = 1-\avg{\xi_0\eta_0}_{\beta,0}
    &=1-\lim_{m^2\downarrow 0} \frac{1+s_0^c(b_0,m^2)}{\beta(b_0,m^2)}\avg{\bar\psi_0\psi_0}_{m^2,b_0}
    \nnb
    & =1-\lim_{m^2\downarrow 0} \frac{b_0}{1+s_0^c(b_0,m^2)}\avg{\bar\psi_0\psi_0}_{m^2,b_0},
  \end{align}
  where we used \eqref{e:xietapsi} and
  that
  $\beta=\beta(b_0,m^2)$ is as in \eqref{e:betab0m}.
  To compute $\avg{\bar\psi_0\psi_0}_{m^2,b_0}$, we apply     \eqref{e:1pt-b-thm}.
  Since $\tilde u_{N,N}^c = \tilde a_{N,N}^ct_N +
  O(b_0L^{-\kappa N})$ and $t_N = m^{-2}+O(L^{2N})$, 
  \begin{align}
       \lim_{m^2\downarrow 0}\avg{\bar\psi_0\psi_0}_{m^2,b_0}
    &= \gamma(b_{0},0) + \lim_{m^{2}\downarrow 0} \frac{\lambda(b_{0},m^{2})
      t_{N}|\Lambda_{N}|^{-1}}{1+\tilde
      a_{N,N}^{c}(b_{0},m^{2})t_{N}+O(b_{0}L^{-\kappa N})} +
      \lim_{m^{2}\downarrow 0}E_{00}
      \nnb
    &= \gamma(b_{0},0) +
      \frac{\lambda(b_{0},0)|\Lambda_{N}|^{-1}}{\tilde
      a_{N,N}^c(b_{0},0)} + \lim_{m^{2}\downarrow 0}E_{00}.
  \end{align}
The limits in the second line exist by
Theorem~\ref{thm:psi4-full} and Lemma~\ref{cor:macro},
which in particular implies $\tilde a_{N,N}^c(b_0,0) > 0$ since $b_0>0$.
As $m^2\downarrow 0$, the error term $E_{00}$ is bounded by
$O(b_0L^{-\kappa N}/(|\Lambda_N|\tilde a_{N,N}^c))
= (\lambda(b_{0},0) |\Lambda_N|^{-1}/\tilde a_{N,N}^c)O(b_0L^{-\kappa N})$
since $\lambda(b_{0},0) =1-O(b_0)\geq 1/2$, finishing the proof.
\end{proof}

Given Theorem~\ref{thm:psi4-full}, the following lemma is the main step in the proof of Theorem~\ref{thm:forest-macro}.
It uses the asymptotic representation of $\tilde a_{N,N}^c$ to exhibit cancelations in
expressions in Theorem~\ref{thm:psi4-full}.

\begin{lemma} \label{cor:conn4pt}
  Under the conditions of
  Theorem~\ref{thm:psi4-full} and if $b_{0}>0$,
  \begin{multline}
  \label{E:tredens}
  \P_{\beta,0}^{\Lambda_N}[0\conn x]
    = \theta_d(\beta)^2
    + 
    2 \frac{b_0}{1+s_0^c}
    \clambda \theta_d(\beta)
    (-\Delta^{\Z^d})^{-1}(0,x)
    \\
    + O(b_0^{2}|x|^{-(d-2)-\kappa}) 
    + O(b_0L^{-(d-2)N})
    + O(b_0^2L^{-\kappa N}),
  \end{multline}
  where $\theta_d(\beta)$ is defined in \eqref{e:constants}.
\end{lemma}

\begin{proof}
By the last expression for
$\P_{\beta,0}[0\conn x]$ in \eqref{e:P0x4pt} and \eqref{e:xietapsi}, \eqref{e:betab0m}:
\begin{equation} \label{e:P0x4ptpf}
  \P_{\beta,0}^{\Lambda_N}[0\conn x]
  = 1-\lim_{h\downarrow 0}\avg{\xi_0\eta_0\xi_x\eta_x}_{\beta,h}
  = 1- \lim_{m^2\downarrow 0}\qa{ \frac{b_0^2}{(1+s_0^c)^2}
    \avg{\bar\psi_0\psi_0\bar\psi_x\psi_x}_{m^{2},b_{0}}}.
\end{equation}

To compute
$\lim_{m^2\downarrow
  0}\avg{\bar\psi_0\psi_0\bar\psi_x\psi_x}_{m^{2},b_{0}}$ we start
from \eqref{e:2pt-b-thm}.  By Lemma~\ref{cor:macro},
as $m^2\downarrow 0$ with $\Lambda_N$ fixed,
\begin{equation}
  \frac{1}{1+\tilde u_{N,N}^c} \sim
  \frac{1}{m^{-2}\tilde a_{N,N}^c(b_0,0)}
  =  
  O(\frac{m^2|\Lambda_N|}{b_0}).
\end{equation}
This implies the error term in \eqref{e:2pt-b-thm} is, 
as $m^2\downarrow 0$ with $\Lambda_N$ fixed,
\begin{equation}
  |E_{00xx}|
   \leq O(|x|^{-(d-2)-\kappa}) + O(L^{-\kappa N})
  .
\end{equation}

For the main term we have (recall $W_{N}(x)=W_{N,m^{2}}(x)$, see~\eqref{e:WN-def})
\begin{align}
  \lim_{m^{2}\downarrow 0}\avg{\bar\psi_0\psi_0\bar\psi_x\psi_x}_{m^2,b_0}
  - \lim_{m^2\downarrow 0} |E_{00xx}|
  &=
    -\clambda^2 W_{N,0}(x)^2 +\gamma^2
    +
    \lim_{m^{2}\downarrow 0}\frac{
    -2\clambda^2 W_{N}(x)
    + 2\clambda\gamma 
    }{1+\tilde u_{N,N}^c} {t_N|\Lambda_N|^{-1}}
    \nnb
  &=  -\clambda^2 W_{N,0}(x)^2 +\gamma^2
    + 2(-\clambda W_{N,0}(x)+\gamma)\frac{
    \clambda}{\tilde a_{N,N}^c|\Lambda_{N}|}
    ,
    \label{e:connmain}
\end{align}
where on the right-hand side the functions $\lambda$, $\gamma$, and $\tilde a_{N,N}^c$ are evaluated at $m^2=0$.
By Lemma~\ref{cor:anm0},
\begin{equation}
  \label{eq:1}
  \frac{b_{0}}{1+s_{0}^{c}}
  \frac{\clambda}{\tilde a_{N,N}^{c}|\Lambda_{N}| } =
  \pa{1-\frac{b_{0}\cgamma}{1+s_{0}^{c}}}
  (1+O(b_0L^{-\kappa N}))
\end{equation}
so that
\begin{align}
  \label{eq:5}
  -\pa{\frac{b_{0}}{1+s_{0}^{c}}}^{2}\frac{
  2\lambda^2 W_{N,0}(x)}{\tilde a_{N,N}^c|\Lambda_{N}|}
  { (1+O(b_0L^{-\kappa N}))}
  &= -\frac{2b_{0}}{1+s_{0}^{c}} 
    (1-\frac{b_{0}\cgamma}{1+s_{0}^{c}})\clambda W_{N,0}(x) 
  \\
  \pa{\frac{b_{0}}{1+s_{0}^{c}}}^{2}\frac{2\clambda\gamma 
  }{\tilde a_{N,N}^c|\Lambda_{N}|}
  { (1+O(b_0L^{-\kappa N}))}
  &= 2\cgamma
     \frac{b_{0}}{1+s_{0}} - 2\cgamma^{2}\pa{\frac{b_{0}}{1+s_{0}^{c}}}^{2}.
\end{align}
Substituting these bounds into \eqref{e:connmain} and then  \eqref{e:P0x4ptpf} we obtain
\begin{align}
  \label{eq:2}
  \P_{\beta,0}^{\Lambda_N}[0\conn x]
  &=
  1-\pa{\frac{b_{0}}{1+s_{0}^{c}}}^{2}\lim_{m^{2}\downarrow
    0}\avg{\bar\psi_0\psi_0\bar\psi_x\psi_x}_{m^2,b_0}
    \nnb
  &= (1-\frac{\cgamma b_{0}}{1+s_{0}})^{2}
    +\frac{2b_{0}\clambda}{1+s_{0}^{c}} (1-\frac{b_{0}\cgamma}{1+s_{0}^{c}}) W_{N,0}(x)
    + \pa{\frac{b_{0}\clambda
    W_{N,0}(x)}{1+s_{0}^{c}}}^{2}
    \nnb
  &\qquad  +O(b_0^2 L^{-\kappa N} W_{N,0}(x))  + O(b_0^2 L^{-\kappa N})
    + O(b_0^2|E_{00xx}|).
\end{align}
Using the definition \eqref{e:constants} of $\theta_d(\beta)$,
that $W_{N,0}(x) = (-\Delta^{\Z^d})^{-1}(0,x)+O(L^{-(d-2)N})$ by   \eqref{e:greenfLambdaZd},
and in particular $W_{N,0}(x) = O(|x|^{-(d-2)})$,
the claim follows.
\end{proof}

The next (and final) lemma 
is inessential for the main conclusions, but will allow us to identify
the constants from the infinite volume and the finite volume analyses.

\begin{lemma} \label{cor:aNlimit}
  Under the conditions of Theorem~\ref{thm:psi4-full} and if $b_{0}>0$, then $\lambda \theta_d(\beta) = 1$.
\end{lemma}

\begin{proof}
  Let
  \begin{equation} \label{e:wNdef}
    w_{N} \bydef \frac{b_{0}}{1+s^{c}_{0}} \frac{1}{\tilde a_{N,N}^{c}|\Lambda_{N}|}
    = \E^{\Lambda_N}_{\beta,0} \frac{|T_0|}{|\Lambda_N|}  + O(b_0L^{-\kappa N}+b_0L^{-(d-2)N}),
  \end{equation}
  where the second equality is due to Lemma~\ref{cor:macro}.
  The density $\E^{\Lambda_N}_{\beta,0} |T_0|/|\Lambda_N|$ can also be
  computed by summing the estimate in Lemma~\ref{cor:conn4pt} and dividing by $|\Lambda_N|$.
  Subtracting this result from \eqref{e:wNdef} gives
  \begin{equation}
    w_N-\theta_d(\beta)^2 
    =O(b_0L^{-\kappa N}).
          \label{e:r1}
  \end{equation}
  On the other hand, \eqref{e:anm0} shows that
  \begin{equation}
   \label{e:sys0}
   \clambda w_N  
   -\theta_d(\beta)
   = O(b_0L^{-\kappa N}).
  \end{equation}
  The limit $w= \lim_{N\to\infty} w_N$ thus exists and satisfies
  $\clambda w = \theta_d(\beta)$ and $w=\theta_d(\beta)^2$.
  Since $\theta_d(\beta)=1- O(1/\beta)\neq 0$
  this implies $\lambda \theta_d(\beta) = 1$.
\end{proof}

\begin{proof}[Proof of Theorem~\ref{thm:forest-macro}]
  The proof follows by rewriting Lemma~\ref{cor:conn4pt}. 
  Let  $c_d$ be the constant in the Green function asymptotics
  of~\eqref{eq:LapZd},
  and recall the constants $\theta_d(\beta)$ and ${\sf c}_i(\beta)$ from \eqref{e:constants}.
  Theorem~\ref{thm:forest-macro} then follows from Lemma~\ref{cor:conn4pt} by setting
  \begin{equation}
    \zeta_d(\beta)
    = \theta_d(\beta)^2, \qquad
      \consto(\beta)
      =(1+s_{0}^{c})2\clambda \theta_{d}(\beta) c_d,
  \end{equation}
  and simplifying the error terms using
  $O(b_0|x|^{-(d-2)-1})+O(b_0^2 |x|^{-(d-2+\kappa)}) = O(\beta^{-1}|x|^{-(d-2+\kappa)})$ and
  $O(b_0L^{-(d-2)N})+O(b_0^2 L^{-\kappa N}) = O(\beta^{-1}L^{-\kappa N})$.
\end{proof}

\begin{proof}[Completion of proof of Theorems~\ref{thm:forests0} and~\ref{thm:forests}]
    For Theorem~\ref{thm:forests0}, $\zeta_d(\beta) =
    \theta_d(\beta)^2$ was established in the 
    previous proof. For Theorem~\ref{thm:forests}, the identity
    $(\constb(\beta)/\consta(\beta)^{2})\theta_{d}(\beta)^{2}=1$ is
    equivalent (by \eqref{e:constants}) to $\theta_{d}(\beta)\lambda=1$, i.e.,
    Lemma~\ref{cor:aNlimit}. Similarly,
    $\consto(\beta)=2\clambda\theta_{d}(\beta) \consta(\beta) = 2\consta(\beta)$.
\end{proof}
  
\begin{remark}
  To compute $\P_{\beta,0}[0\conn x]$ we started from the expression
  $1-\avg{\xi_0\eta_0\xi_x\eta_x}_{\beta,0}$ in \eqref{e:P0x4pt}.  An
  alternative route would have been to start from
  $\avg{\xi_0\eta_x}_{\beta,0}$.  For technical reasons arising in
  Section~\ref{sec:obs-flow} it is, however, easier to obtain sufficient
  precision when working with $\avg{\xi_0\eta_0\xi_x\eta_x}_{\beta,0}$.
\end{remark}

\section{The bulk renormalisation group flow}
\label{sec:bulk}

We will prove Theorems~\ref{thm:psi4} and~\ref{thm:psi4-full} by a renormalisation group
analysis that is set up following \cite{MR2523458,MR3332942} and
\cite{MR3345374,MR3339164}; see also \cite{MR3969983} for a conceptual
introduction.  Our proof is largely self-contained.
The exceptions to self-containment
concern general properties about finite range decomposition, norms, and approximation
by local polynomials that were developed systematically in
\cite{MR3129804,MR3332938,MR3332939}. The properties we need
are all reviewed in this section. The first six subsections set up
the framework of the analysis, and the remaining three define and
analyse the renormalisation group flow.

Throughout $\Lambda\bydef\Lambda_N$ is the discrete torus of side length $L^N$.
We leave $L$ implicit; it will eventually be chosen large. We
sometimes omit the $N$ when it does not play a role.

\subsection{Finite range decomposition}
\label{sec:frd}

Let $\Delta$ denote the lattice Laplacian on $\Lambda_N$, and let $m^2>0$.
Our starting point for the analysis is the decomposition
\begin{equation} \label{e:C-def}
  C = (-\Delta+m^2)^{-1} = C_1 + \cdots +C_{N-1}+C_{N,N} 
\end{equation}
where the $C_j$ (with $j<N$) and $C_{N,N}$ are positive semidefinite $m^2$-dependent matrices indexed by $\Lambda_N$.
These covariances can be chosen with the following properties, see
\cite[Proposition~3.3.1 and Section~3.4]{MR3969983} and Appendix~\ref{app:decomp}.
The notation $C_{N,N}$ for the last covariance is explained below.

\subsubsection*{Finite range property}

For $j<N$, the covariances $C_j$ satisfy the finite range property
\begin{equation} \label{e:C-range}
  C_j(x,y) = 0 \qquad \text{ if } |x-y|_\infty \geq \frac12 L^{j}.
\end{equation}
Moreover, they are invariant under lattice symmetries and independent of $\Lambda_N$
in the sense that $C_j(x,y)$ can be identified as a function of $x-y$ that is independent of the torus $\Lambda_N$.
They are defined and continuous for $m^2\geq 0$ including the endpoint $m^2=0$ (and in fact smooth).

\subsubsection*{Scaling estimates}
The covariances satisfy estimates consistent with the decay of the Green function:
\begin{equation} \label{e:C-bd}
  |\nabla^\alpha C_{j+1}(x,y)| \leq O_{\alpha,s}(\vartheta_j(m^2)) L^{-(d-2+|\alpha|)j},
\end{equation}
where for an arbitrary fixed constant $s$,
\begin{equation}
  \vartheta_j(m^2)=\frac{1}{2d+m^2}\pa{1+\frac{m^2L^{2j}}{2d+m^2}}^{-s}.
\end{equation}
The discrete gradient in \eqref{e:C-bd} can act on either the $x$ or
the $y$ variable, and is defined as follows.  Recalling that
$e_1, \dots, e_d$ denote the standard unit vectors generating $\Z^d$,
that $e_{d+j}=-e_j$, and that $\cE_d =\{e_1,\dots,e_{2d}\}$,
for any multiindex $\alpha\in \N_0^{\cE_{d}}$, we define the
discrete derivative in directions $\alpha$ with order
$|\alpha|=\sum_{i=1}^{2d} \alpha(e_i)$ by:
\begin{equation}
  \nabla^\alpha=\prod_{i=1}^{2d} \nabla^{\alpha(e_i)}_{e_i}, \qquad \nabla_e f= f(x+e)-f(x),
\end{equation}
with $\nabla_{e_i}^{k} = \nabla_{e_i}\cdots \nabla_{e_i}$, where there are $k$ terms on the right-hand side.

\subsubsection*{Zero mode}
By the above independence of the covariances $C_j$ with $j<N$ from $\Lambda_N$,
all finite volume torus effects are concentrated in the last covariance
$C_{N,N}$. We further separate this covariance into a bounded part and the zero mode:
\begin{equation} \label{e:CNN-def}
  C_{N,N} = C_{N} + t_NQ_N,
\end{equation}
where $t_N$ is an $m^2$-dependent constant
and $Q_N$ is the projection onto the zero mode,
i.e., the matrix with all entries equal to $1/|\Lambda_N|$.
The bounded contribution $C_{N}$
(which does depend on $\Lambda_N$)
satisfies the estimates \eqref{e:C-bd} with $j=N$
and also extends continuously to $m^2=0$.
The constant $t_N$ satisfies 
\begin{equation}
  \label{e:tN}
  t_{N}>0, \qquad t_{N}-\frac{1}{m^{2}} = O(L^{2N}).
\end{equation}
In this section, we only consider the effect of $C_N$
(which is parallel to that of the $C_j$ with $j<N$)
while the nontrivial finite volume effect of $t_N$ will be analysed in Sections~\ref{sec:suscept}--\ref{sec:obs}.

\bigskip
The above properties imply \eqref{e:greenfLambdaZd} and $W_N(x)$
in \eqref{e:WN-def} is given by $W_N(x) = C_1(x) + \cdots + C_N(x)$.

\subsection{Grassmann Gaussian integration}
\label{sec:grassm-gauss-integr}

For $X \subset \Lambda=\Lambda_N$, we denote by $\cN(X)$ the Grassmann algebra generated by $\psi_x,\bar\psi_x$, $x\in X$
with the natural inclusions $\cN(X) \subset \cN(X')$ for $X \subset X'$.
Moreover, we denote by $\cN(X \sqcup X)$ the doubled algebra with generators $\psi_x,\bar\psi_x,\zeta_x,\bar\zeta_x$
and by $\theta\colon \cN(X) \to \cN(X \sqcup X)$ the doubling homomorphism acting on the generators of $\cN(X)$ by
\begin{equation}\label{e:theta-def}
  \theta \psi_x = \psi_x+\zeta_x, \qquad \theta \bar\psi_x = \bar\psi_x + \bar\zeta_x.
\end{equation}
For a covariance matrix $C$ the associated Gaussian expectation
$\E_{C}$ acts on $\cN(X\sqcup X)$ on the $\zeta,\bar\zeta$
variables. Explicitly, when $C$ is positive definite, $F\in
\cN(X\sqcup X)$ maps to $\E_{C}F \in \cN(X)$ given by 
\begin{equation}
  \E_CF
  = \E_C[F]
  = (\det C)\int \partial_{\zeta}\partial_{\bar\zeta}\, e^{-(\zeta,C^{-1}\bar\zeta)} F.
\end{equation}
Thus $\E_{C}\theta \colon \cN(\Lambda) \to \cN(\Lambda)$ is the
fermionic convolution of $F \in \cN(\Lambda)$ with the fermionic
Gaussian measure with covariance $C$. Recall the following
well-known facts about $\E_{C}\theta$; elementary proofs can be
found in, e.g., \cite{MR3332938}. First, this convolution operator
can be written as
\begin{equation}
  \label{e:fWick}
  \E_{C}\theta F
  = \E_C[\theta F]
  = e^{\cL_{C}}F
\end{equation}
where  $\cL_C = \sum_{x,y\in\Lambda} C_{xy} \partial_{\psi_y} \partial_{\bar\psi_x}$.
In particular, it follows that $\E_{C}\theta$ has the semigroup property
\begin{equation} \label{e:E-semigroup}
  \E_{C_2}\theta \circ \E_{C_1}\theta = \E_{C_1+C_2}\theta.
\end{equation}
This formula also holds for $C$ positive \emph{semi}definite if we
take \eqref{e:fWick} as the definition of $\E_{C}\theta F$, which we
will in the sequel.  The identity \eqref{e:fWick} is a
fermionic version of the relation between Gaussian convolution and
the heat equation, and \eqref{e:E-semigroup}, which follows
from~\eqref{e:fWick}, is the analogue of the fact that the sum of
two independent Gaussian processes is Gaussian with covariance given
by the sum of the covariances. The identity \eqref{e:fWick} allows
for the evaluation of moments, e.g.,
$\E_{C}\theta \bar\psi_{x}\psi_{y}=\bar\psi_x\psi_y+C_{xy}$. An
important consequence of the finite range property \eqref{e:C-range}
of $C_j$ is that if $F_i \in \cN(X_i)$ with
$\dist_{\infty}(X_1,X_2)> \frac12 L^j$ then, by \eqref{e:fWick},
\begin{equation}
  \E_{C_j}{\qB{\theta \p{F_1F_2}}} = \pB{\E_{C_j}[\theta F_1]}\pB{\E_{C_j}[\theta F_2]}.
\end{equation}

\subsection{Symmetries}
\label{sec:symmetries}

We briefly discuss symmetries, which are important in extracting the
relevant and marginal contributions in each renormalisation group step
(see Section~\ref{sec:Loc} below).  We call an element
$F \in \cN(\Lambda)$ \emph{symplectically invariant} or
\emph{$U(1)$ invariant} if every monomial in its representation has
the same number of factors of $\bar\psi$ and $\psi$.  We remark that
in \cite{MR3332938,MR3332939}, to which we will sometimes refer,
this property is called (global) gauge invariance. Similarly,
$F \in \cN(\Lambda\sqcup\Lambda)$ is $U(1)$ invariant if the
combined number of factors of $\bar\psi$ and $\bar\zeta$ is the same
as the combined number of factors of $\psi$ and $\zeta$.  We denote by
$\cN_{\rm sym}(X)$ the subalgebra of $\cN(X)$ of $U(1)$
invariant elements and likewise for
$\cN_{\rm sym}(\Lambda\sqcup\Lambda)$.  The maps $\theta$ and $\E_{C}$
preserve $U(1)$ symmetry.

A bijection $E\colon \Lambda\to\Lambda$ is an \emph{automorphism}
of the torus $\Lambda$ if it
maps nearest neighbours to nearest neighbours. Bijections act as homomorphisms
on the algebra $\cN(\Lambda)$ by $E\psi_x = \psi_{Ex}$ and
$E\bar\psi_x = \bar\psi_{Ex}$ and similarly for
$\cN(\Lambda\sqcup\Lambda)$. If $C$ is invariant under lattice
symmetries, i.e., $C(Ex,Ey)=C(x,y)$ for all automorphisms
$E$, then the convolution $\E_C\theta$ commutes
with automorphisms of $\Lambda$, i.e.,
$E \E_{C}\theta F = \E_{C}\theta E F$. In particular $E\E_{C_{j}}\theta
F = \E_{C_{j}}\theta E F$ for the covariances of the finite range decomposition
\eqref{e:C-def}.
An important consequence of this discussion is that if
$X\subset \Lambda$, $F\in \cN_{\rm sym}(X)$ and $F$ is
invariant under lattice symmetries that fix $X$, then
$\E_{C_{j}}\theta F\in \cN_{\rm sym}(X)$ is also invariant
under such lattice symmetries.

\subsection{Polymer coordinates}
\label{sec:polymers}

We will use \eqref{e:E-semigroup} and the
decomposition \eqref{e:C-def} to study the progressive integration
\begin{equation} \label{e:EZj}
  Z_{j+1} = \E_{C_{j+1}} \qB{\theta Z_j }, 
\end{equation}
for a given $Z_{0}\in \cN(\Lambda)$. To be concrete here, the reader
may keep $Z_0=e^{-V_0(\Lambda)}$ with $V_{0}(\Lambda)$
from~\eqref{e:V0} in mind, but to compute correlation functions we
will consider generalisations of this choice of $Z_{0}$ in
Section~\ref{sec:obs}.  The analysis is performed by defining suitable
coordinates \emph{(polymer coordinates)} and norms (on
polymer coordinates) that enable the progressive integration to be
treated as a dynamical system: this is the \emph{renormalisation
group}. Towards this end, this section defines local polymer coordinates
as in \cite{MR2523458,MR3332942}. Section~\ref{sec:norms}
then defines relevant norms, and norms on polymer coordinates are
introduced in Section~\ref{sec:step} after other preliminary
material is introduced.

\begin{figure}
  \input{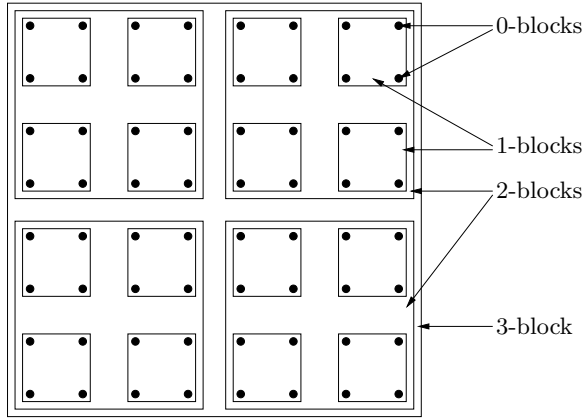}
  \caption{Illustration of $j$-blocks when $L=2$.\label{fig:blocks}}
\end{figure}

\subsubsection{Blocks and Polymers}
Recall $\Lambda\bydef\Lambda_{N}$ denotes a torus of side length
$L^{N}$. Partition $\Lambda_N$ into nested scale-$j$ blocks $\cB_j$ of
side lengths $L^j$ where $j=0,\dots, N$.  Thus scale-$0$ blocks are
simply the points in $\Lambda$, while 
the only scale-$N$ block is $\Lambda$ itself, see
Figure~\ref{fig:blocks}.  The set of \emph{$j$-polymers}
$\cP_j=\cP_j(\Lambda)$ consists of finite unions of blocks in
$\cB_j$. To define a notion of connectedness, say $X,Y \in \cP_j$
\emph{do not touch} if $\inf_{x\in X,y\in Y} |x-y|_\infty > 1$.  A
polymer is \emph{connected} if it is not empty and there is a path of
touching blocks between any two blocks of the polymer.  The subset of
connected $j$-polymers is denoted $\cC_j$.  We will drop $j$- prefixes
when the scale is clear.

For a fixed $j$-polymer $X$, let $\cB_j(X)$ denote the set of $j$-blocks
contained in $X$ and let $|\cB_j(X)|$ be the number of such blocks.
Connected polymers $X$ with $|\cB_j(X)| \leq 2^d$ are called
\emph{small sets} and the collection of all small sets is denoted
$\cS_j$. Polymers which are not small will be called \emph{large}.
Finally, for $X \in \cP_j$ we define its \emph{small set
neighbourhood} $X^\square \in\cP_j$ as the union of all small sets
containing a block in $\cB_j(X)$, and its \emph{closure}
$\bar{X}$ as the smallest $Y\in\cP_{j+1}$ such that $X\subset Y$.

\subsubsection{Coordinates}

We will write $Z_j$ in the form
\begin{equation} \label{e:ZVK-bis}
  Z_{j}
  = e^{-u_{j}|\Lambda_{N}|} \sum_{X\in\cP_{j}} e^{-V_{j}(\Lambda_{N}\setminus X)} K_{j}(X),
\end{equation}
where the $u_j$ are constants (essentially the free energy), the
$V_j(X)$ are functions of the fields $\bar\psi_x,\psi_x$ for $x$ in a
neighbourhood of $X$, parametrised by finitely many \emph{coupling
constants} which require special attention (and are independent of
$X$), and everything else is organised into the functions $K_j(X)$,
which will be called \emph{polymer activities}.  Unlike the $V_{j}$,
the polymer activities track quantities whose precise value is not
important.  Explicit (somewhat complicated) formulas for the evolution
of $K_{j}$ will be given below.  An essential point will be that they
can be tracked in terms of estimates.  The tuple $(V_j,K_j)$ together
with the representation \eqref{e:ZVK-bis} will be referred to as
\emph{polymer coordinates}.  In the remainder of this subsection, we
will discuss some structural properties of these coordinates.

\medskip\noindent\emph{Coupling constants.}
We will always identify $V_j$ with the coupling constants which parametrise it.
Explicitly,
for coupling constants $V_j=(z_j,y_j,a_j,b_j) \in \C^4$
and a set  $X \subset \Lambda_{N}$, let
\begin{multline} \label{e:VX-bis}
  V_j(X)
  =
  \\
  \sum_{x\in X} \qa{ y_j (\nabla\psi)_x (\nabla\bar\psi)_x + \frac{z_j}{2}((-\Delta\psi)_x\bar\psi_x + \psi_x(-\Delta\bar\psi)_x) + a_j \psi_x\bar\psi_x + b_j \psi_x\bar\psi_x (\nabla \psi)_x(\nabla\bar\psi)_x}
  .
\end{multline}
For the scale $j=0$, if we set $ Z_0 = e^{-V_0(\Lambda_{N})}$, then the
polymer coordinates take the simple form
\begin{equation} \label{e:Z0}
  Z_0 = e^{-V_0(\Lambda_{N})}
  = e^{-u_{0}|\Lambda_{N}|}\sum_{X \subset \Lambda_{N}} e^{-V_0(\Lambda_{N} \setminus X)} K_0(X),
\end{equation}
with
\begin{equation}
  K_0(X) = 1_{X=\varnothing}, \qquad u_{0}=0.
\end{equation}
To study the recursion $ Z_{j+1}=\E_{C_{j+1}}\theta Z_{j}$ at a general scale $j=1,\dots,N$, we will make a choice of coupling constants $V_{j}$
and of polymer activities $K_j=(K_j(X))_{X\in \cP_j(X)}$ such that
\begin{equation} \label{e:ZVK}
  Z_{j}
  = e^{-u_{j}|\Lambda_{N}|} \sum_{X\in\cP_{j}} e^{-V_{j}(\Lambda_{N}\setminus X)} K_{j}(X).
\end{equation}

\smallskip\noindent\emph{Polymer activities.}
The $K_j$ will be defined in such a way that they satisfy the locality and symmetry property
$K_j(X) \in \cN_{\rm sym}(X^\square)$ and the following important
component factorisation property: for $X,Y\in \cP_{j}$ that do not
touch,
\begin{equation} \label{e:K-factor}
  K_j(X \cup Y) = K_j(X)K_j(Y).
\end{equation}
Note that since they are $U(1)$ symmetric, the $K_j(X)$ are even elements of $\cN$, so they commute and the
product on the right-hand side is unambiguous.
Using the previous identity,
\begin{equation}
  \label{e:K-prod}
  K_j(X) = \prod_{Y \in \Comp(X)} K_j(Y),
\end{equation}
where $\Comp(X)$ denotes the set of connected components of the polymer $X$.
In particular, 
each $K_j= (K_{j}(X))_{X\in \cP_{j}(\Lambda_{N})}$ satisfying \eqref{e:K-factor} can be identified
with its restriction $K_j= (K_{j}(X))_{X\in \cC_{j}(\Lambda_{N})}$.
We say that $K_{j}$ is automorphism invariant if
$EK_{j}(X) = K_{j}(E(X))$ for all $X\in \cP_{j}(\Lambda_{N})$ and all
torus automorphisms $E \in {\rm Aut}(\Lambda_N)$ that map blocks in
$\cB_j$ to blocks in $\cB_j$.

\begin{definition} \label{def:cKbulk}
Let $\cKbulk_j(\Lambda_N)$ be the linear space of
automorphism invariant $K_j = (K_j(X))_{X\in\cC_j(\Lambda_N)}$
with $K_j(X) \in \cN_{\rm sym}(X^\square)$ for every $X \in \cC_j$. 
\end{definition}

Polymer coordinates at scale $j$ are thus a choice of (the coupling constants) $V_j$ together with a choice of (polymer activities) $K_j$
from the space $\cKbulk_{j}$.
The renormalisation group map is a particular choice of a map
$(V_j,K_j)\mapsto (V_{j+1},K_{j+1})$. 

For a given $Z_j$,
the above conditions do \emph{not} determine $K_j$ uniquely given
$V_j$ (see the proof of Proposition~\ref{prop:ZVK}, where the
non-uniqueness is apparent). We will state our specific choice
  of such a map in Section~\ref{sec:step-def} below. 
The goal is to choose $V_{j}$ such that the size of the $K_{j}$
decrease rapidly as $j$ increases when the sizes of $V_{j}$ and
$K_{j}$ are measured in appropriate norms.
Thus $K_j$ will capture the \emph{irrelevant} (or contracting) directions of the
renormalisation group dynamics, while the \emph{relevant} (or expanding) and \emph{marginal}
directions will be captured by the $V_j$ coordinates.
The next section defines the norms we will use.

\subsection{Norms}
\label{sec:norms}

We now define the $T_j(\ell)$ norms we will use on the Grassmann algebras $\cN(\Lambda)$.
General properties of these norms were systematically developed in \cite{MR3332938}, to which we will refer for some proofs.
To help the reader, in places where we specialise the definitions
of~\cite{MR3332938} we indicate the more general notation that is
used in~\cite{MR3332938}.

We start with some notation.
For any set $S$, we write $S^*$ for the set of finite sequences in $S$.
We write $\Lambda_f = \Lambda \times \{\pm 1\}$
and for $(x,\sigma) \in \Lambda_f$ we write $\psi_{x,\sigma} = \psi_x$ if $\sigma=+1$ and $\psi_{x,\sigma}=\bar\psi_x$ if $\sigma=-1$.
Then every element $F \in \cN(\Lambda)$ can be written in the form
\begin{equation} \label{e:Fzsum}
  F = \sum_{z\in\Lambda_f^*} \frac{1}{z!}F_z \psi^z
\end{equation}
where $\psi^z = \psi_{z_1}\cdots \psi_{z_n}$ if $z=(z_1,\dots, z_n)$.
We are using the notation that $z!=n!$ if the sequence $z$ has length
$n$.  The representation in~\eqref{e:Fzsum} is in general not
unique. To obtain a unique representation we require that the $F_{z}$
are antisymmetric with respect to permutations of the components of
$z$ (this is possible due to the antisymmetry of the Grassmann variables).
Antisymmetry implies that $F_{z}=0$ if $z$ has length
exceeding $2|\Lambda|$ or if $z$ has any repeated entries.

\begin{definition} \label{def:Phi}
  Let $p_{\Phi}=2d$.
  The space of test functions
  $\Phi_j(\ell)$ is defined as the set of functions
  $g\colon \Lambda_f^* \to \R$, $z\mapsto g_{z}$ together with
  norm
  \begin{equation} \label{e:Phi-def} \|g\|_{\Phi_j(\ell)} =
    \sup_{n\geq 0}\sup_{z\in\Lambda_f^n} \sup_{|\alpha_i| \leq  p_\Phi}
    \ell^{-n}L^{j(|\alpha_1|+\cdots|\alpha_n|)}|\nabla^{\alpha_1}_{z_1}
    \cdots \nabla^{\alpha_n}_{z_n} g_z|.
  \end{equation}
\end{definition}
In this definition, $\nabla_{z_i}^{\alpha_i}$ denotes the discrete
derivative $\nabla^{\alpha_i}$ with multiindex $\alpha_i$ acting on
the spatial part of the $i$th component of the finite sequence $z$.

The $\Phi_j(\ell)$ norm measures spatial smoothness of test functions,
which act as substitutes for fields.  Restricted to sequences of fixed
length, it is a lattice $C^{p_\Phi}$ norm at spatial scale~$L^j$ and
field scale~$\ell$.  We will mainly use the following choice of $\ell$
when using the $\Phi_j(\ell)$ norm:
\begin{equation}
  \label{e:ellj}
  \ell_j = \ell_0 L^{-\frac12(d-2)j}
\end{equation}
for a large constant $\ell_0$, and $\ell_{j}$ will always be as in~\eqref{e:ellj}.
This choice captures the size of the covariances in the decomposition \eqref{e:C-def}.
Indeed, regarding the covariances $C_{j}$ as functions of sequences of length $2$
(i.e., as the coefficient in \eqref{e:Fzsum} of $F= \sum_{x,y} \bar\psi_{x}\psi_{y}C_{j}(x,y)$),
the bounds \eqref{e:C-bd} imply
\begin{equation} \label{e:Cleq1}
  \|C_{j}\|_{\Phi_j(\ell_j)} \leq 1,
\end{equation}
when $\ell_0$ is chosen as a large ($L$-dependent, due to the
index $j+1$ on the left-hand side of~\eqref{e:C-bd}) constant relative
to the constants in  \eqref{e:C-bd} with $|\alpha|\leq 2p_\Phi$.
From now on, we will always assume that $\ell_0$ is fixed in this way.

\begin{definition} \label{def:Tphi}
 We define $T_j(\ell)$ to be the algebra $ \cN(\Lambda)$ together with the dual norm
 \begin{equation} \label{e:Tphi-def}
   \|F\|_{T_j(\ell)} = \sup_{\|g\|_{\Phi_j(\ell)} \leq 1}| \avg{F,g}|
   ,
   \qquad \text{where }
   \avg{F,g} = \sum_{z \in \Lambda_f^*}  \frac{1}{z!} F_z g_z
 \end{equation}
 when $F\in \cN(\Lambda)$ is expressed as in \eqref{e:Fzsum}.
 
 An analogous definition applies to $\cN(\Lambda \sqcup \Lambda)$, and we
 then write $T_j(\ell \sqcup \ell) \bydef T_j(\ell)$ for this norm
 (with the first notation to emphasise the doubled algebra),
 where we recall that $\cN(\Lambda\sqcup\Lambda)$ is defined above \eqref{e:theta-def}.
\end{definition}

The $T_j(\ell)$ norm measures smoothness of field functionals $F\in \cN(\Lambda)$ 
with respect to fields whose size is measured by $\Phi_j(\ell)$.
They therefore implement the power counting on which renormalisation relies.
Important, but relatively straightforwardly verified, properties of these norms are systematically developed in \cite{MR3332938}; we summarise the ones we need now.

\smallskip\noindent\emph{Product property.}
First, the $T_j(\ell)$ norm defines a Banach algebra,
i.e., the following product property holds  (see \cite[Proposition~3.7]{MR3332938}):
for $F_{1},F_{2} \in \cN(\Lambda)$,
\begin{equation} \label{e:T-product}
  \|F_1F_2\|_{T_j(\ell)} \leq \|F_1\|_{T_j(\ell)}\|F_2\|_{T_j(\ell)}.
\end{equation}
Using the product property, we may gain some intuition regarding these norms by  considering the following simple examples:
\begin{gather}
  \|\psi_x\bar\psi_x\|_{T_j(\ell)} \leq \|\psi_x\|_{T_j(\ell)} \|\bar\psi_x\|_{T_j(\ell)}= \ell^2, \\
  \|(\nabla_e\psi)_x\bar\psi_x\|_{T_j(\ell)}
  \leq \|\nabla_e\psi_x\|_{T_j(\ell)} \|\bar\psi_x\|_{T_j(\ell)} = \ell^2 L^{-j}.
\end{gather}
The following more subtle example relies on $\psi_{x}^{2}=\bar\psi_{x}^{2}=0$
and plays an important role for our model:
\begin{equation}
  \|\psi_x\bar\psi_x\psi_{x+e}\bar\psi_{x+e}\|_{T_j(\ell)}
  = \|\psi_x\bar\psi_x(\nabla_e\psi)_{x}(\nabla_e\bar\psi)_{x}\|_{T_j(\ell)}
  \asymp \ell^4L^{-2j}.
\end{equation}
In general each factor of the fields contributes a factor $\ell$ and
each derivative a factor $L^{-j}$.

\smallskip\noindent\emph{Monotonicity.} 
Second, as follows immediately from the definition, the following monotonicity properties hold:
for $\ell\leq \ell'$ and $F \in \cN(\Lambda)$,
\begin{equation} \label{e:T-monotone}
  \|F\|_{T_j(\ell)} \leq  \|F\|_{T_j(\ell')}, \qquad
  \|F\|_{T_{j+1}(\ell)} \leq  \|F\|_{T_j(\ell')}.
\end{equation}

\smallskip\noindent\emph{Doubling map.} 
Third, the doubling map satisfies (see \cite[Proposition~3.12]{MR3332938}): for $F \in \cN(\Lambda)$,
\begin{equation} \label{e:theta-bd}
  \|\theta F\|_{T_j(\ell)} \leq \|F\|_{T_j(2\ell)}
\end{equation}
where the norm on the left-hand side is the
$T_j(\ell) \bydef T_j(\ell\sqcup \ell)$
 norm on $\cN(\Lambda\sqcup\Lambda)$.

 \smallskip\noindent\emph{Gram inequality.} 
Finally, the following contraction bound for the fermionic Gaussian
expectation is an application of the Gram inequality whose importance
is well-known in fermionic renormalisation.  It is proved in
\cite[Proposition~3.19]{MR3332938}.

\begin{proposition}
  Assume $C$ is a covariance matrix with
  $\|C\|_{\Phi_{j}(\ell)} \leq 1$.  For
  $F \in \cN(\Lambda\sqcup\Lambda)$, then
  \begin{equation} \label{e:E-contract2}
    \|\E_{C}F\|_{T_j(\ell)} \leq \|F\|_{T_j(\ell)}.
  \end{equation}
  In particular, for $F\in\cN(\Lambda)$, by \eqref{e:theta-bd} the fermionic Gaussian
  convolution satisfies
  \begin{equation} \label{e:E-contract}
    \|\E_{C} \theta F\|_{T_{j}(\ell)}
    \leq \|F\|_{T_{j}(2\ell)}.
  \end{equation}
\end{proposition}

For our choices of $\ell_j$ and of the finite range covariance matrices $C_j$,
the inequalities  \eqref{e:T-monotone} and \eqref{e:E-contract} in particular imply
\begin{equation} \label{e:normcontract}
  \|F\|_{T_{j+1}(\ell_{j+1})} \leq
  \|F\|_{T_{j+1}(2\ell_{j+1})} \leq \|F\|_{T_j(\ell_j)},
  \qquad
  \|\E_{C_{j+1}}\theta F\|_{T_{j+1}(\ell_{j+1})} \leq \|F\|_{T_j(\ell_j)}.
\end{equation}
We remark that the existence of this contraction estimate for the
expectation combined with \eqref{e:Loc-contract} below is what makes
renormalisation of fermionic fields much simpler than that of bosonic
ones.

\subsection{Localisation}
\label{sec:Loc}

To define the renormalisation group map we need one more
important ingredient: the \emph{localisation operators}
$\Loc_X$ and $\Loc_{X,Y}$ that will be used to extract the relevant and
marginal terms from the $K_j$ coordinate to incorporate them in the
renormalisation from $V_j$ into $V_{j+1}$. These operators are
generalised Taylor approximations which take as inputs $F\in \cN(X)$
and produce best approximations of $F$ in a finite dimensional space
of \emph{local field polynomials}.

\subsubsection*{Local field polynomials}

By \emph{formal local field polynomials} we refer to formal polynomials in the symbols
$\psi,\bar\psi,\nabla\psi,\nabla\bar\psi,\Delta\psi,\Delta\bar\psi,\nabla^2\psi,\dots$ (without spatial index).
The \emph{dimension} of a formal local field monomial is given by $(d-2)/2$ times the number of factors of $\psi$ or $\bar\psi$
plus the number of discrete derivatives $\nabla$ in its
representation, where $\Delta$ is treated as two discrete derivatives.
The classification of local monomials according to dimension is known as power counting in the renormalisation group literature.
Relevant monomials are those with dimension strictly greater than $d$, marginal ones those with dimension equal to $d$,
and irrelevant those with dimension strictly less than $d$.
Concretely, we consider the following space of formal local field polynomials,
consisting of the relevant and marginal monomials consistent with symmetry constraints.

\begin{definition} \label{def:cV}
Let $\cVbulk \cong \C^4$ be the linear space of formal local field
polynomials of the form
\begin{equation} \label{e:V-def}
  V= y (\nabla\psi) (\nabla\bar\psi) + \frac{z}{2}((-\Delta\psi)\bar\psi + \psi(-\Delta\bar\psi))
  + a \psi\bar\psi + b \psi\bar\psi (\nabla \psi)(\nabla\bar\psi)
  .
\end{equation}
\end{definition}

We will identify elements $V\in\cVbulk$ with their coupling constants
$(z,y,a,b) \in \C^4$.  Sometimes we include a constant term $u$ and write
$u+V \in \C \oplus \cVbulk$ with $u+V\cong (u,z,y,a,b)\in\C^5$.

Given a set $X\subset \Lambda$, a formal local field polynomial $P$
can be specialised to an element of $\cN(\Lambda)$ by replacing formal
monomials by evaluations. For example, if $P=\bar\psi\psi$,
$P(X) = \sum_{x\in X}\bar\psi_{x}\psi_{x}$. We call polynomials
arising in this way \emph{local polynomials}.  The most important case
is $V \mapsto V(X)$, with
\begin{equation} \label{e:VX}
  V(X) =
  \sum_{x\in X} \qa{ y (\nabla\psi)_x (\nabla\bar\psi)_x + \frac{z}{2}((-\Delta\psi)_x\bar\psi_x + \psi_x(-\Delta\bar\psi)_x) + a \psi_x\bar\psi_x + b \psi_x\bar\psi_x (\nabla \psi)_x(\nabla\bar\psi)_x}
  ,
\end{equation}
where $\Delta \bydef -\frac12 \sum_{e\in \cE_d} \nabla_{-e}\nabla_{e}$
and
$(\nabla\psi)_{x}(\nabla\bar\psi)_{x} \bydef \frac12
\sum_{e\in\cE_d}\nabla_{e}\psi_{x}\nabla_{e}\bar\psi_{x}$ are the
lattice Laplacian and the square of the lattice gradient; recall that
$\cE_d = \{e_1, \dots, e_{2d}\}$.  For a constant $u\in \C$ we write
$u(X) = u|X|$, where $|X|$ is the number of points in $X \subset
\Lambda$.  Thus $(u+V)(X)= u(X)+V(X)=u|X|+V(X)$.

\begin{definition}
  For $X\subset \Lambda$,
  define $\cVbulk(X) = \{V(X): V\in \cVbulk\} \subset \cN(\Lambda)$ 
  and analogously $(\C \oplus \cVbulk)(X) = \{ u|X|+V(X): u\in \C, \,
  V \in \cVbulk\} \subset \cN(\Lambda)$.
\end{definition}

The space $\cVbulk$ contains all formal local field polynomials
whose constituent monomials have dimension at most $d$ that are (i)
$U(1)$ invariant, (ii) respect lattice symmetries (if $EX=X$
for an automorphism $E$, then $EV(X)=V(X)$), (iii)
$V(X)\neq 0$, and (iv) have no constant terms.  Note that
  $\E_C\theta$ preserves $\cVbulk(X)$ by the discussion in
  Section~\ref{sec:symmetries}.  We emphasise that there is
no $(\bar\psi\psi)^2$ term, which would be consistent with
  having dimension as most $d$ (if $d=3,4$) and symmetries, because
it vanishes upon specialisation by anticommutativity of the fermionic
variables.

Two further remarks are in order.
First, the monomial $\psi\bar\psi (\nabla\psi)(\nabla\bar\psi)$ has dimension
$2d-2>d$ for $d\geq 3$; we include it in $\cVbulk$ 
since it occurs in the initial potential. Second,
the monomials multiplying $z$ and $y$ are equivalent upon
specialisation when $X=\Lambda$
by summation by parts, and differ
only by boundary terms for general $X \subset \Lambda$.
This would
allow us to keep only one of them, but it will be simpler to keep both.

\subsubsection*{Localisation}
The localisation operators $\Loc_X$ and $\Loc_{X,Y}$ associate local field
monomials to elements of $\cN(X)$.  In renormalisation group
terminology, the image of $\Loc$ projects onto the space of all
relevant and marginal local polynomials.
The precise definitions of the localisation operators do not play a
direct role in this paper.  Rather, only their abstract properties,
summarised in the following Proposition~\ref{prop:Loc}, will be required.
Nonetheless, to give some intuition for the action of $\Loc_X$ and
$\Loc_{X,Y}$, we include the following typical examples (see also
\cite[Section~1.5]{MR3332939}).  The examples indicate (as
stated at the beginning of this section) that the localisation
operators are generalised Taylor approximations.

\begin{example}
  (i) Let $F$ be a monomial in $\{\bar\psi_x, \psi_x\}$ of degree greater than four. Then $\Loc_X F = 0$.

  \smallskip
  \noindent
  (ii) Consider $F= \sum_{x\in X}\sum_{y\in\Lambda}
  q(x-y)\bar\psi_y\psi_y$ where the kernel $q\colon \Z^d \to \R$ has
  finite support and is invariant under lattice rotations. Then 
  provided $\Lambda$ is large enough,
  \begin{equation} \label{e:ex-Loc}
    \Loc_X F = \sum_{x\in X} \qa{ q^{(1)} \bar\psi_x\psi_x +
      q^{(**)}\pa{\frac12 \bar\psi_x(\Delta\psi)_x + \frac12
        (\Delta\bar\psi)_x\psi_x + (\nabla\bar\psi)_{x}(\nabla\psi)_{x}} } = \sum_{x\in X} P_x,
  \end{equation}
  where $q^{(1)}= \sum_{y \in\Z^d} q(y)$ and
  $q^{(**)} = \sum_{y \in\Z^d} y_1^2 q(y)$ (and $y_1$ denotes the
  first component of $y\in\Z^d$), and with the same $P_x$ as in
  \eqref{e:ex-Loc},
  \begin{equation}
    \Loc_{X,Y} F = \sum_{y\in Y} P_y.
  \end{equation}
  Thus $\sum_{i=1}^n \Loc_{X,X_i} F =\Loc_X F$ if $X$ is the disjoint union of $X_1,\dots, X_n$.
\end{example}

For the definition of $\Loc_X$ and $\Loc_{X,Y}$, we use the
general framework developed in \cite{MR3332939}.  In short, the
definitions of $\Loc_X$ and $\Loc_{X,Y}$ are those of
\cite[Definition~1.6 and~1.15]{MR3332939}.  These definitions require
a choice of field dimensions, which we choose as
$[\psi]=[\bar\psi] = (d-2)/2$, a choice of maximal field dimension
$d_+$, which we choose as $d_{+}=d$, and a choice of a space $\hat P$
of test polynomials, which we define exactly as in
\cite[(1.19)]{MR3332939} with the substitution
$\nabla_e\nabla_e\to -\nabla_e\nabla_{-e}$ explained in
\cite[Example~1.3]{MR3332939}.  The following properties are then
almost immediate from \cite{MR3332939}.

\begin{proposition} \label{prop:Loc} For $L=L(d)$ sufficiently large
  there is a universal $\bar C>0$ such that: for $j<N$ and any small
  sets $Y \subset X \in \cS_j$, the linear maps
  $\Loc_{X,Y}\colon \cN(X^\square) \to \cN(Y^\square)$ have the following properties:
  \smallskip
  
  \noindent
  (i) They are bounded:
  \begin{equation} \label{e:Loc-bd}
    \|\Loc_{X,Y}F\|_{T_{j}(\ell_{j})} \leq \bar C \|F\|_{T_j(\ell_{j})}.
  \end{equation}
\smallskip
  \noindent
  (ii) The maps $\Loc_X \bydef \Loc_{X,X}\colon \cN(X^\square) \to \cN(X^\square)$ satisfy the contraction bound
  \begin{equation} \label{e:Loc-contract}
    \|(1-\Loc_X)F\|_{T_{j+1}(2\ell_{j+1})} \leq \bar C L^{-d}
    {L^{-\dplus}}
\|F\|_{T_j(\ell_j)}.
  \end{equation}
\smallskip
  \noindent
  (iii) If $X$ is the disjoint union of $X_1, \dots, X_n$ then
  $\Loc_X = \sum_{i=1}^n \Loc_{X,X_i}$.
  
  \smallskip
  \noindent
  (iv) The maps are Euclidean invariant:
  if $E\in {\rm Aut}(\Lambda_{N})$ then
  $E \Loc_{X,Y} F = \Loc_{EX,EY} EF$.

  \smallskip
  \noindent
  (v) For a block $B$, small polymers $X_1,\dots, X_n$, and any $F_i \in \cN_{\rm sym}(X_i^\square)$ such that $\sum_{i=1}^n\Loc_{X_i,B} F_i$ is invariant
  under automorphisms of $\Lambda_N$ that fix $B$,
  \begin{equation} \label{e:Loc-image}
    \sum_{i=1}^n \Loc_{X_i,B}F_i \in (\C \oplus \cVbulk)(B).
  \end{equation}
\end{proposition}
We remark that the image of $\Loc_{X,Y}$ is in general a
larger space of local field monomials than $\cVbulk(Y)$, often denoted
$\cV$ in \cite{MR3332939} --- for example first gradients of
the field can arise which only need cancel upon the symmetrisation
in \eqref{e:Loc-image}. Since we will not use this larger space directly
we have not assigned a symbol for it.

\begin{proof}[Proof of Proposition~\ref{prop:Loc}]
The bound (i) is \cite[Proposition~1.16]{MR3332939},
the contraction bound (ii) is \cite[Proposition~1.12]{MR3332939},
the decomposition property (iii) holds by the definition of $\Loc_{X,Y}$ in \cite[Definition~1.15]{MR3332939},
and the Euclidean invariance (iv) is \cite[Proposition~1.9]{MR3332939}.
Note that the parameter $A'$ in \cite[Proposition~1.12]{MR3332939} does
not appear here as it applies to the boson field $\phi$; our fermionic context corresponds to $\phi=0$.
For the application of \cite[Proposition~1.12]{MR3332939} we have used
that $p_\Phi$ was fixed to be $2d$ in Definition~\ref{def:Phi},
and that we have only considered the action of $\Loc$ on small sets.

Finally, property (v) follows from
\cite[Proposition~1.10]{MR3332939} and the fact that the space
$\cVbulk$ defined in Definition~\ref{def:cV} contains all local
polynomials of dimension at most $d$ invariant under lattice
automorphisms that fix a point. 
\end{proof}

\subsection{Definition of the renormalisation group map}
\label{sec:step-def}

The renormalisation group map $\Phi_{j+1} = \Phi_{j+1,N,m^{2}}$ is a map
\begin{equation}
  \Phi_{j+1}\colon (V_j,K_j) \mapsto (u_{j+1},V_{j+1},K_{j+1})
\end{equation}
acting on
\begin{equation}
  V_j  \in \cVbulk, 
  \qquad K_j \in \cKbulk_j(\Lambda_N), 
\end{equation}
with $u_{j+1}\in\C$, the space of coupling constants $\cVbulk$ as in Definition~\ref{def:cV}, and
the space of polymer activities
$\cKbulk_j(\Lambda_N)$ as in Definition~\ref{def:cKbulk}.
The map will have mild dependence on $m^2$ and $\Lambda_N$ as a
consequence of this dependence of the covariance matrices. As
indicated above the $u$-coordinate does not influence the dynamics of
the remaining coordinates. Thus we can always explicitly assume that
the incoming $u$-component of $\Phi_{j+1}$ is $0$ and separate it from
$V_{j+1}$ in the output.  This means that we will often regard
$\Phi_{j+1}$ as a map $(V_j,K_j) \mapsto (V_{j+1},K_{j+1})$
where $u_j=u_{j+1}=0$.

The explicit definition of the map $\Phi_{j+1}$ is given in
\eqref{e:V+-def} and \eqref{e:K+-def} below. The essential
consequences of the definition are
Proposition~\ref{prop:ZVK}, which enables the iterative application
of the renormalisation group maps, and the estimates of
Theorem~\ref{thm:step}.

At first sight, the definition of $\Phi_{j+1}$ may appear
somewhat complicated, but it follows from simple principles
that are outlined in the proof of Proposition~\ref{prop:ZVK} below.
Compared to other implementations of the
fermionic renormalisation group, the finite range property of the
covariances in our implementation means we do not require infinite
expansions, nor do we require norms which control the spatial
complexity of polymers beyond simple volume estimates.  As such,
establishing useful norm estimates becomes an essentially
combinatorial problem.  This feature is especially useful in models
with bosonic fields, see \cite{MR2070102,MR2523458} and
\cite[Appendix~A]{MR3969983} for introductory discussions, but it
also provides appealing features in the present fermionic
context. For example, the flows on tori
with two distinct side lengths $L^{N_1}<L^{N_2}$ \emph{coincide} up
to the final length scale $L^{N_1-1}$ 
for polymers which do not wrap around either torus, making the
definition of the infinite volume flow and its relation to the
finite volume one particularly transparent (see
Proposition~\ref{prop:consistency} below).

For the definition of the renormalisation group map $\Phi_{j+1}$,
we identify $V_j \in \cVbulk$ with the tuple $(V_j(B))_{B \in \cB_j(\Lambda_N)}$,
i.e., the field monomials corresponding to the coupling constants $V_j$ evaluated over a block $B$,
and the tuple $(K_j(X))_{X\in \cC_j(\Lambda_N)}$ with its extension $(K_j(X))_{X\in\cP_j(\Lambda_N)}$
determined by the component factorisation property \eqref{e:K-factor}.
We also introduce, assuming $j+1<N$,
\begin{align} \label{e:Q-def}
  Q(B) &= \sum_{X\in \cS_j: X \supset B} \Loc_{X,B} 
         K_j(X), &\qquad& (B \in \cB_j),  \\
  \label{e:J-defB}
  J(B,B) &= - \sum_{X\in \cS_j \setminus \cB_j: X \supset B} \Loc_{X,B} 
           K_j(X), &\qquad& (B \in \cB_j),\\
  \label{e:J-def}
  J(B,X) &= \Loc_{X,B} 
           K_j(X), &\qquad& (X \in \cS_j \setminus \cB_j, B \in \cB_j(X)),
\end{align}
and $J(B,X)=0$ otherwise. If $j+1=N$ we simply set $Q=J=0$.

As a consequence of the properties of $\Loc$ from
Proposition~\ref{prop:Loc} (c.f.~in particular property (v)), $Q(B)$
arises from an element of $\cVbulk$ and represents the marginal and
relevant contributions from $K_j$ associated with the block $B$.
These contributions (which are marginal or relevant in the sense of
power counting) only come
from $K(X)$ for small sets $X$, as large
sets $X$ will yield contracting contributions for entropic reasons
(as opposed to power counting reasons) considered later. The
$J(B,X)$ are a technical device for removing the $Q$ contribution
from $K_j$.  An important property is that
\begin{equation} \label{e:J-zero}
  \sum_X J(B,X) = 0, \qquad (B \in \cB_j).
\end{equation}
The application of this property occurs in \eqref{e:K-crucial}.
We indicate the motivation for the form of $J$ below. For a fuller
discussion we refer to \cite[Lectures~4--5]{MR2523458} and~\cite[Appendix~A, 12.3.2]{rg-brief}.

\medskip
We will specify the $V$- and $K$-components of the renormalisation group map separately.

\medskip
\noindent
\emph{$V$-component.}
The first definition defines the $V$-component of the renormalisation group map.
This map is given by first-order perturbation theory, i.e., $\E_{C_{j+1}}[\theta V_j(B)]$,
plus the higher-order contribution $\E_{C_{j+1}}[\theta Q(B)]$ representing the marginal and relevant
contributions from $K_j$ as discussed above.
Here recall the definition of the doubling map $\theta$ from \eqref{e:theta-def}, i.e.,
$\theta F$ is obtained from $F$ by replacing $\psi$ by $\psi+\zeta$ and $\bar\psi$ by $\bar\psi+\bar\zeta$,
and that the expectation only acts on $(\zeta,\bar\zeta)$.

\begin{definition} \label{def:V+}
The map $(V_j,K_j) \mapsto (u_{j+1},V_{j+1})$ is defined by
\begin{equation} \label{e:V+-def}
  u_{j+1}|B| + V_{j+1}(B) = 
  \E_{C_{j+1}}\qB{\theta \p{V_j(B) - Q(B)}}, \qquad (B\in\cB_{j}).
\end{equation}
\end{definition}

We emphasise that $V_{j+1}$ is evaluated on $B\in\cB_j$
here; $V_{j+1}$ can then be
extended to $\cB_{j+1}$ by additivity.  When $K_{j}$ is automorphism
invariant, which is the case if $K_{j}\in\cKbulk_{j}(\Lambda_{N})$,
the right-hand side of \eqref{e:V+-def} is in $(\C\oplus\cVbulk)(B)$
and can thus be identified with an element of
$\C\oplus\cVbulk \cong \C^5$.  This can be checked by using
Proposition~\ref{prop:Loc}~(iv) and (v) and the properties of
progressive integration discussed in
Section~\ref{sec:grassm-gauss-integr}.  Recall that we sometimes write
the left-hand side as $(u+V)_{j+1}(B)$. Since $V_{j+1}(B)$ has no
constant term by definition, the constant $u_{j+1}$ is unambiguously defined.

\medskip
\noindent
\emph{K-component.}  The following formula for the
  $K$-component of the renormalisation group map is more involved.  It
  is engineered to achieve the desired factorisation and contraction
  properties of the renormalisation group map.  The explicit formula
  will enable a relatively straightforward verification of the
  estimates which follow from it; the formula itself is the result of
  relatively simple manipulations explained in the proof
  Proposition~\ref{prop:ZVK} below.

\begin{definition} \label{def:K+}
For $U \in \cP_{j+1}$,
the map $(V_j,K_j) \mapsto K_{j+1}(U)$  is defined by
\begin{equation} \label{e:K+-def}
  K_{j+1}(U)=
  e^{u_{j+1}|U|}
  \sum_{(\cX,\check{X}) \in \cG(U)}  e^{-(u+V)_{j+1}(U \setminus
    \check{X} \cup X_{\cX})} \E_{C_{j+1}} \qa{ \check{K}_j(\check{X})
  \prod_{(B,X) \in \cX}\theta J(B,X)}
\end{equation}
where 
\begin{alignat}{2} \label{e:K+-def2}
  \check{K}_{j}(X) &= \prod_{W \in \Comp(X)}
  \check{K}_{j}(W),
  &\qquad
  \check{K}_{j}(W)
  &= 
    \sum_{Y \in \cP_j(W)} (\theta K_j(W\setminus Y)) (\delta I)^Y
    - \sum_{B\in \cB_j(W)} \theta J(B, W),
    \\
    \label{e:deltaI}
  (\delta I)^X &= \prod_{B \in \cB_j(X)} \delta I(B), &\qquad
  \delta I(B) &= \theta e^{-V_j(B)}-e^{-(u+V)_{j+1}(B)}.
\end{alignat}
Following \cite[Section~5.1.2]{MR2523458}, we define the set $\cG(U)$ (and the corresponding notation $\cX$ and $X_\cX$) as follows:
$\check{X}\in\cP_j$ and $\cX$ is a set of pairs $(B,X)$ with
$X \in \cS_j$ and $B \in \cB_j(X)$ with the following properties:
each $X$ appears in at most one pair $(B,X)\in \cX$,
the different $X$ do not touch, $X_{\cX}= \cup_{(B,X)\in\cX} X$
does not touch $\check{X}$,
and the closure of the union of $\check{X}$ with $\cup_{(B,X)\in \cX} B^\square$ is $U$.
\end{definition}

The following proposition is essentially
\cite[Proposition~5.1]{MR2523458}.  The only differences are that 
we have factored out the factor $e^{-u_{j+1}|\Lambda|}$ and that 
the doubling map $\theta$ is explicit (it is implicit in \cite{MR2523458}).
Explicitly, note that $K_{j+1}(U)$ and $V_{j+1}(X)$ are elements of
the Grassmann algebra $\cN(\Lambda)$, i.e., they depend on the fields $(\psi,\bar\psi)$,
but since the doubling map $\theta$ replaces $(\psi,\bar\psi)$ by $(\psi+\zeta,\bar\psi+\bar\zeta)$,
the functions $\check{K}_j(X)$ and $(\delta I)^X$ above depend on all of $(\psi,\bar\psi,\zeta,\bar\zeta)$.

For convenience and because it also demystifies the somewhat
complicated formula for the renormalisation group map, we have included a proof
along with some additional explanations.
The proof of this proposition does not rely on the specific choice of $Q$ and $J$
in \eqref{e:Q-def}--\eqref{e:J-def} or on the property \eqref{e:J-zero}.
This choice only becomes important in the proof of Theorem~\ref{thm:step}.

\begin{proposition} \label{prop:ZVK} Given $(V_j,K_j)$ define $Z_j$ by
  \eqref{e:ZVK} with $u_{j}=0$. Suppose $K_{j}$ has the factorisation
  property~\eqref{e:K-factor} with respect to $\cP_{j}$. Then with the
  above choice of $(u_{j+1},V_{j+1},K_{j+1})$ and $Z_{j+1}$ given by
  \eqref{e:ZVK} with $j+1$ in place of $j$, we have
  $Z_{j+1}=\E_{C_{j+1}}\theta Z_j$, and $K_{j+1}$ has the
  factorisation property \eqref{e:K-factor} with respect to
  $\cP_{j+1}$.  Moreover, if $K_j$ is automorphism invariant then so
  is $K_{j+1}$.
\end{proposition}

\begin{proof}
The proof essentially consists of algebraically manipulating the expression
\begin{equation}
 \label{e:Zj}
  Z_j = \sum_{X\in \cP_j} I^{\Lambda\setminus X}K(X)
\end{equation}
where $I(B) = e^{-V_j(B)}$ and $K(X)= K_j(X)$.  These manipulations
only rely on factorisation properties of $I$ and $K$ (and not on their
precise definitions). Hence we will explicitly state the required
factorisation properties, and later specialise to the context of
Proposition~\ref{prop:ZVK}.  We will use that
$I^Y = \prod_{B \in \cB_j(Y)} I(B)$ factors over blocks and $K(X)$
factors over connected components of $X$.

\smallskip\noindent
\emph{Change of coupling constants.}
Given any
$\tilde I(B) \in\cN(B)$ for $B \in \cB_j$, let  $\delta I(B) = \theta I(B)-\tilde I(B)$ and
$\tilde I^{Y}\bydef \prod_{B\in \cB_{j}(Y)}\tilde I(B)$,
i.e., $\theta I(B)$ depends on $(\psi+\zeta,\bar\psi+\bar\zeta)$ by definition of $\theta$,
$\tilde I(B)$ on $(\psi,\bar\psi)$, and $\delta I(B)$ depends on $(\psi,\bar\psi,\zeta,\bar\zeta)$.
The binomial expansion identity 
\begin{equation}
  \label{e:thetaI}
  \theta I^{\Lambda\setminus X} = \sum_{Y\subset \Lambda\setminus X}
  \tilde I^{Y}(\delta I)^{\Lambda\setminus (X\cup Y)}
\end{equation}
and \eqref{e:Zj} lead to, after changing the index of summation,
\begin{align}
  \label{e:ZtildeItildeK}
  \theta Z_j = \sum_{X\in \cP_j} \tilde I^{\Lambda\setminus X} \tilde K(X),
  \qquad
    \tilde K(X) 
  = \sum_{Y \in \cP_j(X)} (\delta I)^Y \theta K(X\setminus Y).
\end{align} 

We will later make the particular choice of $\tilde I$ that corresponds to \eqref{e:deltaI}.
  Thus $\tilde I$ corresponds to the next-scale coupling constants $V_{j+1}$.
This will be important for obtaining the desired contractive properties of the
renormalisation group map in Theorem~\ref{thm:step}.
Exhibiting that the $K$ coordinate is contractive will be aided by the following re-arrangements.

\smallskip
\noindent
\emph{Cancellation of small sets.}
Keeping in mind that $\tilde K$ factors over components (since $K$ factors over components), we can then define
$\check{K}(Y)$ to be $\tilde K(Y) - \sum_{B\in\cB_{j}(Y)}\theta J(B,Y)$
for any connected polymer $Y$ (and zero otherwise), where $J(B,Y)\in\cN(B)$ are given.
This yields the following formula for $\tilde K$:
\begin{equation} \label{e:tildeK-checkK}
  \tilde K(X) = \prod_{Y \in \Comp(X)} \pa{\check{K}(Y) + \sum_{B \in \cB_j(Y)} \theta J(B,Y)}.
\end{equation}
Again we will later specialise to $J$ as defined in~\eqref{e:J-defB} and~\eqref{e:J-def},
  in particular $J(B,Y) = 0$ unless $Y$ is a small set.
The motivation for this step is that, for $Y$ that are small sets but not blocks, the effect is that $\check{K}$
has the relevant and marginal contributions corresponding to $\sum_B J(B,Y)$ removed.
This step does not remove  the relevant and marginal contributions of blocks due to the choice \eqref{e:J-defB}.
However, relevant and marginal contributions of blocks will be removed by appropriate choice of $\tilde I$.
Both cancellations occur in the estimates in Section~\ref{sec:K-smallset}.

We next substitute \eqref{e:tildeK-checkK} into
\eqref{e:ZtildeItildeK} and re-arrange the resulting sum.  Expanding
the product, \eqref{e:tildeK-checkK} can be written as
\begin{equation}
  \tilde K(X) = \sum_{\check{X} \subset \Comp(X)}  \check{K}(\check{X}) \prod_{Y \in \Comp(X \setminus \check{X})} \sum_{B \in \cB_j(Y)} \theta J(B,Y)
  .
\end{equation}
Given the polymer $X \setminus \check{X}$, 
there is a $(X\setminus \check{X})$-dependent set of sets
$\cX 
\subset \{(B,Y): B\in \cB_{j}(Y), Y\in \cC_{j}\}$
such that
\begin{equation}
  \prod_{Y \in \Comp(X \setminus \check{X})} \sum_{B \in \cB_j(Y)}
  J(B,Y) = \sum_{\cX} \prod_{(B,Y)\in \cX} J(B,Y),
\end{equation}
where the sum on the right-hand is over the aforementioned sets of
$\cX$. Explicitly, the sets comprising $\cX$ are sets 
of pairs $(B,Y)$ where (i) $B$ is a block in $Y$, (ii) each component
$Y$ occurs in exactly one pair, and (iii) $Y$ is a component of
$X\setminus \check{X}$. In particular, $X = \check{X}\cup X_{\cX}$.
Thus
\begin{equation}
  \label{eq:tildeK}
  \tilde K(X) = \sum_{ (\cX,\check{X}) 
  } \check{K}(\check{X}) \prod_{(B,Y) \in \cX} \theta J(B,Y)
  .
\end{equation}

Substituting this into \eqref{e:ZtildeItildeK}, and using that $\tilde I^X \in \cN(\Lambda)$ is
a constant with respect to $\E_{C_{j+1}}$,
i.e., the expectation acts on $(\zeta,\bar\zeta)$ while $\tilde I^{X}$ depends on $(\psi,\bar\psi)$,
\begin{align}
    \label{e:EcthetaZ1}
  \E_{C_{j+1}} \theta Z_j
  & = \sum_{X \in \cP_{j}} \tilde I^{\Lambda\setminus X} \E_{C_{j+1}}\tilde K(X)
    \nnb
  &= \sum_{X \in \cP_j}
  \sum_{(\cX,\check{X})}
    \tilde I^{\Lambda\setminus (\check X \cup X_{\cX})}
        \E_{C_{j+1}} \qa{ \check{K}(\check{X})
    \prod_{(B,Y) \in \cX} \theta J(B,Y)}
\end{align}
where $X_{\cX}$ is by definition
$\bigcup_{(B,Y)\in \cX}Y$, and hence $X=\check{X}\cup
X_{\cX}$ by the definition of the set $\cX$.

\smallskip
\noindent\emph{Reblocking.}
Next we organise the last right-hand side of \eqref{e:EcthetaZ1} as a
sum over next-scale polymers $U\in \cP_{j+1}$.  We start by inserting
the partition of unity
\begin{equation}
  1=\sum_{U\in\cP_{j+1}}1_{\overline{\check{X}\cup[\cup_{(B,Y)\in\cX}B^{\square}]}=U}
  \qquad \text{for every $(\cX,\check{X})$}
\end{equation}%
into the sum
and changing the order of the sums. This gives
\begin{equation}
  \E_{C_{j+1}} \theta Z_j
  =
 \sum_{U\in\cP_{j+1}}
 \tilde I^{\Lambda\setminus U}
    K'(U)
\end{equation}
with
\begin{equation}
   K'(U) = \sum_{(\cX,\check{X})\in \cG(U)}
  \tilde I^{U\setminus (\check X \cup X_{\cX})}
      \E_{C_{j+1}} \qa{ \check{K}(\check{X})
  \prod_{(B,Y) \in \cX} \theta J(B,Y) },
\end{equation}
where we make the definition that $\cG(U)$ consists of
$(\cX,\check{X})$ such that $\check{X}\in \cP_{j}$, $\cX$ satisfies
(i) and (ii) above, $X_{\cX}$ does not touch $\check{X}$, and
$\overline{\check X \cup[\cup_{(B,Y)\in\cX}B^\square]}=U$. The sum
over $X$ in \eqref{e:EcthetaZ1} has been incorporated into the sum
over $(\check{X},\cX)$. 

\smallskip
\noindent
\emph{Conclusion.}
We now specialise to the setting of Proposition~\ref{prop:ZVK}.
Thus we take $J$ as in~\eqref{e:J-defB} and~\eqref{e:J-def}, and
$\tilde I(B)= e^{-(u+V)_{j+1}(B)}$ with the exponent as
defined in \eqref{e:V+-def}, and  $K_{j+1}(U)$ as defined in \eqref{e:K+-def}.
The arguments above show that 
\begin{equation}
  \E_{C_{j+1}} \theta Z_j=e^{-u_{j+1}|\Lambda|}\sum_{U \in \cP_{j+1}} e^{-V_{j+1}(\Lambda\setminus U)} K_{j+1}(U).
\end{equation}
What remains is to prove the claims regarding factorisation and
automorphism invariance. For factorisation, note that
$\sum_{\cG(U_{1}\cup U_{2})} = \sum_{\cG(U_{1})}\sum_{\cG(U_{2})}$ for
$U_{1},U_{2}\in \cP_{j+1}$ that do not touch. Moreover, since $U_{1}$
and $U_{2}$ are separated by a distance at least
$L^{j+1}>\frac{1}{2}L^{j+1}+2^{d+1}L^{j}$, the expectations in the
definition of $K_{j+1}(U_{1}\cup U_{2})$ factor.  Here we have used
our standing assumption that $L>2^{d+2}$, that $J(B,Y)=0$ if
$Y\notin \cS_{j}$, and that the range of $C_{j+1}$ is
$\frac{1}{2}L^{j+1}$.
Automorphism invariance follows from the formula for $K_{j+1}$ and the
properties of $\E_{C_{j+1}}\theta$ discussed in
Section~\ref{sec:symmetries}.
\end{proof}

Proposition~\ref{prop:ZVK} implies in particular that if
$K_{j}$ has the factorisation property~\eqref{e:K-factor}, then we can identify
$(K_{j+1}(X))_{X\in\cP_{j+1}(\Lambda_{N})}$ with its restriction to connected polymers
$(K_{j+1}(X))_{X\in\cC_{j+1}(\Lambda_{N})}$.
Moreover, if the $K_{j}$ are also automorphism invariant, then $K_{j+1}\in\cKbulk_{j+1}(\Lambda_{N})$.

By construction and the consistency of the covariances $C_j$ with
$j<N$ for different values of $N$, the maps defined for different
$\Lambda_N$ are also consistent in the following sense:

\begin{proposition} \label{prop:consistency} For $j+1<N$ and
  $U \in \cP_{j+1}(\Lambda_N)$, $V_{j+1}(U)$ and $K_{j+1}(U)$ above
  depend on $(V_j,K_j)$ only through $V_j(X),K_j(X)$ with
  $X \in \cP_j(U^\square)$.  Moreover, for
  $U \in \cP_{j+1}(\Lambda_N) \cap \cP_{j+1}(\Lambda_{M})$ with the
  natural local identification of $\Lambda_N$ and $\Lambda_M$, the map
  $(V_j,K_j) \mapsto (V_{j+1}(U),K_{j+1}(U))$ is independent of $N$
  and $M$.
\end{proposition}

Temporarily indicating the $N$-dependence of $\Phi_{j+1}=\Phi_{j+1,N}$
explicitly, consistency implies the existence of an infinite volume
limit $\Phi_{j+1,\infty}=\lim_{N\to\infty} \Phi_{j+1,N}$ defined for
arguments $V_j \in \cVbulk$ and
$K_j =(K_j(X))_{X\in\cC_j(\Z^d)} \in \cKbulk_{j}(\Z^{d})$.
Explicitly, if we write
$\Phi_{j+1,N}(V_{j},K_{j})= (V^{N}_{j+1},K^{N}_{j+1})$ and omit the
$N$ for the infinite volume map,
$K_{j+1}(U) = \lim_{N\to\infty} K^{N}_{j+1}(U)$, and similarly for
$V_{j+1}$. The limits exist as the sequences are constant after
finitely many terms.  This infinite volume limit does not carry the
full information from the $\Phi_{j+1,N}$ because terms indexed by
polymers that wrap around the torus are lost, but it does carry
complete information about small sets at all scales and thus about the
flow of $V_j$. As for the finite-volume maps, the infinite
volume limit carries a mild dependence on $m^{2}$. We typically omit this
from the notation.

\subsection{Estimates for the renormalisation group map}
\label{sec:step}

The renormalisation group map $\Phi_{j+1}= \Phi_{j+1,N}$ is a function
of $(V,K)\in\cVbulk \oplus \cKbulk_j(\Lambda_N)$.
The size of $V$ and $K$ will be measured in the norms
\begin{align}
  \label{e:V-norm}
  \|V\|_j &= \sup_{B\in \cB_j} \|V(B)\|_{T_j(\ell_j)}
  \\
  \label{e:K-norm}
  \|K\|_j &= \sup_{X\in \cC_j} A^{(|\mathcal{B}_j(X)|-2^d)_+} \|K(X)\|_{T_j(\ell_j)}
\end{align}
where $A>1$ is a parameter that will be chosen sufficiently large.
The space of bulk coupling constants $\cVbulk \cong \C^4$ has
finite dimension.  The space for polymer activities
$\cKbulk_j(\Lambda_N)$ is also finite-dimensional for $N<\infty$
since an element $K_j \in \cKbulk_j(\Lambda_N)$ is a finite
collection of elements $K_j(X)$ of the finite-dimensional Grassmann
algebra $\cN(\Lambda_N)$.  Thus
$\cVbulk \oplus \cKbulk_j(\Lambda_N)$ is a finite-dimensional complex
normed vector space with the above norms, and therefore a Banach space.

\begin{theorem} \label{thm:step} Let $d \geq 3$, $L \geq L_0(d)$, and
  $A \geq A_0(L,d)$. Assume that $u_j=0$.  There exists
  $\epsilon = \epsilon(L,A)>0$ such that if $j+1<N$ and 
  $ \|V_j\|_j + \|K_j\|_j \leq \epsilon$ then
  \begin{align}
   \label{e:step-V}
     \|u_{j+1} + V_{j+1} - \E_{C_{j+1}}\theta V_j\|_{j+1} &\leq  O(L^d\|K_j\|_j)
      \\
      \label{e:step-K}
      \|K_{j+1}\|_{j+1} &\leq
                          O(L^{-\dplus}
                          +A^{-\eta})\|K_j\|_j
                          +
                          O(A^{\nu})(  \|V_j\|_j^2+\|K_j\|_j^2 ),
  \end{align}
  where $\eta=\eta(d)$ and $\nu=\nu(d)$ are positive geometric constants.
  The maps $\Phi_{j+1}$ are entire in $(V_j,K_j)$ 
  and hence all derivatives of any order are uniformly bounded on
  $\|V_j\|_j+\|K_j\|_j \leq \epsilon$. Moreover, the maps $\Phi_{j+1}$
  are continuous in $m^2 \geq 0$. 
  
  The last renormalisation group map $\Phi_{N}$
  satisfies the same bound with {$L^{-\dplus}$}  replaced by $1$. 
\end{theorem}

Theorem~\ref{thm:step} is the
analogue of \cite{MR3332941,MR3332942} for the four-dimensional weakly
self-avoiding walk, but much simpler since (i) we are only
working with fermionic variables, and (ii) we are above the lower
critical dimension (two for our model).  The factors $L^d$ and $A^\nu$ in the error
bounds are harmless.
On the other hand, it is essential that
$O(L^{-\dplus}+A^{-\eta}) < 1$ for $L$ and $A$ large: this estimate
establishes that $K$ is \emph{irrelevant} (contracting) in renormalisation
group terminology.

The remainder of this subsection proves Theorem~\ref{thm:step}.
Readers not familiar with the use of the renormalisation group might want to
skip the somewhat technical proof on a first reading and proceed to Sections~\ref{sec:flow}
and~\ref{sec:suscept} to get an idea of how these estimates are used.

\medskip

\emph{Throughout the rest of Section~\ref{sec:step} the hypotheses of
  Theorem~\ref{thm:step} will be assumed to hold.}
\medskip

The substantive claims of Theorem~\ref{thm:step} are the estimates
\eqref{e:step-V} and \eqref{e:step-K}: these quickly yield the claims
regarding derivatives by a standard Cauchy estimate, as we now explain. 
Recall that given two Banach spaces $X$ and $Y$ and a domain
$D\subset \C$ we say that a function $g\colon D\rightarrow X$ is
analytic if it satisfies the Cauchy-Riemann equation
$\partial_{\bar{z}} g=0$.  For an open set $O\subseteq X$, we then say
that a function $F\colon O \rightarrow Y$ is analytic if $F\circ g$ is
analytic for every analytic function $g\colon D\rightarrow X$. After
(possibly) adding some additional coordinates to ensure all necessary
monomials are in the domain, the maps
$(V_j,K_j) \mapsto (V_{j+1},K_{j+1})$ are multivariate polynomials,
and the norm estimates~\eqref{e:step-V} and~\eqref{e:step-K} extend to
this larger space.  Being multivariate polynomials, the $\Phi_{j+1}$
are analytic functions.

We use analyticity and the Cauchy integral formula to extract
derivatives.
Namely, if $(V, K)$ and $(\dot{V}^{(i)}, \dot{K}^{(i)})_{i=1}^n$ are
collections of polymer coordinates at scale $j$ satisfying
$\|V\|_j+\|K\|_j\leq \epsilon/2$ and
$\|\dot V^{(i)}\|_j+\|\dot K^{(i)}\|_j \leq 1$, then
\begin{equation}
  D^{n} \Phi_{j+1}|_{(V, K)} (\dot{V}^{(i)}, \dot{K}^{(i)})_{i=1}^{n}
  =\oint \cdots \oint \prod_{i=1}^{n} \frac{ \textrm{d} w_i}{w_i^2} \Phi_{j+1}(V+ \sum_{i=1}^{n} w_i \dot{V}^{(i)}, K+\sum_{i=1}^{n} w_i \dot{K}^{(i)} )
\end{equation}
where the $n$-tuple of contours are circles around $0$ with radius $\epsilon/(2 n)$.
The statement of Theorem~\ref{thm:step} regarding boundedness of derivatives follows.

The asserted continuity in $m^{2} \geq 0$ follows from the explicit
formulas for $(V_{j+1},K_{j+1})$, that $\Loc$ is linear, and that the
covariances $C_{j}$ are continuous in $m^{2} \geq 0$.

\subsubsection{Coupling constants}
\label{sec:K-coupling}

We begin with the bound \eqref{e:step-V} for $V_{j+1}$.  The first
term on the right-hand side in the definition \eqref{e:V+-def} of
$u_{j+1}+V_{j+1}$ produces the expectation term in \eqref{e:step-V}.
For $B\in \cB_j$, the remainder in \eqref{e:V+-def} is bounded as
follows:
\begin{align}
\label{E:Q}
  \|Q(B)\|_{T_{j}(\ell_{j})}
  &\leq
    \sum_{X\in \cS_j: X \supset B}  \|\Loc_{X,B} K_j(X)\|_{T_{j}(\ell_{j})}
    \nnb
    &\leq O(1)
      \sup_{B,X}\|\Loc_{X,B} K_j(X)\|_{T_{j}(\ell_{j})}
    \leq O(1)
    \|K_j\|_j
\end{align}
where we have used that the number of small sets containing a fixed
block is $O(1)$ in the second step, and \eqref{e:Loc-bd} in
the third.  Since each block in $\cB_{j+1}$ contains $L^d$
blocks in $\cB_j$, by using \eqref{e:normcontract} to bound
the expectation and change of scale in the norm, the first claim
\eqref{e:step-V} follows.

\smallskip The remainder of Section~\ref{sec:step}
  establishes the bound~\eqref{e:step-K} for $K_{j+1}$.  The following
  basic observations will be used repeatedly. Note that by
\eqref{e:normcontract} the main term contributing to
$u_{j+1}|B|+V_{j+1}(B)$ is bounded by, for $B\in \cB_j$,
\begin{equation}
  \|\E_{C_{j+1}}\theta V_j(B)\|_{T_{j+1}(\ell_{j+1})} \leq \|V_j(B)\|_{T_j(\ell_j)}.
\end{equation}
By combining this with \eqref{E:Q} we have that, for $B\in\cB_{j}$,
\begin{equation}
\label{E:vj}
u_{j+1}|B| \leq \|V_j\|_j + O( \|K_j\|_j),
\qquad
\|V_{j+1}(B)\|_{T_{j+1}(\ell_{j+1})} \leq \|V_j\|_j + O( \|K_j\|_j).
\end{equation}

\subsubsection{Preparation for bound of the non-perturbative coordinate}

To derive \eqref{e:step-K},
we first separate from $K_{j+1}(U)$ a leading contribution.
This contribution is: 
\begin{multline} \label{e:K-linear}
  \cL_{j+1}(U)
  \bydef \sum_{X\in \cC_j: \bar X=U} e^{-V_{j+1}(U\setminus X)} e^{u_{j+1}|X|} \E_{C_{j+1}}\qa{\theta K_j(X) 
    - \sum_{B\in \cB_j} \theta J(B,X)}\\
     +\sum_{X\in \cP_j: \bar X=U} e^{-V_{j+1}(U\setminus X)} e^{u_{j+1}|X|} \E_{C_{j+1}} \qB{ (\delta I)^X }.
\end{multline}
Note that while the first sum on the right-hand side is over connected
polymers, the second is over all polymers.  This expression includes
the contributions to $K_{j+1}$ explicitly linear in $K_{j}$,
and 
all other terms in the definition of $K_{j+1}$ are higher order, see
Section~\ref{sec:K-nonlinear} below.

We may divide each of the sums on the right-hand side in \eqref{e:K-linear} into the contributions
from small sets $X \in \cS_j$ and large sets
$X \in \cP_j\setminus \cS_j$. We recall that small sets are, by definition, connected. 
These restricted sums will be denoted by $\cL_{j+1, \cS}(U)$
and $\cL_{j+1, \cP\setminus \cS }(U)$ respectively:
\begin{equation}\label{e:K-linear-split}
  \cL_{j+1}(U) = \cL_{j+1,\cS}(U) +   \cL_{j+1,\cP\setminus \cS}(U).
\end{equation}

Large sets 
are easier to handle because they lose combinatorial entropy under
change of scale (reblocking), i.e.,
$|\cB_{j}(X)|$ will be significantly larger than $|\cB_{j+1}(\bar{X})|$. In renormalisation group
terminology, they are \emph{irrelevant}.
Small sets, on the other hand, require careful treatment.

\subsubsection{Small sets}
\label{sec:K-smallset}

The main estimate on small sets is summarised as follows.

\begin{proposition}
  \label{prop:SmallSetLin}
  The contribution $\cL_{j+1,\cS}$ to \eqref{e:K-linear} satisfies
  \begin{equation}
  \label{e:SmallSetLin}
  \|\cL_{j+1, \cS}\|_{j+1} =
  O(L^{-\dplus}\|K_j\|_j+L^{d}(\|V_j\|_j^2 + \|K_j\|_j^2)).
\end{equation}
\end{proposition}

In order to bound
\begin{equation} \label{e:K-linear-smallset}
    \cL_{j+1,\cS}(U)
  \bydef \sum_{X\in \cS_j: \bar X=U} e^{-V_{j+1}(U\setminus X)} e^{u_{j+1}|X|} \E_{C_{j+1}}\qa{\theta K_j(X) 
    - \sum_{B\in \cB_j} \theta J(B,X)
    + (\delta I)^X },
\end{equation}
we consider the terms $X \in \cS_j \setminus \cB_j$ and $X\in \cB_j$
in the outer sum separately.  By the definition of $J$ in
\eqref{e:J-def}, for any $X \in \cS_j \setminus \cB_j$,
\begin{equation}
  \sum_{B \in \cB_j(X)}
  \E_{C_{j+1}} \qB{ \theta  J(B,X)} 
  =
  \sum_{B \in \cB_j(X)} 
  \E_{C_{j+1}} \qB{ \theta \Loc_{X,B}K_j(X) }
  =
    \E_{C_{j+1}} \qB{ \theta  \Loc_{X}  K_j(X) },
\end{equation}
where the final equality follows from
Proposition~\ref{prop:Loc}~(iii). Thus the contribution to
$\cL_{j+1, \cS}$ from $X \in \cS_j\setminus \cB_j$ is
\begin{equation} \label{e:Ksmallset}
  \E_{C_{j+1}} \qB{ \theta (1-\Loc_X)  K_{j}(X)}
  + \E_{C_{j+1}}\qB{ (\delta I)^X} .
\end{equation}
The contribution to $\cL_{j+1, \cS}$ 
from $X= B \in \cB_j$ is
\begin{align} \label{e:Kblock}
    \E_{C_{j+1}}\qB{\theta K_j(B) 
      + 
      \delta I(B) 
      - 
      \theta J(B,B)}
  &=\E_{C_{j+1}} \qB{ \theta (1-\Loc_B)K_j(B)}
    \nnb
    &\qquad
  + \E_{C_{j+1}} \qB{\delta I(B)  + \theta Q(B)}.
\end{align}

We now give a series of estimates, bounding the right-hand sides of \eqref{e:Ksmallset} and \eqref{e:Kblock}
and the term outside the expectation in \eqref{e:K-linear-smallset},
and then assemble them into a proof for Proposition~\ref{prop:SmallSetLin}.

\begin{lemma} \label{lem:smallset-K}
  For any $U \in \cC_{j+1}$,
  \begin{equation}
    \sum_{X \in \cS_j: \bar X=U} \normB{
      \E_{C_{j+1}}\qB{\theta(1-\Loc_X) K_j(X)}
    }_{T_{j+1}(\ell_{j+1})} = 
    O(L^{-\dplus}) \|K_j\|_j.
  \end{equation}
\end{lemma}
\begin{proof}
  Note that $\bar X \in \cS_{j+1}$ if $X \in \cS_j$, so it suffices to
  prove the lemma for $U\in \cS_{j+1}$.  Now for any
  $U \in \cS_{j+1}$, since there are $O(L^d)$
  small sets $X\in \cS_j$ such that $\bar X=U$ we get 
  \begin{align} \label{e:Kpf-contract}
    &\sum_{X \in \cS_j: \bar X=U} \normB{\E_{C_{j+1}}\qB{\theta(1-\Loc_X)  K_j(X)}}_{T_{j+1}(\ell_{j+1})}
      \nnb
  &\leq O(L^d)   \sup_{X \in \cS_j} \normB{\E_{C_{j+1}}\qB{\theta(1-\Loc_X) K_j(X)}}_{T_{j+1}(\ell_{j+1})}
    \nnb
  &\leq O(L^d)   \sup_{X \in \cS_j} \|(1-\Loc_X) K_j(X)\|_{T_{j+1}(2\ell_{j+1})}
    \nnb
  &\leq O(L^d)  O(L^{-d}) (L^{-\dplus}
    ) \sup_{X \in \cS_j} \|K_j(X)\|_{T_{j}(\ell_{j})}
    \nnb
    &\leq
    O( L^{-\dplus}) \|K_j\|_j
\end{align}
where we have used the contraction estimate \eqref{e:E-contract} for the expectation
in the second step and
the contraction estimate \eqref{e:Loc-contract} for $\Loc_X$ in the third step.
\end{proof}

\begin{lemma} \label{lem:smallset-V}
  For  $B \in \cB_j$, 
  \begin{equation} 
  \label{e:smallset-V-2}
  \normB{\E_{C_{j+1}}\qB{\delta I(B) + \theta Q(B)}
  }_{T_{j+1}(\ell_{j+1})} = O(\|V_j\|_j^2 + \|K_j\|_j^2),
  \end{equation}
\end{lemma}

\begin{proof}
By the definition of $(u+V)_{j+1}$ in \eqref{e:V+-def} we have
\begin{align} \label{e:Kblock2}
  \E_{C_{j+1}}\qB{\delta I(B)+\theta Q(B)}
  &=  \E_{C_{j+1}}\qB{\theta e^{-V_j(B)}-1 +  \theta V_j(B)}
    \nnb
    &\qquad - \qB{e^{-(u+V)_{j+1}(B)}-1+ (u+V)_{j+1}(B)}.
\end{align}
By the product property \eqref{e:T-product}, if for some $V$ and some $k$ we have $\|V(B)\|_{T_k(\ell_k)} \leq 1$, then 
\begin{equation}
  \|e^{-V(B)}-1+V(B)\|_{T_k(\ell_k)} \leq O(\|V(B)\|_{T_k(\ell_k)}^2).
\end{equation}
Recall that $\E_{C_{j+1}}\theta$ is contractive as a map from
$T_{j}(\ell_{j})$ to $T_{j+1}(\ell_{j+1})$ by \eqref{e:normcontract}. 
Applying these estimates to the $T_{j+1}(\ell_{j+1})$ norm of
\eqref{e:Kblock2} and using \eqref{E:vj} gives the bound
\eqref{e:smallset-V-2}.
\end{proof}

\begin{lemma}\label{lem:smallset-V1}
  For $X \in \cP_j$,
  \begin{equation}
        \label{e:smallset-V-1}
        \normB{\E_{C_{j+1}} \qB{(\delta I)^X}}_{T_{j+1}(\ell_{j+1})}
        = (O(\|V_j\|_j+\|K_j\|_j))^{|\cB_j(X)|}.
  \end{equation}
\end{lemma}

\begin{proof}
  Using that $\E_{C_{j+1}}$ satisfies the contraction estimate \eqref{e:E-contract2},
  it suffices to show
  \begin{equation} \label{e:deltaIbd}
    \|(\delta I)^X\|_{T_{j+1}(\ell_{j+1})} 
    = (O(\|V_j\|_j+\|K_j\|_j))^{|\cB_j(X)|}.
  \end{equation}
  By the product property \eqref{e:T-product} it suffices to prove this estimate for a
  single block. In this case,
  \begin{align}
    \|\delta I(B)\|_{T_{j+1}(\ell_{j+1})} 
    &\leq
    \|\theta(e^{-V_j(B)}-1)\|_{T_{j+1}(\ell_{j+1})} 
    +
      \|e^{-(u+V)_{j+1}(B)}-1\|_{T_{j+1}(\ell_{j+1})}
      \nnb
    &\leq
            O(\|V_j(B)\|_{T_{j+1}(2\ell_{j+1})})
    +
    O(\|(u+V)_{j+1}(B)\|_{T_{j+1}(\ell_{j+1})})
  \end{align}
  by the product property \eqref{e:T-product} of the norms 
  and \eqref{e:theta-bd}.
  Using $2\ell_{j+1} \leq \ell_j$ and \eqref{e:T-monotone} for the first term
  and \eqref{E:vj} for the second term bounds 
  the right-hand side by $O(\|V_j\|_j + \|K_j\|_j)$ as needed.
\end{proof}

We need one further general estimate.
\begin{lemma}
    \label{lem:pert12}
    If $\|K_j\|_j + \|V_j\|_j \leq\epsilon=\epsilon(d,L)$ is
    sufficiently small, then if $\bar X = U \in \cP_{j+1}$,
    \begin{equation}
      \label{e:pert12}
      \|e^{-V_{j+1}(U\setminus X)+u_{j+1}|X|}\|_{T_{j+1}(\ell_{j+1})}
      \leq 2^{|\cB_{j}(X)|}.
    \end{equation}
  \end{lemma}
  \begin{proof}
    By the product property \eqref{e:T-product}  and \eqref{E:vj} to bound $V_{j+1}$ and $u_{j+1}$,
  \begin{equation}
     \|e^{-V_{j+1}(U\setminus X)+u_{j+1}|X|}\|_{T_{j+1}(\ell_{j+1})} \leq (1+{ O(}\epsilon))^{|\cB_{j}(U)|}, 
  \end{equation}
  and $|\cB_{j}(U)|$ is at most $L^{d}|\cB_{j+1}(U)|\leq
  L^{d}|\cB_{j}(X)|$. The claim follows provided $(1+{ O(}\epsilon))^{L^{d}}\leq 2$.
\end{proof}

\begin{proof}[Proof of Proposition~\ref{prop:SmallSetLin}] To
    estimate the summands of $\cL_{j+1, \cS}(U)$, we use the product
    property of the $\|\cdot\|_{T_{j+1}(\ell)}$ norm to combine
    Lemma~\ref{lem:pert12} with Lemma~\ref{lem:smallset-K},
    Lemma~\ref{lem:smallset-V} for $X\in \cB_j$, and with
    Lemma~\ref{lem:smallset-V1} for $X\in \cS_{j}\backslash \cB_j$.
    For the sum of the terms $(\delta I)^{X}$ we use that
    $(1+\|V_{j}\|_{j}^{2} + \|K_{j}\|_{j}^{2})^{L^{d}}\leq 2$ provided
    $\epsilon=\epsilon(L)$ is small enough.  Altogether, we obtain
\begin{equation}
\| \cL_{j+1, \cS}(U)\|_{T_{j+1}(\ell_{j+1})}= O(L^{-\dplus} \|K_j\|_j+L^{d}(\|V_j\|_j^2 + \|K_j\|_j^2)),
\end{equation}
which proves the lemma.
\end{proof}

\subsubsection{Large sets}
\label{sec:K-largeset}

Next we consider the contribution to \eqref{e:K-linear} from the terms
$X \not\in \cS_j$ in the sums, i.e., $\cL_{j+1,\cP\setminus\cS}$.
The main estimate of this section is summarised in the following proposition.

\begin{proposition}
  \label{prop:LargeSetLin}
  The contribution $\cL_{j+1,\cP\setminus\cS}$ to \eqref{e:K-linear} satisfies
  \begin{equation}
  \label{e:LargeSetLin}
  \|\cL_{j+1, \cP \setminus \cS}\|_{j+1} =
  O(  A^{-\eta}\|K_j\|_j+A^\nu[\|V_j\|_j + \|K_j\|_j]^{2})
\end{equation}
\end{proposition}

We begin by recording a combinatorial fact, see \cite[Lemmas~6.15 and 6.16]{MR2523458}
or \cite[Lemma~C.3]{MR3332942} for details on its
proof. For the statement, recall that if $X\in\cP_{j}$, then
its closure $\bar X\in \cP_{j+1}$ denotes the smallest 
next-scale polymer containing $X$.

\begin{lemma}
  \label{lem:comb1}
  Let $L \geq 2^d+1$. 
  There is a geometric constant $\eta=\eta(d)>0$ depending only on $d$ such  that
  for all $X \in \cC_j \setminus \cS_j$,
  \begin{equation} \label{e:largeset-contract} |\cB_{j}(X)| \geq
    (1+2\eta) |\cB_{j+1}(\bar X)|.
  \end{equation}
  Moreover, for all $X \in \cP_j$, $|\cB_j(X)| \geq |\cB_{j+1}(X)|$ and
  \begin{equation} \label{e:largeset-contract-comp}
    |\cB_j(X)| \geq (1+\eta)|\cB_{j+1}(\bar X)|-(1+\eta)2^{d+1}|\Comp(X)|.
  \end{equation}
\end{lemma}

We also record an application of this estimate to sums indexed by large polymers which will be used in this and in the next section.
By \eqref{e:largeset-contract}, if $A=A(L)$ is large enough,
\begin{equation}
  \label{e:Achoice}
  A^{|\cB_{j+1}(U)|} \sum_{X \in \cC_j\setminus \cS_{j}: \bar X=U} (A/2)^{-|\cB_j(X)|}
  \leq (2^{L^d} 2^{1+2\eta}
  A^{-2\eta})^{|\cB_{j+1}(U)|} \leq A^{-\eta|\cB_{j+1}(U)|}, 
\end{equation}
as the set of $X\in \cP_j$ 
with $\bar{X}=U$ has
size at most $2^{L^{d}|\cB_{j+1}(U)|}$.
Similarly, by \eqref{e:largeset-contract-comp}, if $\alpha \geq
A^{(1+\eta)2^{d +1}}$,
\begin{equation}
  \label{e:Achoice-conn}
  A^{|\cB_{j+1}(U)|} \sum_{X \in \cP_j: \bar X=U} (A/2)^{-|\cB_j(X)|} \alpha^{-|\Comp(X)|}
  \leq A^{-(\eta/2)|\cB_{j+1}(U)|}
  .
\end{equation}

\begin{proof}[Proof of Proposition~\ref{prop:LargeSetLin}]
From \eqref{e:K-linear} and \eqref{e:K-linear-split},
recall that (as $J(B,X)=0$ for large $X$)
\begin{align}\label{e:K-linear-large}
  \cL_{j+1, \cP\setminus \cS}(U)
  &= \sum_{X\in \cC_j \setminus \cS_j: \bar X=U} e^{-V_{j+1}(U\setminus X)+u_{j+1}|X|} \E_{C_{j+1}}\qB{ \theta K_j(X)}
  \\
  &\qquad + \sum_{X\in \cP_j \setminus \cS_j: \bar X=U} e^{-V_{j+1}(U\setminus X)+u_{j+1}|X|} \E_{C_{j+1}}\qB{ (\delta I)^X }.\nonumber
\end{align}

We first consider the case $U\in \cC_{j+1}\setminus \cS_{j+1}$, and proceed as follows:
for $\|K_j\|_j + \|V_j\|_j \leq\epsilon$ with $\epsilon$
sufficiently small, by Lemma~\ref{lem:pert12} the $j+1$
norm of the $K$ contribution to \eqref{e:K-linear-large} is bounded by
\begin{equation} \label{e:largeset2}
  A^{|\cB_{j+1}(U)|-2^d}
  \sum_{X\in \cC_j \setminus \cS_j: \bar X=U}  2^{|\cB_j(X)|} \normB{\E_{C_{j+1}}\qB{\theta K_j(X)}}_{T_{j+1}(\ell_{j+1})}.
\end{equation}
By the definition of $\|K_{j}\|_{j}$ and
noting that $(|\cB_{j}(X)|-2^{d})_{+}=|\cB_{j}(X)|-2^{d}$ since $X\not\in\cS_{j}$,
\begin{equation}
  \label{E:largeset-inter}
  \normB{ \E_{C_{j+1}}\qB{\theta K_j(X)}}_{T_{j+1}(\ell_{j+1})}\leq A^{-(|\cB_{j}(X)|-2^{d})}\|K_j\|_j,
\end{equation}
where we have also used the contraction estimates
\eqref{e:E-contract} and \eqref{e:T-monotone}.  Inserting this bound into
\eqref{e:largeset2} and using \eqref{e:Achoice} gives that the $K$
contribution to \eqref{e:K-linear-large} is bounded by
\begin{equation} \label{e:largeset3}
   A^{|\cB_{j+1}(U)|}
    \sum_{X\in \cC_j \setminus \cS_j: \bar X=U} (A/2)^{-|\cB_j(X)|}\|K_j\|_{j} \leq A^{-\eta} \|K_j\|_{j}.
\end{equation} 
This is the desired bound for the first term in \eqref{e:K-linear-large}.

To bound the $j+1$ norm of the $\delta I$ contribution to \eqref{e:K-linear-large},
Lemmas~\ref{lem:smallset-V1} and~\ref{lem:pert12} and the product
property yield (provided $\epsilon$ is sufficiently small depending on $L$)
\begin{align}
 \label{e:largeset4}
& A^{|\cB_{j+1}(U)|-2^d}  \normbb{  \sum_{X\in {\cP_j} \setminus \cS_j: \bar X=U} e^{-V_{j+1}(U\setminus X)+u_{j+1}|X|}  \E_{C_{j+1}}\qB{ (\delta I)^X} }_{T_{j+1}(\ell_{j+1})}\nnb
&\leq  A^{|\cB_{j+1}(U)|-2^d}
    \sum_{X\in {\cP_j }\setminus \cS_j: \bar X=U} \qB{ 2O(\|V_j\|_j + \|K_j\|_j) }^{|\cB_j(X)|}   .
\end{align}
Since $U \in \cC_{j+1} \setminus \cS_{j+1}$ and $\bar X=U$, each $X$
in the last sum must have $|\cB_{j}(X)|\geq 2^{d}+1$.  If
$\|V_j\|_j + \|K_j\|_j<\epsilon$ and $\epsilon$ is sufficiently small
(depending on $A$), then the quantity in brackets is less than
$1/A^{2+2(1+\eta)2^{d+1}}$.  By the elementary inequality
$(c^{2})^{n-2}\leq c^{n}$ for $c\in (0,1)$, $n>4$ and using that
$|\cB_{j}(X)|\geq 2^{d}+1>4$ for each summand, we obtain the upper
bound
\begin{equation}
  [O(\|V_j\|_j + \|K_j\|_j)]^{2} A^{|\cB_{j+1}(U)|} \sum_{X\in {\cP_j}
    \setminus \cS_j: \bar X=U} A^{-|\cB_{j}(X)|} A^{-(1+\eta)2^{d+1}|\cB_j(X)|}. 
\end{equation}
Using \eqref{e:Achoice-conn}, it follows that the
$\delta I$ contribution to~\eqref{e:K-linear-large} is bounded by
\begin{equation}
  O(A^{-\eta/2}[\|V_j\|_j + \|K_j\|_j]^{2})
  =  O([\|V_j\|_j + \|K_j\|_j]^{2})
  ,
\end{equation}
for $A$ sufficiently large.
We have now completed the bound on \eqref{e:K-linear} provided $U\in \mathcal C_{j+1}\backslash \mathcal S_{j+1}$.

The argument is similar if $U \in \cS_{j+1}$.  In this case the
prefactor $A^{|\cB_{j+1}(U)|-2^d}$ gets replaced by $1$ in
\eqref{e:largeset2} and \eqref{e:largeset4}.  For the $K$
contribution, in place of \eqref{e:largeset3} we obtain, since
$1+ 2^d\leq |\cB_j(X)|\leq L^d |\cB_{j+1}(U)|\leq (2L)^d$ and the
number of summands in this case is at most $2^{(2L)^d}$,
\begin{align}
  \normbb{\sum_{X\in \cC_j \setminus \cS_j: \bar X=U}
   e^{-V_{j+1}(U\setminus X)+u_{j+1}|X|} \E_{C_{j+1}}\qB{\theta K_j(X)}
  }_{T_{j+1}(\ell_{j+1})}
  &\leq A^{-1}2^{2(2L)^d}\|K_j\|_j
    \nnb
  &=O(A^{-\eta} \|K_j\|_j)
\end{align}
for $A$ large enough depending on $L$ and $d$.  For the $\delta I$
contribution, in place of \eqref{e:largeset4} we have
\begin{equation}
   \normbb{ \sum_{X\in \cP_j \setminus \cS_j: \bar X=U}   e^{-V_{j+1}(U\setminus X)+u_{j+1}|X|} \E_{C_{j+1}}\qB{(\delta I)^X} }_{T_{j+1}(\ell_{j+1})}
  \leq  O\left(2^{2 {(2L)^{d}}} 
    [\|V_j\|_j + \|K_j\|_j]^{2}\right)
\end{equation}
since each summand on the left-hand side has $|\cB_j(X)|\geq  2$.

Thus for $A=A(L,d)$ sufficiently large and $\epsilon=\epsilon(A,L)$
sufficiently small, the expression \eqref{e:K-linear-large} is bounded
in the $T_{j+1}(\ell_{j+1})$ norm by $
O(A^{-\eta}\|K_j\|_j+A^\nu[\|V_j\|_j + \|K_j\|_j]^{2})$ in all cases.
\end{proof}
\subsubsection{Non-linear part}
\label{sec:K-nonlinear}

Finally, we consider the non-linear contribution $K_{j+1}-\cL_{j+1}$.
To conclude the proof of \eqref{e:step-K}, and hence the proof of
Theorem~\ref{thm:step}, we prove the following estimate.

\begin{proposition} \label{prop:K-linear-nonlinear}
\begin{equation}
\|K_{j+1}-\cL_{j+1}\|_{j+1}\leq A^\nu O(\|K_j\|(\|K_j\|_j+\|V_j\|_j)).
\end{equation}
\end{proposition}

Before diving into the proof, let us review the terms which remain to
be estimated. Recall the definition of $K_{j+1}(U)$ from
\eqref{e:K+-def} and the leading part $\cL_{j+1}(U)$ from
\eqref{e:K-linear}. Write $|\cX|$ for the number of pairs in $\cX$.
With respect to the indexing of summands for $K_{j+1}(U)$, the leading
part $\cL_{j+1}(U)$ results from the terms with $|\cX|=0$ and
$\check X=X$ by only keeping the terms in the formula for
$\check{K}(X)$ with either a single factor $\theta K_j(X)$ when
$X\in\cC_j$, a single factor $(\delta I)^X$ when $X\in \cP_j$, or a
single factor $\sum_B \theta J(B,X)$.  It follows that we can write
\begin{equation} \label{e:KminuscL}
  K_{j+1}(U)-\mathcal{L}_{j+1}(U)
  = \cR^1(U) + \cR^2(U) + \cR^3(U)
    ,
\end{equation}
where
\begin{gather}
  \cR^1(U)
  = e^{u_{j+1}|U|} \sum_{\cG_1(U)} e^{-(u+V)_{j+1}(U\setminus X_{\cX})} \E_{C_{j+1}}
    \qbb{
    \prod_{(B,X) \in \cX} \theta J(B,X) },
  \\
  \cR^2(U)
  =e^{u_{j+1}|U|} \sum_{\cG_2(U)}
  e^{-(u+V)_{j+1}(U \setminus \check X \cup X_{\cX})} \E_{C_{j+1}}
  \qbb{
  (\check{K}(\check{X})-(\delta I)^{\check X}\1_{|\cX|=0})
  \prod_{(B,X)\in\cX}\theta J(B,X)},
\end{gather}
and
\begin{equation}
  \cR^3(U)
  =e^{u_{j+1}|U|} \sum_{\cG_3(U)}
  e^{-(u+V)_{j+1}(U \setminus \check X)} \E_{C_{j+1}}
    \qbb{\check{K}(\check{X}) - \theta K(\check{X}) -(\delta I)^{\check X} + \sum_{B \in \cB_j} \theta J(B,\check{X})} ,
\end{equation}
when the subsets $\cG_i(U)\subset \cG(U)$ are defined as follows.  The
set $\cG_1(U)$ is defined by imposing the conditions $|\cX|=1$ and
$\check X = \varnothing$.  The set $\cG_2(U)$ is defined to
  consist of $(\cX, \check{X})$ such that $X_\cX \cup \check{X}$ has
  at least two components.  In particular, if $\cX=\varnothing$,
  $\check{X}$ has least two components and if $\check{X}=\varnothing$
  then $|\cX|\geq 2$.  Finally, $\cG_3(U)$ is defined by the
  conditions $|\cX|=0$ and $\check{X}\in \cC_j$.

The next lemma clearly implies Proposition~\ref{prop:K-linear-nonlinear}.
\begin{lemma} \label{lem:Ri-bd}
For $i\in \{1, 2, 3\}$, 
\begin{equation}
  \|\cR^i\|_{{j+1}}   \leq 
  A^\nu O(\|K_j\|(\|K_j\|_j+\|V_j\|_j)).
\end{equation}
\end{lemma}

\begin{proof}[Proof of Lemma~\ref{lem:Ri-bd} for $i=1$]
We begin by bounding $\cR^1$.   This bound exploits that $\sum_{X} J(B,X)=0$ for every $B\in\cB_j$,
  see \eqref{e:J-defB}--\eqref{e:J-def} and \eqref{e:J-zero}.
  As $\cX$ is a single pair $\{ (B,X)\}$, $X_{\cX}=X$. Since
  $\check{X}=\varnothing$  we can write  
  \begin{equation}
    \cR^1(U) = e^{u_{j+1}|U|}\sum_{B\in \cB_j}\sum_{X_\cX \in \cS_j}e^{-(u+V)_{j+1}(U\setminus X_{\cX})} \E_{C_{j+1}}\qB{\theta J(B,X_\cX)} \1_{\overline{B^\square}=U }  .
  \end{equation}
  Since $\sum_{X_{\cX}} J(B,X_{\cX})=0$ for $B\in\cB_j$, we can
  rewrite
\begin{equation} \label{e:K-crucial}
  \cR^1(U) = e^{u_{j+1}|U|}
  \sum_{B\in \cB_j}\sum_{X_\cX \in \cS_j}
  e^{-(u+V)_{j+1}(U\setminus  X_{\cX})}
  (1-e^{(u+V)_{j+1}(X_{\cX})}) \E_{C_{j+1}}\qB{\theta J(B,X_\cX)} \1_{\overline{B^\square}=U }.
\end{equation}
Since $X_{\cX} \in \cS_j$ we have
$\|1-e^{(u+V)_{j+1}(X_{\cX})}\|_{T_{j+1}(\ell_{j+1})} =
O(L^d(\|V_j\|_j+\|K_j\|_j))$ by \eqref{E:vj}.  Moreover,
\eqref{e:Loc-bd} implies
$\|J(B, X)\|_{T_{j}(\ell_{j})} =O(\|K_j\|_{j})$, so
$\|\E_{C_{j+1}}\theta J(B, X)\|_{T_{j+1}(\ell_{j+1})} =O(\|K_j\|_{j})$
since $\E_{C_{j+1}}\theta$ is a contraction.  Finally, exactly as in
the proof of Lemma~\ref{lem:pert12},
$\|e^{u_{j+1}(U)}e^{-(u+V)_{j+1}(U\setminus
  X_{\cX})}\|_{T_{j+1}(\ell_{j+1})}\leq
(1+O(\epsilon))^{|\cB_j(U)|}\leq 2$ for $\epsilon=\epsilon(L)$
  small enough, since $U$ is the closure of $B^\square$.  As there
are $O(L^{d})$ summands in~\eqref{e:K-crucial} we have shown
\begin{equation}
  \|\cR^1(U)\|_{T_{j+1}(\ell_{j+1})} \leq O(L^{2d}(\|V_j\|_j+\|K_j\|_j)\|K_j\|_j)
  \leq O(A^{\nu}(\|V_j\|_j+\|K_j\|_j)\|K_j\|_j).
\end{equation}
Since $A^{(|\cB_{j+1}(U)|-2^{d})_{+}}=1$ for any contributing $U$
(as $U$ is the closure of $B^{\square}$ for some block $B$), this concludes the
desired bound on $\cR^{1}(U)$.
\end{proof}

Before proceeding to the proof of Lemma~\ref{lem:Ri-bd} for $i=2$,
we first provide estimates on the norms of
$\check{K}(\check{X})$ and $\check{K}(\check{X})-(\delta I)^{\check{X}}$.

\begin{lemma} \label{lem:NonlinearCheckK}
  If $\|V_{j}\|_{j}+\|K_{j}\|_{j}\leq \epsilon$ and $\epsilon=\epsilon(A, L)$ is
  sufficiently small, then
  \begin{align}
    \label{e:NonlinearCheckK}
    \|\check{K}(\check X)\|_{T_{j+1}(\ell_{j+1})}
    &\leq [ A^{2^d}O( \|V_j\|_j +
  \|K_j\|_j)]^{|\Comp(\check X)|}
      (\frac{A}{2})^{-|\cB_j(\check{X})|},
  \\
    \label{e:NonlinearCheckK2}
    \|\check{K}(\check X)-(\delta I)^{\check X}\|_{T_{j+1}(\ell_{j+1})}
    &\leq
      [A^{2^d}O( \|V_j\|_j + \|K_j\|_j)]^{|\Comp(\check X)|-1}O(A^{2^d}\|K_j\|_j)
      (\frac{A}{2})^{-|\cB_j(\check{X})|}
      .
\end{align}
\end{lemma} The proof of these estimates is postponed until the conclusion of the main argument.

\begin{proof}[Proof of Lemma~\ref{lem:Ri-bd} for $i=2$]
If $\|V_j\|_j + \|K_j\|_j$ is small enough, arguing as in \eqref{e:pert12} implies
$\|e^{u_{j+1}(U)}e^{-(u+V)_{j+1}(U\setminus \check X \cup
  X_{\cX})}\|_{T_{j+1}(\ell_{j+1})}\leq 2^{|\cB_{j}(U)|}$, and by
\eqref{e:Loc-bd}, $\|J(B, X)\|_{T_{j}(\ell_{j})} =O(\|K_j\|_{j})$.
Thus using that $\E_{C_{j+1}}$ contracts from
$T_{j+1}(\ell_{j+1}\sqcup \ell_{j+1})$ into $T_{j+1}(\ell_{j+1})$
we obtain
\begin{equation}
\label{e:nonper}
\|\cR^2(U)\|_{T_{j+1}(\ell_{j+1})}
\leq 2^{|\cB_j(U)|}  \sum_{\cG_2(U)}
[O(\|K_j\|_j)]^{|\cX|} \|
\check{K}(\check{X})
-(\delta I)^{\check X}\1_{\cX=\varnothing}
\|_{T_{j+1}(\ell_{j+1})}.
\end{equation}
For brevity let us write $b$ for the factors
$O(\|V_j\|_{j}+\|K_j\|_{j})$ above.  By \eqref{e:nonper} and
Lemma~\ref{lem:NonlinearCheckK}, it suffices to show
\begin{equation} \label{e:nonper-red}
  A^{|\cB_{j+1}(U)|}2^{|\cB_{j}(U)|}
  \sum_{\cG_{2}(U)}
  (b A^{2^{d}}
  )^{|\cX|+|\Comp(\check X)|}
  (\frac{A}{2})^{-|\cB_{j}(\check{X})|} \leq A^{\nu} O(b^2)
  .
\end{equation}
Indeed, \eqref{e:nonper} is bounded by
$O(A^{-|\cB_{j+1}(U)|}\|K_j\|_j/b)$ times this quantity. The small
$\|K_j\|_j/b$ is due to the fact that if $|\cX|\geq 1$ there is a
factor $\|K_j\|_j$ in \eqref{e:nonper} and if $|\cX|=0$ then
\eqref{e:NonlinearCheckK2} provides such a factor in place of
$b$. Hence~\eqref{e:nonper-red} gives
\begin{equation}
  \|\cR^2\|_{j+1}
  \leq O(A^\nu)(\|V_j\|_j+\|K_j\|_j)\|K_j\|_j.
\end{equation}

To verify \eqref{e:nonper-red}, first note that
since $|\cB_{j}(U)|\leq L^{d}|\cB_{j+1}(U)|$, for any $c>0$  the prefactor can be bounded by
\begin{equation}
  A^{|\cB_{j+1}(U)|}2^{|\cB_{j}(U)|} \leq
  (\frac{A}{2})^{(1-c)|\cB_{j+1}(U)|}{2^{(L^{d}+1)|\cB_{j+1}(U)|}} 
  (\frac{A}{2})^{c|\cB_{j+1}(U)|}.
\end{equation}
Taking $c>1$, the product of the first two terms on the last
right-hand side is less than $1$ for $A$ sufficiently large.  It thus
suffices to prove that for some $c>1$
\begin{equation}
  (\frac{A}{2})^{c|\cB_{j+1}(U)|}\sum_{\cG_{2}(U)} (
  bA^{2^d} 
  )^{|\cX|+|\Comp(\check X)|}
  (\frac{A}{2})^{-|\cB_{j}(\check{X})|} \leq   A^{\nu}O(b^2).
\end{equation}
At this point we appeal to \cite[proof of Lemma~6.17]{MR2523458}; this
result estimates the same sum but over $\cG(U)$ instead of $\cG_{2}(U)$.
However, following exactly the same proof as
in~\cite{{MR2523458}} but using that the sum is over $\cG_{2}(U)$,
the supremum over $n\geq 1$ in \cite[(6.85)]{MR2523458} becomes a
supremum over $n\geq 2$ since $|\cX| + |\Comp(\check X)|\geq 2$.
This yields that there exists $c>1$ such that if
$A=A(L,d)$ is large enough, then there is an $m$ such that for all
$U\in\cP_{j+1}$,
\begin{equation}
  (\frac{A}{2})^{c|\cB_{j+1}(U)|}\sum_{\cG_{2}(U)}(
  bA^{2^d} 
  )^{|\cX|+|\Comp(\check X)|}
  (\frac{A}{2})^{-|\cB_{j}(\check{X})|} = O( (bA^{m})^{2}),
\end{equation}
which is $A^{\nu}O(b^{2})$ as needed.
\end{proof}

\begin{proof}[Proof of Lemma \ref{lem:NonlinearCheckK}]
  For notational convenience, for $Y\in \mathcal C_j$ let
  \begin{equation}
    \tilde K(Y)= \sum_{W \in \cP_j(Y)} \theta K_j(Y\setminus W) (\delta I)^W.
  \end{equation}
  We first establish that the claimed bounds follow from the definition of $\check{K}(X)$ in \eqref{e:K+-def2} if we show, for $Y\in \cC_j$,
  \begin{align}
    \label{e:ky}
    \normbb{\tilde K(Y)-\sum_{B \in \cB_j(Y)} {\theta} J(B,Y)}_{T_{j+1}(\ell_{j+1})}
    &\leq A^{2^d}
    O (\|V_j\|_j + \|K_j\|_j)
      (\frac{A}{2})^{-|\cB_j(Y)|}
    \\ \label{e:kyi}
    \normbb{\tilde K(Y)-(\delta I)^Y-\sum_{B \in \cB_j(Y)} {\theta} J(B,Y)}_{T_{j+1}(\ell_{j+1})}
    &\leq A^{2^d}
    O (\|K_j\|_j)
      (\frac{A}{2})^{-|\cB_j(Y)|}.
  \end{align}
  Indeed, though $\check{K}(\check{X}) - (\delta I)^{\check{X}}$ does
  not factor over components $X$ of $\check{X}$, it can be written as a
  sum of $|\Comp(\check{X})|$ terms, each of which is a product over
  the components $X$ of $\check{X}$. That is, we use the
    formula $(a+b)^{n}-a^{n}=\sum_{k=0}^{n-1}a^{k}b(a+b)^{n-k-1}$ with
    $a=(\delta I)^{X}$ and $b=\check{K}(X)-(\delta I)^{X}$. Thus each
    summand
  contains one factor $\check{K}(X)-(\delta I)^{X}$ and the rest of
  the factors are either $\check{K}(X)$ or $(\delta
    I)^{X}$. The estimates
  \eqref{e:NonlinearCheckK}-\eqref{e:NonlinearCheckK2} then follow by
  using \eqref{e:ky}-\eqref{e:kyi} and
  Lemma~\ref{lem:smallset-V1}.
  
  To establish \eqref{e:ky}--\eqref{e:kyi} we apply the triangle inequality. Since $J(B,Y)=0$ if $Y\notin \cS_{j}$, 
  \begin{equation}
    \label{eq:Jterm}
    \normbb{\sum_{B \in \cB_j(Y)}\theta  J(B,Y)
    }_{T_{j+1}(\ell_{j+1})} \leq O(\|K_{j}\|_{j}) 
  \end{equation}
  where we have used
  $\|J(B, X)\|_{T_{j}(\ell_{j})}=O(\|K_j\|_{j})$,
  that $\theta$ contracts from $T_j(\ell_j)$ into $T_{j+1}(\ell_{j+1})$
  and that $|\cB_j(Y)|\leq 2^d$. This shows the $J$
  contributions to~\eqref{e:ky} and~\eqref{e:kyi} satisfy the
  requisite bounds. For the other contributions, note that
  by \eqref{e:deltaIbd}, component factorisation of $K_j$,
  and the contraction property of the norms and $\theta$,
  for $B\in \cB_j$ and $Z\in \cP_j$ we have
  \begin{align}
    \|\delta I(B)\|_{T_{j+1}(\ell_{j+1})}
    &\leq O(\|V_j\|_j+\|K_j\|_j), \\
    \|\theta K_j(Z)\|_{T_{j+1}(\ell_{j+1})} &\leq A^{-\sum_{W\in
        \Comp(Z)} (|\cB_j(W)|-2^d)_+}\|K_j\|_j^{|\Comp(Z)|}. 
  \end{align}

  We now impose the condition that $\epsilon \leq A^{-2^d}$ and that
  $O(\epsilon) \leq A^{-1}$ in the implicit bound below.  Plugging the
  previous bounds into the expression for $\tilde{K}(Y)$ we have
    \begin{align}
      &\|\tilde K(Y)-(\delta I)^Y\|_{T_{j+1}(\ell_{j+1})}
        \nnb
  & \leq  \sum_{Z \in \cP_j(Y):Y \neq Z}
    \|(\delta I)^Z \theta K_j(Y \setminus Z)\|_{T_{j+1}(\ell_{j+1})}
   \nnb
   &\leq  \sum_{Z \in \cP_j(Y):Y \neq Z} 
     (O(\|V_j\|_{j}+ \|K_j\|_{j}))^{|\cB_j(Z)|}{\|K_j\|_j^{|\Comp(Y\backslash Z)|}}
    A^{-\sum_{W\in \Comp(Y\backslash Z)}  (|\cB_j(W)|-2^d)_+}
     \nnb
         &\leq \sum_{Z \in \cP_j(Y):Y \neq Z}
           {\left(A^{2^{d}}\|K_j\|_{j}\right)^{|\Comp(Y\backslash  Z)|}}
           (O(\|V_j\|_{j}+\|K_j\|_{j}))^{|\cB_j(Z)|}A^{-|\cB_{j}(Y\backslash Z)|}
           \nnb
&\leq A^{2^{d}} \|K_j\|_{j}
  (O(\|V_j\|_{j}+ \|K_j\|_{j}) + A^{-1})^{|\cB_{j}(Y)|}
                  \nnb
                  \label{e:NonlinearCheckK2a}
      &\leq A^{2^{d}}\|K_j\|_{j}\left(\frac{A}{2}\right)^{-|\cB_{j}(Y)|} .
\end{align}
Since $\|(\delta I)^Y\|_{T_{j+1}(\ell_{j+1})} \leq [O(\|V_j\|_j+\|K_j\|_j)]^{|\cB_j(Y)|} \leq A O(\|V_j\|_j+\|K_j\|_j) A^{-|\cB_j(Y)|}$ if $O(\epsilon)\leq A^{-1}$,
by the triangle inequality we also have
\begin{equation}
  \|\tilde K(Y)\|_{T_{j+1}(\ell_{j+1})}
  \leq A^{2^{d}}O(\|V_j\|_j+\|K_j\|_{j})\left(\frac{A}{2}\right)^{-|\cB_{j}(Y)|} .
\end{equation}
Together with \eqref{eq:Jterm} this proves the lemma.
\end{proof}

\begin{proof}[Proof of Lemma~\ref{lem:Ri-bd} for $i=3$] 
  The bound on $\cR^3(U)$ is similar to that of $\cR^2(U)$ but simpler
  since only connected $\check{X}$ are involved. In particular,
  the analogue of Lemma~\ref{lem:NonlinearCheckK} only involves the
  reasoning leading to~\eqref{e:NonlinearCheckK2a}, with the key
  difference being that now also $Z\neq\emptyset$. We omit the
  details.
\end{proof} 
\subsection{Flow of the renormalisation group}
\label{sec:flow}

Recall the infinite volume limit of the renormalisation
group maps $\Phi_{j+1,\infty}$ discussed below
Proposition~\ref{prop:consistency}. We equip $\cKbulk_{j}(\Z^{d})$ with the norm $\|K\|_{j}$ defined by \eqref{e:K-norm}.
This space is now infinite dimensional, but it is clear that it is still complete as a normed space, i.e., a Banach space.
Moreover, by the consistency of the finite volume renormalisation group maps (Proposition~\ref{prop:consistency}), the
estimates given in Theorem~\ref{thm:step} also hold for the infinite volume limit.
Next we study the iteration of the renormalisation group maps as a dynamical system.
In what follows $K_0=0$ means $K_0(X)=1_{X=\varnothing}$ for $X\in\cP_j$.

In the next theorem (and later in the paper) we write $\OL(\cdot)$ to indicate
a bound with a constant that may depend on $L$, but where the constant
is uniform in $j$, i.e., $f_{j} = \OL (L^{-j})$ if there is a
$C=C(L)$ such that $f_{j}\leq C L^{-j}$ for all $j$.

\begin{theorem} \label{thm:flow} Let $d\geq 3$, $L\geq L_0$, and
  $A \geq A_0(L)$.  For $m^2 \geq 0$ arbitrary and $b_0$ small, there
  exist $V_0^c(m^{2},b_{0})$ 
  and $\kappa>0$ such that if
  $(V_0,K_0)=(V_0^c(m^2,b_0),0)$ and
  $(V_{j+1},K_{j+1})= \Phi_{j+1,\infty,m^{2}}(V_j,K_j)$ is the flow of
  the infinite volume renormalisation group map then
  \begin{equation} \label{e:flowVK}
    \|V_j\|_j = \OL(b_0L^{-\kappa j}),
    \qquad 
    \|K_j\|_j = \OL(b_0^2 L^{-\kappa j}).
  \end{equation}
  The components of
  $V_0^c(m^2,b_0)$ are continuous and uniformly bounded in $m^2 \geq 0$
  and differentiable in $b_0$ with uniformly bounded derivative.
\end{theorem}

\begin{proof}[Proof of Theorem~\ref{thm:flow}]
  The proof is by a version of the stable manifold theorem for
  smooth dynamical systems. Specifically, we use \cite[Theorem~2.16]{MR2523458}.

  To start, we write down the dynamical system corresponding
  to the renormalisation group map. The definition of $V_{j+1}$ is
  \eqref{e:V+-def}.  We start with the contribution to $V_{j+1}$
  arising from the first term
  \begin{equation}
  \label{eq:E}
    \E_{C_{j+1}}\qB{\theta V_{j}(B)} \bydef \tilde u_{j+1}|B|+\tilde V_{j+1}(B),
  \end{equation}
  where $\tilde u_{j+1}$ and $\tilde V_{j+1}$ are defined by the right-hand side. These
  can be computed by the Wick formula \eqref{e:fWick}, which gives
  \begin{equation}
    \label{eq:tildeflow}
    \tilde z_{j+1} = z_{j} + \kappa_j^{zb}b_j, 
    \qquad \tilde y_{j+1} = y_{j} + \kappa_{j}^{yb}b_{j},
    \qquad \tilde a_{j+1} = a_{j} + \kappa_{j}^{ab}b_{j},
    \qquad \tilde b_{j+1} = b_{j},
  \end{equation}
  with $\kappa_{j}^{yb} = -C_{j+1}(0)$,
  $\kappa_{j}^{ab}=\Delta C_{j+1}(0)$,
  and $\kappa_j^{zb} = \frac{1}{2d}\Delta C_{j+1}(0)$. Indeed,
  non-constant contributions only arise from quartic terms, and Wick's formula gives
  \begin{align}
    \E_C\qB{\theta(\psi_x\bar\psi_x(\nabla_e\psi)_x(\nabla_e\bar\psi)_x)}
    &=
      \psi_x\bar\psi_x(\nabla_e\psi)_x(\nabla_e\bar\psi)_x
      -C(0) (\nabla_e\psi)_x(\nabla_e\bar\psi)_x - \nabla_e\nabla_{-e} C(0) \psi_x\bar\psi_x
    \nnb
    &\qquad\qquad 
    + \nabla_{-e}C(0)\psi_x(\nabla_e\bar\psi)_x
    - \nabla_eC(0)\bar\psi_x(\nabla_e\psi)_x + \text{(constant)}.
  \end{align}
  Since $\nabla_e C(0) = C(e)-C(0) = \frac{1}{2d}\sum_{e\in \cE_d} (C(e)-C(0))= \frac{1}{2d}\Delta C(0)$
  for all $e\in \cE_d$ by symmetry,
  and since the lattice Laplacian has the representations $\Delta = -\frac12 \sum_{e\in\cE_d}\nabla_e\nabla_{-e} = \sum_{e\in\cE_d} \nabla_e$,
  therefore
  \begin{align}
    \E_C\qB{\theta(\psi_x\bar\psi_x(\nabla\psi)_x(\nabla\bar\psi)_x)}
    &= \frac12 \sum_{e\in\cE_d}       \E_C \qB{\theta(\psi_x\bar\psi_x(\nabla_e\psi)_x(\nabla_e\bar\psi)_x)}
      \nnb
    &=
      \psi_x\bar\psi_x(\nabla\psi)_x(\nabla\bar\psi)_x
      -C(0)(\nabla\psi)_x(\nabla\bar\psi)_x + \Delta C(0)\psi_x\bar\psi_x
      \nnb
      &\qquad\qquad
      +\frac{\Delta C(0)}{2d}(\frac12 \psi_x(\Delta\bar\psi)_x+\frac12 (\Delta\psi)_x\bar\psi_x) + \text{(constant)}.
  \end{align}
  Since $\|V_j(B)\|_{T_j(\ell_j)}$ is comparable with
  $|z_j| + |y_j| + L^{2j}|a_j| + L^{-(d-2)j} |b_j|$, i.e.,
  $\|V_j(B)\|_{T_j(\ell_j)} = \OL(|z_j| + |y_j| + L^{2j}|a_j| + L^{-(d-2)j} |b_j|) = \OL(\|V_j(B)\|_{T_j(\ell_j)})$,
  it is natural to define the rescaled variables and coefficients
  $\hat z_{j}=z_{j}$, $\hat y_{j}=y_{j}$,
  $\hat a_{j}=L^{2j}a_{j}$, $\hat b_{j}=L^{-(d-2)j}b_{j}$,
  $\hat \kappa_{j}^{ab}=L^{2+dj}\kappa_{j}^{ab}$,
  $\hat \kappa_{j}^{yb}=L^{(d-2)j}\kappa_{j}^{yb}$,
  and $\hat\kappa_j^{zb}=L^{(d-2)j}\kappa_j^{zb}$.
  The definition \eqref{e:V+-def} of $V_{j+1}$ then becomes
  \begin{alignat}{2}
    \label{eq:flow1}
    \hat z_{j+1} &= \hat z_{j} + \hat\kappa_j^{zb}\hat b_j
    +   \hat r^{z}_{j},
    &\qquad
    \hat y_{j+1}&= \hat y_{j} + \hat\kappa_{j}^{yb}\hat b_{j} + \hat
    r^{y}_{j}, \\
    \label{eq:flow2}
    \hat a_{j+1}&= L^{2}\hat a_{j} + \hat\kappa_{j}^{ab}\hat b_{j} + \hat r^{a}_{j}, &\qquad
    \hat b_{j+1}&= L^{-(d-2)}\hat b_{j} + \hat r^{b}_{j}.
  \end{alignat}
  Here $\hat r_{j}$ is the $4$-component vector of real numbers determined by the $\Loc$ step of the renormalisation group map.
  Each component is thus a linear function of $K_j$, and by
  \eqref{e:step-V} of Theorem~\ref{thm:step} has size
  $O(L^{d}\norm{K_{j}}_{j})$.  The $\hat \kappa_{j}$ 
  are uniformly bounded in $j$ by the covariance estimates~\eqref{e:C-bd}.

  We now reorganize variables.  Set $v_j=(\hat y_j,\hat z_j,\hat a_j)$
  and $w_j=(\hat b_j,K_j)$ and use $\|\cdot\|_{j}$ for the norm
  given by maximum of the (norm of the) respective components.
  The index $j$ does not play a role for $\|v_{j}\|_{j}$, but it
  does for $\|w_{j}\|_{j}$. By the computation above and
  Theorem~\ref{thm:step} the infinite volume renormalisation group map
  can be written in the block diagonal form
  \begin{equation}\label{e:flowblock}
    \begin{pmatrix}v_{j+1} \\ w_{j+1}\end{pmatrix}
    = \begin{pmatrix} E & B_j \\ 0 & D_j \end{pmatrix}
    \begin{pmatrix}v_{j} \\ w_{j}\end{pmatrix}
    + \begin{pmatrix} 0 \\ g_{j+1}(v_j,w_j)  \end{pmatrix}.
  \end{equation}
  In this formula, $E$ comes from the first terms on the right-hand sides of the $\hat z$, $\hat y$, and $\hat a$ equations in  \eqref{eq:flow1} and \eqref{eq:flow2},
  $B_j$ represents the $\hat\kappa^{xb}_j$ and the $\hat r^x_{j}$ terms with $x=z,y,a$, and $D_j$ represents the first term in the $\hat b$ equation in \eqref{eq:flow2} and the linearisation of $(0, K_j)\mapsto K_{j+1}$.
  Finally, $g_{j+1}(v_j, w_j)$ is then the non-linear remainder of $K_{j+1}$ after the linearisation is removed.  It follows from these identifications that
  $g_j(0,0)=0$ and $Dg_j(0,0)=0$.  Moreover $g_j$ is analytic  in its arguments, so that all structural hypotheses required to apply  \cite[Theorem~2.16]{MR2523458} hold.
  
  As for the requisite norm estimates,  since it is a $3\times 3$ 
  triangular matrix with non-zero diagonal entries, $E$ is
  invertible with bounded inverse $E^{-1}$. As
  indicated above, $\|\hat{r}_j\|_{j+1}=O(L^{d}\|K_j\|_j)$,
  so $\|B_j\|_{j\to j+1}$ is bounded.   Finally,
  the derivative estimates on the renormalisation group map  from Theorem~\ref{thm:step} imply the following norm bounds on $D_j$:
  \begin{equation}
   \|D_j\|_{j\to j+1} \leq b_0 \max\{L^{-(d-2)}, O(L^{-3}+A^{-\eta})\} \leq
   O(b_0 L^{-\kappa}),
  \end{equation}
  with the latter inequality holding provided $A$ is large enough.

  For every $m^2\geq 0$, $b_0$ sufficiently small, and $K_0=0$,  \cite[Theorem 2.16]{MR2523458} now implies that we can find 
  an initial $3$-tuple of coupling constants $v_0^c(m^{2},b_{0})$, or equivalently an initial local potential $V_{0}^{c}(m^{2},b_{0})$, so that
  for some $\kappa>0$,
  \begin{equation}
    \|v_j\|_{j}+\|w_j\|_{j}=\OL(b_0L^{-\kappa j}).
  \end{equation}
  These bounds are not explicit in the statement
  of \cite[Theorem 2.16]{MR2523458}, but are immediate from its proof
  (because the proof constructs a solution in a correspondingly weighted sequence space).
  In particular this bound implies that $\|K_j\|_j+\|V_j\|_j=
  \OL(\|v_j\|_{j}+ \|w_j\|_{j})=\OL(b_0L^{-\kappa j})$.

  Smoothness of the renormalisation group map implies that $V_{0}^{c}(m^{2},b_{0})$ is smooth in $b_{0}$.
  To see that $V_0^c(m^2,b_0)$ is also continuous in $m^2 \geq 0$, one can for example regard $v_j$ and $w_j$
  as bounded continuous functions of $m^2$, i.e., consider
  $v_j \in C_b([0,\infty),\R^3)$ and
  $w_j\in C_b([0,\infty), \R \times \cKbulk_j(\Z^d))$. Since all the
  estimates above are uniform in $m^2 \geq 0$, the previous argument
  applies in these spaces and
  shows that the solution is continuous in $m^{2}$.
\end{proof}

\begin{remark}
Note that while Theorem~\ref{thm:step} assumes that $u_j=0$
and produces $u_{j+1}$, it is trivial to extend the statement to $u_j\neq 0$
by simply adding $u_j$ to the $u_{j+1}$ produced for $u_j=0$.
\end{remark}

\medskip

By consistency, the finite volume renormalisation group flow for $V_j$ agrees
with the infinite volume renormalisation group flow up to scale $j<N$
provided both have the same initial condition. As a result we obtain
the following corollary by iterating the recursion \eqref{e:step-K}
for the $K$-coordinate using the \emph{a priori} knowledge that
$\|V_j\|_j=\OL(b_0L^{-\kappa j})$ due to Theorem~\ref{thm:flow}.

\begin{corollary} \label{cor:flow-finvol}
  Under the same assumptions as in Theorem~\ref{thm:flow}, the same
  estimates hold for the finite volume renormalisation group flow
  for all $j \leq N$,
  and the $V_j$ and $u_j$ produced by the finite volume  renormalisation group flow
  are the same as those for the infinite volume flow when $j<N$.
\end{corollary}

From this it follows that if $e^{-u_{N}|\Lambda_N|}$ denotes the total
prefactor accumulated in the renormalisation group flow up to scale $N$,
then $u_{N}$ is uniformly bounded in $N$ and $m^2$ as $m^2\downarrow 0$ if
we begin with $V_0$ as in Theorem~\ref{thm:flow}.  Indeed, up to scale
$N-1$ this follows from the bounds \eqref{e:flowVK} and \eqref{E:vj}.
In passing from the scale $N-1$ to $N$, the renormalisation group step
is $\Lambda_N$-dependent, but is nevertheless uniformly bounded by the
last statement of Theorem~\ref{thm:step}.

\section{Computation of the susceptibility}
\label{sec:suscept}

In the remainder of the paper, we will use the notation (with$\Lambda = \Lambda_{N}$)
\begin{equation} \label{e:avgV0}
  \avg{F} = \avg{F}_{V_0} = \frac{\E_C\qB{e^{-V_0(\Lambda)}F}}{\E_C\qB{e^{-V_0(\Lambda)}}}
\end{equation}
and assume that $(V_j,K_j)_{j=0,\dots,N}$ is a renormalisation group flow, i.e., $(V_{j+1},K_{j+1})=\Phi_{j+1}(V_j,K_j)$.

In this section we express the susceptibility in terms of the
dynamical system generated by the (bulk) renormalisation group flow.
First recall that $Z_0=e^{-V_0(\Lambda)}$ and that
\begin{equation}
  \label{e:S4C}
  C = (-\Delta+m^2)^{-1} = C_1 + \cdots + C_{N-1} + C_{N,N}, \qquad C_{N,N} = C_N + t_N Q_N,
\end{equation}
where $\Delta$ is the Laplacian on $\Lambda_N$, $t_N = m^{-2}- O(L^{2N})$ is a constant,
  and the matrix $Q_N$ is the orthogonal projection onto constants (i.e., all entries equal $1/|\Lambda_N|$).
Using \eqref{e:E-semigroup} and \eqref{e:EZj}, with $u_N$ as in \eqref{e:ZVK},
we then set
\begin{equation} \label{e:ZNN-def}
  Z_{N,N} = \E_{t_NQ_N}\qB{\theta Z_N} =  \E_C\qB{\theta Z_0},
  \qquad \tilde Z_{N,N} = e^{u_N|\Lambda_N|} Z_{N,N} 
\end{equation}
where $\E_{t_NQ_N}\theta$ is the fermionic Gaussian convolution with
covariance $t_NQ_N$ defined in Section~\ref{sec:frd}.  Thus
$\tilde Z_{N,N}$ is still a function of $\psi,\bar\psi$, i.e., an
element of $\cN(\Lambda)$. Note that $\E_{C}[Z_{0}]$ is the constant term of $Z_{N,N}$, i.e.,
obtained from $Z_{N,N}$ by formally setting $\bar\psi$ and $\psi$ to $0$.
The following technical device of restricting to
constant fields $\psi,\bar\psi$ will be useful for extracting
information.  By restriction to constant $\psi,\bar\psi$ we mean
applying the homomorphism from $\cN(\Lambda)$ onto itself that acts on
the generators by
$\psi_x \mapsto \frac{1}{|\Lambda|}\sum_{x\in\Lambda} \psi_x$ and
likewise for the $\bar\psi_x$.  The result is an element in the
subalgebra of $\cN(\Lambda)$ generated by
$\frac{1}{|\Lambda|}\sum_{x\in\Lambda} \psi_x$ and
$\frac{1}{|\Lambda|}\sum_{x\in\Lambda} \bar\psi_x$; we will simply
denote these generators by $\psi$ and $\bar\psi$ when no confusion can
arise. To explain the notation in the next proposition, note
that a general even element of this subalgebra can be written as
$F^{0}+F^{2}\psi\bar\psi$ for some constants $F^{0},F^{2}$, c.f.\ \eqref{e:Fzsum}.

\begin{proposition} \label{prop:ZNN}
  Restricted to constant $\psi,\bar\psi$,
  \begin{equation} \label{e:ZNN}
    \tilde Z_{N,N}
    =
    1+\tilde u_{N,N} - |\Lambda_N| \tilde a_{N,N}\psi\bar\psi,
    \qquad
    \tilde u_{N,N} =k_N^0+\tilde a_{N,N}t_N,
    \qquad \tilde a_{N,N} = a_N- \frac{k_N^2}{|\Lambda_N|}
  \end{equation}
  where 
  \begin{equation} \label{e:ZNN-bd}
    k_N^0 = O(\|K_N\|_N), \qquad k_N^2=O(\ell_N^{-2}\|K_N\|_N)
    .
  \end{equation}
  If $V_{0}, K_{0}$ are continuous in $m^{2}\geq 0$ and
  $b_{0}$ small enough, then so are $k_{N}^{0}$ and $k^{2}_{N}$.
\end{proposition}

\begin{proof}
  Since the set of $N$-polymers $\cP_N(\Lambda_N)$ is $\{\varnothing,
  \Lambda_{N}\}$ and $e^{u_{N}|\Lambda_N|}$ is a constant,  
  \eqref{e:ZVK} and \eqref{e:ZNN-def} simplify to
  \begin{equation}
    \label{e:P41integrand}
    \tilde Z_{N,N} = \E_{t_NQ_N}\qB{\theta(e^{-V_N(\Lambda_{N})}+K_N(\Lambda_{N}))}.
  \end{equation}
  We now evaluate the integral over the zero mode with covariance
  $t_NQ_N$.  To this end, we restrict $V_N(\Lambda_{N})$ and
  $K_N(\Lambda_{N})$ to spatially constant $\psi,\bar\psi$ and denote
  these restrictions by $\widetilde{V}_{N} (\Lambda_{N})$ and
  $\widetilde{K}_{N}(\Lambda_{N})$.  Since
  $\widetilde{V}_{N}$ and $\widetilde{K}_{N}$ are
  even, they are of the form
  \begin{align}
    \widetilde{V}_{N} (\Lambda_{N},\psi,\bar\psi) &= |\Lambda_{N}|a_N \psi\bar\psi
    \\
    \widetilde{K}_{N} (\Lambda_{N},\psi,\bar\psi) &= k_{N}^{0} + k_N^{2} \psi\bar\psi,
  \end{align}
  where the form of $\widetilde{V}_{N}$ follows from the representation \eqref{e:VX}.
  Thus restricted to constant $\psi$ and $\bar\psi$ the integrand in \eqref{e:P41integrand} is
  \begin{equation}
    e^{- {\widetilde{V}_{N}} (\Lambda_{N})} + \widetilde{K}_{N} (\Lambda_{N})
    = 1 +  k_N^0  - (|\Lambda_{N}|a_N - k_N^2)\psi\bar\psi
    .
  \end{equation}
  Therefore applying the fermionic Wick formula
  $\E_{t_NQ_N}\theta \psi\bar\psi = -t_N|\Lambda_{N}|^{-1}
  +\psi\bar\psi$, we obtain \eqref{e:ZNN}.  The continuity claims for
  $k_{N}^{0}$ and $k_{N}^{2}$ follow as 
  $(V_j,K_j)$ is a renormalisation group flow (see below
  \eqref{e:avgV0}) and since the
  renormalisation group map has this   continuity. 

  The bounds \eqref{e:ZNN-bd} follow from the definition of the
  $T_N(\ell_N)$ norm.  Indeed, since $k_0$ is the constant coefficient of
  $K_N(\Lambda_{N})$, clearly $k_N^0=O(\|K_N\|_N)$.  Similarly, setting
  $g_{(x,1),(y,-1)} = 1$ for all $x,y\in\Lambda_{N}$ and $g_z=0$ for all
  other sequences, we have $\|g\|_{\Phi_N(\ell_N)} = \ell_N^{-2}$ and
  \begin{equation} \label{e:kN2normbd}
    |k_N^2| = |\avg{K_N(\Lambda_{N}),g}|
    \leq \|g\|_{\Phi_N(\ell_N)}\|K_N\|_N
    =
    \ell_N^{-2}\|K_N\|_N
  \end{equation}
  where $\avg{\cdot,\cdot}$ is the pairing from Definition~\ref{def:Tphi}.
\end{proof}

\begin{proposition} \label{prop:suscept}
  Using the notation of Proposition~\ref{prop:ZNN},
  \begin{equation} \label{e:suscept-atilde}
    \sum_{x\in\Lambda_N} \avg{\bar\psi_0\psi_x} = \frac{1}{m^2} 
    - \frac{1}{m^4} 
    \frac{\tilde a_{N,N}}{1+\tilde u_{N,N}}
    .
  \end{equation}
\end{proposition}

\begin{proof} 
  We amend the algebra $\cN(\Lambda_{N})$ by two Grassmann variables $\sigma$ and $\bar\sigma$ which we view as constant fields
  $\sigma_x=\sigma$ and $\bar\sigma_x=\bar\sigma$. We then consider
  the fermionic cumulant generating function (an element of the
  Grassmann algebra generated by $\sigma$ and $\bar\sigma$) 
  \begin{equation} \label{e:Gammadef}
    \Gamma(\sigma,\bar\sigma)= \log \E_{C}\qB{Z_0(\psi,\bar\psi)e^{(\sigma,\bar\psi)+(\psi,\bar\sigma)}},
  \end{equation}
  where $C$ is as in \eqref{e:S4C}. By translation invariance
  \begin{equation}
    \sum_{x\in\Lambda_{N}}\avg{\bar\psi_0\psi_x}
    = \frac{1}{|\Lambda_{N}|}\sum_{x,y\in\Lambda_N}\avg{\bar\psi_x\psi_y}
    = \frac{1}{|\Lambda_{N}|} \partial_{\bar\sigma}\partial_{\sigma}
    \Gamma(\sigma,\bar\sigma).
  \end{equation}
  The linear change of generators $\psi\mapsto \psi+C\sigma$, $\bar\psi\mapsto
  \bar\psi + C\bar\sigma$ yields
  \begin{equation} 
    \Gamma(\sigma,\bar\sigma)= (\sigma,C\bar\sigma) +
    \log \E_{C}\qB{\theta Z_{0}(C\sigma,C\bar\sigma)},
  \end{equation}
  where the right-hand side is to be interpreted as applying
  the doubling homomorphism and then restricting to constant $\sigma,\bar\sigma$.
  Since $\sigma$ is constant in $x\in\Lambda_N$, we have $C\sigma = m^{-2} \sigma$.
  With \eqref{e:ZNN-def} thus
  \begin{equation} \label{e:Gammasquare}
    \Gamma(\sigma,\bar\sigma)= m^{-2}|\Lambda_{N}|\sigma\bar\sigma
    + \log \tilde Z_{N,N}(m^{-2}\sigma,m^{-2}\bar\sigma) - u_N|\Lambda_N|.
  \end{equation}
  As a result, by \eqref{e:ZNN}--\eqref{e:ZNN-bd}, 
  \begin{equation}
    \frac{1}{|\Lambda_{N}|}\partial_{\bar\sigma} \partial_{\sigma}\Gamma(\sigma,\bar\sigma)
    =  
    \frac{1}{m^2}
    - \frac{1}{m^4} 
    \frac{\tilde a_{N,N}}{1+\tilde u_{N,N}} .\qedhere
  \end{equation}
\end{proof}

\section{The observable renormalisation group flow}
\label{sec:obs-flow}

Recall that $\avg{\cdot}$ denotes the expectation \eqref{e:avgV0}, in
which we will ultimately choose $V_0=V_0^c(b_0,m^2)$ as in
Theorem~\ref{thm:flow}. This section sets up and analyses the
renormalisation group flow associated to source fields. This will
enable the computation of correlation functions (observables) like
$\avg{\bar\psi_{a}\psi_{b}}$ in Section~\ref{sec:obs}. Our strategy is
inspired by that used in \cite{MR3345374,MR3459163}, but with
important differences arising due to the presence of a non-trivial
zero mode in our setting.

\subsection{Observable coupling constants}
\label{sec:obsfields}

As in the proofs in Section~\ref{sec:suscept}, we amend the Grassmann
algebra by two source fields.  Now, however, the additional fields
are not constant in space but rather are localised at two points
$a,b\in\Lambda = \Lambda_{N}$.
We distinguish between two cases:

\medskip \noindent
\emph{Case~(1).}
For the two point function $\avg{\bar\psi_a\psi_b}$
(which we call `Case~(1)'), the additional source fields $\sigma_a$ and
$\bar\sigma_b$ are two additional Grassmann variables that anticommute
with each other and the $\psi,\bar\psi$.

\medskip \noindent
\emph{Case~(2).}
For the quartic correlation
function $\avg{\bar\psi_a\psi_a\bar\psi_b\psi_b}$ (called `Case~(2)'),
we introduce additional Grassmann variables
$\bar\vartheta_{x},\vartheta_{x}$ for $x\in \{a,b\}$ and the
additional source fields $\sigma_a$ and $\sigma_b$ are the commuting
variables $\sigma_{x}=\bar\vartheta_{x}\vartheta_{x}$ for $x\in \{a,b\}$.
This explicit $U(1)$ invariant choice will be convenient when discussing symmetries.

\medskip

In both cases we relabel the initial potential $V_0$ from Section $3$
as $\Vbulk_0$ and set $V_0 = \Vbulk_0+\Vobs_0$ where $\Vobs_0$ is an
observable part to be defined.

\medskip \noindent
\emph{Case~(1).}
In this case, 
\begin{equation} \label{e:Vobs0-a}
  \Vobs_0 = - \lambda_{a,0} \sigma_a \bar\psi_a 
  - \lambda_{b,0}\psi_b \bar\sigma_b  .
\end{equation}
The spatial index of $\Vobs_{0}$ signals the local nature of
the source fields. More precisely, the evaluation of
$\Vobs_{0}$ on a set $X$ is defined to be spatially localised:
$\Vobs_{0}(X) = - \lambda_{a,0} \sigma_a \bar\psi_a 1_{a\in X} -
\lambda_{b,0}\psi_b \bar\sigma_b 1_{b\in X}$. Recalling that
  $C=(-\Delta + m^{2})^{-1}$, see \eqref{e:C-def}, it follows that
\begin{equation}
  \label{e:s5cora}
  \avg{\bar\psi_a\psi_b} = \frac{1}{\lambda_{a,0}\lambda_{b,0}}
  \partial_{\bar\sigma_b}\partial_{\sigma_a}
  \log \E_{C} \qB{e^{-V_0(\Lambda)}}.
\end{equation}
Obtaining \eqref{e:s5cora} is
just a matter of expanding $e^{-V_{0}^{\obs}(\Lambda)}$, using
$\avg{\bar\psi_{a}}=\avg{\psi_{b}}=0$, and applying the rules of
Grassmann calculus.  Note the order of $\partial_{\bar\sigma_b}$ and
$\partial_{\sigma_a}$, which is important to obtain the correct sign.
Although \eqref{e:s5cora} holds for any constants
$\lambda_{a,0}, \lambda_{b,0}$, it is convenient for us to leave these
as variables to be tracked with respect to the renormalisation group flow.

\medskip \noindent
\emph{Case~(2).}
Similarly to the previous case, we choose
\begin{equation} \label{e:Vobs0-b}
  \Vobs_0 = - \lambda_{a,0} \sigma_a \bar\psi_a\psi_a 
  -
  \lambda_{b,0} \sigma_b\bar\psi_b\psi_b, 
\end{equation}
so that
\begin{align}
  \label{e:s5corb1}
  \avg{\bar\psi_a\psi_a}
  &= \frac{1}{\lambda_{a,0}} \partial_{\sigma_a} \log \E_{C} \qB{e^{-V_0(\Lambda)}} \Big|_{\lambda_{b,0}=0}
  \\
  \label{e:s5corb2}
  \avg{\bar\psi_a\psi_a\bar\psi_b\psi_b}-\avg{\bar\psi_a\psi_a}\avg{\bar\psi_b\psi_b}
  &=
    \frac{1}{\lambda_{a,0}\lambda_{b,0}} \partial_{\sigma_b}\partial_{\sigma_a} \log \E_{C} \qB{e^{-V_0(\Lambda)}}
    .
\end{align}
To distinguish the coupling constants in the two cases, we will
sometimes write $\lambda_{a,0}^{(p)}$ with $p=1$ or $p=2$ instead of
$\lambda_{a,0}$, and analogously for the other coupling constants.

\subsection{The free observable flow}
\label{sec:obs-free}

To orient the reader and motivate the discussion which follows, let us
first consider the noninteracting case $\Vbulk_0=0$, in which the
microscopic model is explicitly fermionic Gaussian.  In this case, one
may compute all correlations explicitly by applying the fermionic Wick
rule.  The same computation can be carried out inductively using the
finite range decomposition of the covariance $C$, and we review this now
as it will be the starting point for our analysis of the interacting case.

To begin the discussion, observe that all source fields square
to zero, i.e., $\sigma_a^2 = \bar\sigma_{b}^{2} = \sigma_b^2 = 0$.
This implies that $\Vobs_0(\Lambda)^3=0$  since $\Vobs_{0}(\Lambda)$ has no constant
term and has at least one least source field in each of its two summands.
Given $\Vobs_{0}$, we inductively
define renormalised interaction potentials that share this property:
\begin{equation} \label{e:Vobspt}
u_{j+1}^{\obs}{(\Lambda)}+  \Vobs_{j+1}(\Lambda) \bydef
  \E_{C_{j+1}}\qB{\theta \Vobs_j(\Lambda)} - \frac12 \E_{C_{j+1}}\qB{\theta \Vobs_j(\Lambda); \theta \Vobs_j(\Lambda)}
\end{equation}
where
\begin{equation}
  \E_{C_{j+1}}\qB{\theta \Vobs_j(\Lambda); \theta \Vobs_j(\Lambda)}
  \bydef
  \E_{C_{j+1}}\qB{\theta \Vobs_j(\Lambda)^{2}}
  -
  \pB{\E_{C_{j+1}}\qB{ \Vobs_j(\Lambda)}}^{2}
\end{equation}
and $u_{j+1}^{\obs}{(\Lambda)}$ collects the terms that do not contain $\psi$ or $\bar\psi$. 
Consequently, one can check that
\begin{equation}
  \E_{C_{j+1}}\qB{\theta e^{-\Vobs_j(\Lambda)}}
  =
  \E_{C_{j+1}}\qB{\theta (1-\Vobs_j(\Lambda) + \frac12 \Vobs_j(\Lambda)^2)}
  = e^{-u_{j+1}^{\obs}{(\Lambda)}-\Vobs_{j+1}(\Lambda)}.
\end{equation}

For convenience, in the last step when $j=N$, we set $C_{N+1}=t_NQ_N$. 
This separation of the zero mode is not essential here but will be useful for our
analysis in the interacting case.

For $j>0$, the $\Vobs_j$ have 
terms not present in $\Vobs_{0}$, for example the terms involving
$q$ in the next definition.
The nilpotency of the source fields $\sigma_{a}, \sigma_{b}$ limits the possibilities.

\begin{definition} \label{def:cVobs}
Let  $\cV^\obs$ be the space of formal field polynomials $u^\obs+\Vobs$ of the form:
  \begin{alignat}{1}
 &\left.\begin{aligned}
     \Vobs &=
    -  \lambda_a \sigma_a \bar\psi_a - \lambda_b \psi_b\bar\sigma_b
    + \sigma_a \bar\sigma_b\frac{r}{2}(\bar\psi_a\psi_a+\bar\psi_b\psi_b),
    \;\,\quad
    \\
   \nonumber u^\obs&= - { \sigma_a\bar\sigma_b }  q,
       \end{aligned}
 \right\} \qquad \text{in Case (1),}\\
 &  \left.\begin{aligned}
     \Vobs &=
    { -}\sigma_a \lambda_a\bar\psi_a\psi_a - \sigma_b\lambda_b\bar\psi_b\psi_b
    { -}\sigma_a\sigma_b\frac{\eta}{2} (\bar\psi_a\psi_b+\bar\psi_b\psi_a)
    \\
    & \nonumber \quad \quad  \quad \quad \quad
    + \sigma_a \sigma_b \frac{r}{2}(\bar\psi_a\psi_a+\bar\psi_b\psi_b)),\\
      \nonumber u^\obs&=- \sigma_a  \gamma_a- \sigma_b  \gamma_b-  { \sigma_a\sigma_b }  q ,
       \end{aligned}
 \right\} \qquad \text{in Case (2),}
  \end{alignat}
for observable coupling constants $(\lambda_{a},\lambda_{b},q, r)\in\C^{4}$
respectively $(\lambda_{a},\lambda_{b},\gamma_{a},\gamma_b,q,\eta, r)\in\C^{7}$.
 For $X \subset \Lambda$, we define $(u^{\obs}+\Vobs)(X) \in
 \cN^\obs(X \cap \{a,b\})$ by 
  \begin{equation} \label{e:VobsX}
    (u^\obs+\Vobs)(X) =
    -\lambda_a \sigma_a \bar\psi_a 1_{a\in X} -\lambda_b\psi_b
    \bar\sigma_b 1_{b \in X} 
    - \sigma_a  \bar\sigma_b q 1_{a\in X,b \in X}
    +\sigma_a  \bar\sigma_b \frac{r}{2}(\bar\psi_a\psi_a+\bar\psi_b\psi_b)1_{a\in X,b\in X}
  \end{equation}
  in Case (1), and analogously in Case (2).
\end{definition}

\begin{remark} The terms corresponding to $r$ do not appear at any
  step of the free observable flow \eqref{e:Vobspt} if we
  start with them equal to $0$.  We include them here in preparation
  for the interacting model.
\end{remark}
The evolutions of $\uobs_j+\Vobs_j \to u^{\obs}_{j+1}$ and $\Vobs_j \to \Vobs_{j+1}$ are equivalent to the evolution
of the coupling constants $(\lambda_a,\lambda_b,q, r)$ respectively
$(\lambda_a,\lambda_b,\gamma_a,\gamma_b,q,\eta,r)$.
By computation of the fermionic Gaussian moments in \eqref{e:Vobspt},
the flow of the observable coupling constants according to
\eqref{e:Vobspt} is then given as follows. Note 
that the evolution of coupling constants in $\Vobs$ is independent
  of the coupling constants in $u^\obs$.

\begin{lemma} \label{lem:Vobs-flow-free}
  Let $\Vbulk_0=0$, and let $u^\obs_{j}$ and $\Vobs_j$ be of the form as in Definition~\ref{def:cVobs}.
  The map \eqref{e:Vobspt} is then given as follows.
  In Case (1), for $x\in \{a,b\}$,
\begin{align}
  \lambda_{x,j+1} &=\lambda_{x,j} 
  \\
  q_{j+1} &= q_j + \lambda_{a,j}\lambda_{b,j} C_{j+1}(a,b) + r_jC_{j+1}(0,0)
            \\
  r_{j+1} &= r_j
                  ,
            \intertext{whereas in Case (2), for $x\in\{a,b\}$,}
            \label{e:c2freelam}
  \lambda_{x,j+1} &= \lambda_{x,j} 
  \\
  \gamma_{x,j+1} &= \gamma_{x,j} + \lambda_{x,j}C_{j+1}(0,0) 
  \\
  \label{e:c2freeq}
  q_{j+1} &= q_j + \eta_j C_{j+1}(a,b)
            + r_{j} C_{j+1}(0,0)
            - \lambda_{a,j}\lambda_{b,j} C_{j+1}(a,b)^2 
  \\
  \label{e:c2freeeta}
  \eta_{j+1} &= \eta_j - 2\lambda_{a,j}\lambda_{b,j} C_{j+1}(a,b)
  \\
  r_{j+1} & =r_{j}.
\end{align}
\end{lemma}

\begin{proof}
  This follows from straightforward evaluation of \eqref{e:Vobspt} using~\eqref{e:fWick}.
\end{proof}

To continue the warm-up for the interacting case, we illustrate how
these equations reproduce the direct computations of correlation
functions (and explain the terminology of observable coupling constants).
When $r_{0}=0$, by a computation using
Definition~\ref{def:cVobs}, the formulas~\eqref{e:s5cora}
and~\eqref{e:s5corb1}--\eqref{e:s5corb2} imply 
the correlation
functions in Cases (1) and (2) are given by
\begin{equation}
  \label{e:warmupcor}
  \avg{\bar\psi_a\psi_b} = \frac{ q_{N+1}}{\lambda_{a,0}\lambda_{b,0}},
  \qquad 
  \avg{\bar\psi_a\psi_a} = \frac{\gamma_{a,N+1}}{\lambda_{a,0}},
  \qquad
  \avg{\bar\psi_a\psi_a;\bar\psi_b\psi_b} = \frac{q_{N+1}}{\lambda_{a,0}\lambda_{b,0}},
\end{equation}
with $q\bydef q^{(1)}$ and $\lambda \bydef \lambda^{(1)}$ for the first equation
and $q\bydef q^{(2)}$, $\gamma \bydef \gamma^{(2)}$, and $\lambda \bydef \lambda^{(2)}$ for the last two.
Recalling the convention $C_{N+1}=t_{N}Q_{N}$,
for Case~(2) with $r_{0}=0$ a computation using
\eqref{e:c2freelam}, \eqref{e:c2freeq} and~\eqref{e:c2freeeta} shows
\begin{equation}
  q_{N+1}
  = - \lambda_{0,a}\lambda_{0,b}\pa{ \sum_{k \leq N} C_k(a,b) + t_NQ_N(a,b)}^2 = -\lambda_{0,a}\lambda_{0,b}(-\Delta+m^2)^{-1}(a,b)^2
  ,
\end{equation}
with the final equality by \eqref{e:C-def}. Combined with \eqref{e:warmupcor},  this gives
\begin{equation}
  \label{e:warmupcor-conclusion}
  \avg{\bar\psi_a\psi_a;\bar\psi_b\psi_b} =-(-\Delta+m^2)^{-1}(a,b)^2,
\end{equation}
as expected for free fermions.

\begin{remark}
  In the preceding computation we kept the potential in the
  exponential for the entire computation, whereas in
  Sections~\ref{sec:suscept} and~\ref{sec:obs} the zero mode is
  integrated out directly without rewriting the integrand in this form
  (see, e.g., \eqref{e:ZNN-def}).  We distinguish these two approaches
  by using $N+1$ subscripts for the former and $(N,N)$ for the latter,
  and by putting tildes on quantities associated with the $(N,N)$\textsuperscript{th}
  step as was done in Section~\ref{sec:suscept}.
\end{remark}

Before moving to the interacting model, we introduce the
\emph{coalescence scale $j_{ab}$} as the largest integer $j$ such that
$C_{k+1}(a,b)=0$ for all $k < j$, i.e.,
\begin{equation} 
\label{E:coal}
  j_{ab} = \floor{\log_L (2|a-b|_{{\infty}})}.
\end{equation}
In the degenerate cases $\lambda_a=0$ or $\lambda_b=0$ when only one
of the source fields is present we use the convention
$j_{ab}=+\infty$.  Note that the finite range property
\eqref{e:C-range} implies that $q_j=\eta_j=r_{j}=0$ for $j<j_{ab}$
provided they are all $0$ when $j=0$.  This will also be true in the
interacting case.

In connection with the coalescence scale, we also make a convenient
choice of the block decomposition of $\Lambda_{N}$ based on the
relative positions of $a$ and $b$.  Namely, we center the block
decomposition such that point $a$ is in the center (up to rounding if
$L$ is even) of the blocks at all scales $1 \leq j \leq N$.  This
implies that if $|a-b|_\infty < \frac12 L^{j+1}$ the scale-$j$ blocks
containing $a$ and $b$ are contained in a common scale-$(j+1)$ block.

\subsection{Norms with observables}

To extend the above computation for $\Vbulk=0$ to the interacting
case, we will extend the renormalisation group map to the Grassmann
algebra amended by the source fields.
In Case~(2), recall that the source fields $\sigma_a$ and
$\sigma_b$ are even elements
rather than Grassmann generators themselves,
i.e., they are commuting elements also satisfying
$\sigma_a^2=\sigma_b^2=0$.
In both Cases~(1) and~(2), this algebra has the decomposition
\begin{equation} \label{e:Nabdecomp}
  \cN(X) = \cN^\varnothing(X) \oplus \cN^a(X) \oplus \cN^{b}(X) \oplus \cN^{ab}(X)
  = \cN^\varnothing(X) \oplus \cN^\obs(X)
\end{equation}
where $\cN^\varnothing(X)$ is spanned by monomials with no factors of $\sigma$,
$\cN^a(X)$ is spanned by monomials containing a factor $\sigma_a$ but
no factor $\bar\sigma_b$ (respectively $\sigma_b$), analogously for $\cN^b(X)$,
and $\cN^{ab}(X)$ is spanned by monomials containing $\sigma_a\bar\sigma_b$ respectively $\sigma_a\sigma_b$.
Thus any $F \in \cN(X)$ can be written as
\begin{equation} \label{e:Fdecomp}
  F
  = F^\varnothing + F^\obs
  = 
  \begin{cases}
    F_\varnothing+\sigma_aF_a+\bar\sigma_bF_b+\sigma_a\bar\sigma_bF_{ab}, & \text{Case (1)}
    \\
    F_\varnothing+\sigma_aF_a+\sigma_bF_b+\sigma_a\sigma_bF_{ab}, & \text{Case (2)},
  \end{cases}
\end{equation}
with $F_\varnothing,F_a,F_b,F_{ab} \in \cN^\varnothing(X)$.  We denote
by $\pi_{\varnothing},\pi_{a},\pi_b$ and $\pi_{ab}$ the projections on
the respective components, e.g., $\pi_aF =\sigma_aF_a$, and
$\pi_\obs = \pi_a+\pi_b+\pi_{ab}$.  We will use superscripts instead
of subscripts in the decomposition when the factors of $\sigma$ are
included, e.g., $F^a= \sigma_aF_a$ and $F^\varnothing=F_\varnothing$.

We say that $F$ is $U(1)$ invariant if the number of generators with
a bar is equal to the number without a bar.  Explicitly, in Case~(1) this means
$F_{\varnothing} $ and $F_{ab}$ are $U(1)$ invariant, $F_{a}$ has one
more factor with a bar than without, and similarly for $F_{b}$. In
Case~(2) this means all of $F_{\varnothing},F_{a},F_{b}$ and $F_{ab}$
are $U(1)$ invariant (recall that
$\sigma_{a}=\bar\vartheta_{a}\vartheta_{a}$ and $\sigma_{b}=\bar\vartheta_{b}\vartheta_{b}$).
Denote by $\cN_{\rm sym}(X)$ the subalgebra of
$U(1)$ invariant elements.

For $F$ decomposed according to \eqref{e:Fdecomp} we define
\begin{equation} \label{e:Tphi-obs}
  \|F\|_{T_j(\ell_j)}
  = \|F_\varnothing\|_{T_j(\ell_j)}
  + \ellaj\|F_a\|_{T_j(\ell_j)}
  + \ellbj\|F_b\|_{T_j(\ell_j)}
  + \ellabj\|F_{ab}\|_{T_j(\ell_j)}
\end{equation}
where
\begin{equation} \label{e:ellsigma}
  \ellaj = \ellbj
  = \begin{cases}
    \ell_{j}^{-1}, & \text{Case (1)}\\
    \ell_j^{-2}, 
    & \text{Case (2)},
  \end{cases}
  \qquad
  \ellabj
  = \begin{cases}
    \ell_{j}^{-2}, & \text{Case (1)}\\
    \ell_j^{-2} \ell^{-2}_{j\wedge j_{ab}}, & \text{Case (2)}.
  \end{cases}
\end{equation}
In particular,
$\|\sigma_a\|_{T_j(\ell_j)} = \ellaj$ and $\|\sigma_a\sigma_b\|_{T_j(\ell_j)} = \ellabj$ and, in Cases (1) and (2), respectively,
\begin{equation} \label{e:a-marginal}
  \|\sigma_a\bar\psi_a\|_{T_j(\ell_j)}
  = \ellaj\ell_j = 1,
  \qquad
  \|\sigma_a\bar\psi_a\psi_a\|_{T_j(\ell_j)} = \ellaj\ell_j^2 =1 
\end{equation}
and, again in the two cases respectively, 
\begin{equation} \label{e:ab-marginal}
  \|\sigma_a{\bar \sigma_b}\bar\psi_x\psi_x\|_{T_j(\ell_j)} =
  \ellabj \ell_j^2 = 1, \qquad
  \|\sigma_a\sigma_b\bar\psi_x\psi_x\|_{T_j(\ell_j)} =
  \ellabj \ell_j^2 = \ell^{-2}_{j\wedge j_{ab}}.
\end{equation}
In both cases these terms do not change size under change of
scale, provided that $j\geq j_{ab}$ for the last term.
Thus they are \emph{marginal}.
As will be seen in Section~\ref{sec:obs}, see the paragraph
  following Lemma~\ref{lem:ZNN-obs}, the choices of $\ellaj$ and $\ellabj$
are appropriate to capture the leading behaviour of 
correlation functions.

The extended definition~\eqref{e:Tphi-obs} of the
$T_{j}(\ell_{j})$ norm satisfies the properties discussed in
Section~\ref{sec:norms}, with the exception of the generalisation of the monotonicity
estimate
$\|F_\varnothing\|_{T_{j+1}(2\ell_{j+1})} \leq
\|F_\varnothing\|_{T_{j}(\ell_{j})}$. Checking these properties
is straightforward by using the properties of the bulk norm, and, in
the case of the product property, using that
$\ell_{ab,j}\leq\ell_{a,j}\ell_{b,j}$ (recall \eqref{e:ellj}). Similar reasoning also
yields a weaker monotonicity-type estimate: by
\eqref{e:Tphi-obs}, \eqref{e:ellj}, and
monotonicity in the bulk algebra,
\begin{equation}
  \label{e:Tphi-obs-ub}
  \|F\|_{T_{j+1}(\ell_{j+1})}
  \leq \|F\|_{T_{j+1}(2\ell_{j+1})}
  \leq 16 L^{2(d-2)} \|F\|_{T_{j}(\ell_{j})} .
\end{equation}

\subsection{Localisation with observables}
\label{sec:locobs}

We combine the space $\cVbulk$ of bulk coupling constants
from Definition~\ref{def:cV} 
with the space $\cV^\obs$ of observable coupling constants from
Definition~\ref{def:cVobs} 
into
\begin{equation}
  \cV = \cV^\varnothing \oplus \cV^\obs.
\end{equation}
We extend the localisation operators $\Loc_{X,Y}$ from
Section~\ref{sec:Loc} to the amended Grassmann algebra
\eqref{e:Nabdecomp} as follows. As in the bulk setting, we will
focus on the key properties of the extended localisation operators.
The extension of $\Loc_{X,Y}$ is 
linear and block diagonal with respect to the decomposition
\eqref{e:Nabdecomp}, and so can be defined 
separately on each summand. 
On $\cN^\varnothing(X)$, the restriction $\Loc_{X,Y}$ is defined to
coincide with the operators from 
Proposition~\ref{prop:Loc}.
From now on we denote this restriction by $\Loc_X^\varnothing$
or $\loc^{\varnothing}_{X}$ if
we want to distinguish it from the extended version.  To define the
restriction $\Loc^{\obs}_{X,Y}$ of $\Loc_{X,Y}$ to $\cN^\obs(X)$, we
continue to employ the systematic framework from \cite[Section~1.7]{MR3332939},
as follows. Let $\loc^a$, $\loc^b$, and $\loc^{ab}$ be the localisation
operators from \cite[Definition~1.17]{MR3332939}
with maximal dimensions (for the Cases $(p=1)$ and $(p=2)$)
\begin{equation} \label{e:d+obs}
  d_+^{a} = d_+^b = \frac{p}{2}(d-2)
  ,
  \qquad
  d_+^{ab} = d-2.
\end{equation}
\noindent\emph{Case (1).}
For $\sigma_a F_a\in \cN^a(X)$ we set $\Loc_{X, Y} (\sigma_a F_a) = \sigma_a \loc^a_{X\cap\{a\}, Y\cap\{a\}} F_a$, and likewise for point~$b$.
For $\sigma_a\bar\sigma_b F_{ab} \in \cN^{ab}(X)$ we set $\Loc_{X, Y} (\sigma_a\bar\sigma_b F_{ab}) = \sigma_a \bar\sigma_b \loc^{ab}_{X\cap\{a,b\}, Y\cap\{a,b\}} F_{ab}$.

\smallskip\noindent\emph{Case (2).}
The definitions in Case~(2) are analogous,
but with $\bar\sigma_{b}$ replaced by $\sigma_{b}$.
\smallskip

The superscripts $\varnothing,a,b,ab$ are present to indicate that we have assigned
different maximal dimensions to the summands in~\eqref{e:Nabdecomp}.
We use the same choice of field dimensions $[\psi]=[\bar\psi]=(d-2)/2$
as in Section~\ref{sec:Loc}.
We note that $\loc^a_{X, \varnothing}=\loc^b_{X,\varnothing}=\loc^{ab}_{X,\varnothing}=0$.  
The main difference between these operators
and $\Loc^\varnothing$ is that
the expressions produced by $\loc^a$, $\loc^b$, $\loc^{ab}$ are local, i.e.,
supported near $a$ and $b$. A second difference is that
the maximal dimensions vary.

Before giving the general properties of the extended $\Loc$ we include an example in Case~(1).
  
\begin{example}
  Consider Case (1). Then:
    
    \smallskip
    \noindent
    (i) If  $F\in \cN(\Lambda)$ is a field monomial of degree greater
    than one or if $F$ has gradients in it,  then $\Loc^a_X F =
    \Loc^b_X F=0$. Here (and in the rest of this example) we do not
    count factors of $\sigma_{a}$ and $\sigma_{b}$ in the degree.
    If $F$ has degree greater than $2$ or is degree two and has gradients $\Loc^{ab}_X F =0$. 

  \smallskip
  \noindent
  (ii) If $F=  \sigma_a \bar\psi_x+ \bar{\sigma}_b \psi_y+  \sigma_a \bar{\sigma}_b  \bar\psi_x\psi_y$, then 
  \begin{align} \label{e:ex-Locobs}
    \Loc_X F &= \sigma_a \bar\psi_a 1_{a\in X}+ \bar{\sigma}_b \psi_b  1_{b\in X}\nnb
    &+  \sigma_a \bar{\sigma}_b  \left(\bar\psi_a\psi_a1_{a \in X, b \notin X} + \bar\psi_b\psi_b 1_{ b\in X, a\notin X} +\frac 12 [ \bar\psi_a\psi_a+ \bar\psi_b\psi_b] 1_{a, b\in X} \right).
  \end{align}
\end{example}

The next proposition summarises the key properties of
the operators $\Loc_{X,Y}$. As with Proposition~\ref{prop:Loc},
these properties follow from \cite{MR3332939}.
That the choice of maximal dimensions \eqref{e:d+obs} produce contractive estimates can intuitively be understood by
considering the marginal monomials. By \eqref{e:a-marginal} and \eqref{e:ab-marginal}, these are exactly the monomials
with dimensions $d_+^a=d_+^b$ respectively $d_+^{ab}$.

\begin{proposition} \label{prop:Loc-obs}
  For $L=L(d)$ sufficiently large there is a universal $\bar
  C>0$ such that: for $j<N$ and any small sets $Y \subset X \in \cS_j$,
  the linear maps
  $\Loc_{X,Y}^\obs \colon \cN^\obs(X^\square) \to \cN^\obs(Y^\square)$ have the following properties:
  \smallskip
  
  \noindent
  (i) They are bounded:
  \begin{equation} \label{e:Loc-obs-bd}
    \|\Loc_{X,Y}^\obs F\|_{T_{j}(\ell_{j})} \leq \bar C \|F\|_{T_j(\ell_{j})}.
  \end{equation}
\smallskip
\noindent
  (ii) For $j\geq j_{ab}$, the maps $\Loc_X^{\obs} \bydef \Loc_{X,X}^{\obs}\colon \cN^{\obs}(X^\square) \to \cN^{\obs}(X^\square)$ satisfy the contraction bound:
  \begin{equation} \label{e:Loc-obs-contract}
    \|(1-\Loc_X^\obs)F\|_{T_{j+1}(2\ell_{j+1})} \leq \bar C L^{-\dplus} 
    \|F\|_{T_j(\ell_j)}.
  \end{equation}
  Moreover, the bound~\eqref{e:Loc-obs-contract} holds also for $j<j_{ab}$ if $F^{ab}=0$.

  \smallskip
  \noindent
  (iii) If $X$ is the disjoint union of $X_1, \dots, X_n$ then
  $\Loc_X^{\obs} = \sum_{i=1}^n \Loc_{X,X_i}^{\obs}$.

  \smallskip
  \noindent
  (iv) For a block $B$ and polymers $X \supset B$,
  $\Loc_{X,B}^\obs F \in \cV^\obs(B)$ if $F \in \cN^{\obs}_{\rm sym}(X)$.
\end{proposition}

Properties (i)--(iii) follow from \cite{MR3332939} in the same way as
the corresponding properties in Proposition~\ref{prop:Loc} by making use
of the observation that
\begin{align}
  \|\sigma_a\|_{T_{j+1}(2\ell_{j+1})}
  &\leq
    2L^{d_+^a}\|\sigma_a\|_{T_{j}(\ell_{j})},  \qquad  \|\sigma_b\|_{T_{j+1}(2\ell_{j+1})} \leq
    2L^{d_+^b}\|\sigma_b\|_{T_{j}(\ell_{j})}, \\
  \|\sigma_{a}\sigma_b\|_{T_{j+1}(2\ell_{j+1})}
  &\leq 4L^{d_+^{ab}}
    \|\sigma_{a}\sigma_b\|_{T_{j}(\ell_{j})}, \qquad \text{if $j\geq j_{ab}$},
\end{align}
in Case (2) and analogously in Case (1).  These factors of $L^{d_+}$
correspond to the missing $L^{-d_+}$ factors in \eqref{e:Loc-obs-contract} as compared to  Proposition~\ref{prop:Loc}.
It only remains to verify (iv), i.e., to identify the image of
$\Loc_{X,B}^\obs$ when acting on $F \in \cN^{\obs}_{\rm
  sym}(X)$.

\smallskip
\noindent\emph{Case (1).}
By the choice of dimensions in its specification, the image of
$\sigma_{a}\loc^{a}$ is spanned by the local monomials
$\sigma_{a}, \sigma_{a}\bar\psi_{a}$, $\sigma_a\psi_{a}$.  The
condition of $U(1)$ invariance then implies that if
$\sigma_a F_a\in \cN^{a}_{\rm sym}(X)$ only the monomial
$\sigma_{a}\bar\psi_{a}$ is admissible.  The situation is analogous
for $\loc^b$.  Similarly,
$\sigma_{a}\bar\sigma_{b}\loc^{ab}$ 
has image spanned by 
$\sigma_{a}\bar\sigma_{b}$ and
$\sigma_{a}\bar\sigma_{b}\bar\psi_x\psi_x$ for $x\in \{a,b\}$ 
as well as further first order monomials with at most $(d-2)/2$
gradients, e.g., $\sigma_a\bar\sigma_b\nabla_{e_{1}}\psi_x$. Only
the even monomials $\sigma_{a}\bar\sigma_{b}$,
$\sigma_{a}\bar\sigma_{b}\bar \psi_{a}\psi_{a}$, and $\sigma_{a}\bar\sigma_{b}\bar\psi_{b}\psi_{b}$
are compatible with $U(1)$ symmetry.
In summary, $\Loc^\obs_{X,Y} F$ is contained in $\cV^\obs$ if $F\in
\cN^{\obs}_{\rm sym}(X)$. 

\smallskip
\noindent\emph{Case (2).}
By the choice of dimensions, in this case
$\sigma_{a}\loc^a$ has image spanned by the local monomials
$\sigma_{a},\sigma_{a}\bar\psi_a\psi_a$ 
as well as further first order monomials
with at most $d-2$ gradients,
and $U(1)$ symmetry implies that only the even terms $\sigma_{a}$ and
$\sigma_{a}\bar\psi_a\psi_a$ 
arise in the image if $F \in \cN_{\rm sym}(X)$. The analysis for
$\sigma_{b}\loc^b$ is analogous. Lastly, $\sigma_{a}\sigma_{b}\loc^{ab}$
has image spanned by $\sigma_{a}\sigma_{b}$ and the monomials
$\sigma_{a}\sigma_{b}\bar\psi_x\psi_x$ for $x\in\{a,b\}$ 
and first order monomials with at most $(d-2)/2$ gradients.
Again only the even monomials are compatible with $U(1)$ symmetry.

\subsection{Definition of the renormalisation group map with observables}
\label{sec:rgmap-ext}

In this section the renormalisation group map $\Phi_{j+1}
=\Phi_{j+1,N,m^{2}}$ is extended to
include the observable components (as in Section~\ref{sec:bulk}, we
omit the $N$ and $m^{2}$-dependence when there is no risk of
confusion). To this end, we now call the
renormalisation group map from Section~\ref{sec:step} the \emph{bulk
  component} and denote it by $\Phi^\varnothing_{j+1}$, and 
$\Phi_{j+1}=(\Phi^\varnothing_{j+1},\Phi^\obs_{j+1})$ will now refer to
the renormalisation group map extended to the algebra with
observables.  The map $\Phi^\obs_{j+1}$ is the \emph{observable component} of
the renormalisation group map.  This extension will be defined so that the
bulk components of $K_{j+1}$ and $V_{j+1}$ only depend on the
bulk components of $K_j$ and $V_j$.  In other words,
\begin{equation}
  \pi_\varnothing \Phi_{j+1}(V_j,K_j) = \Phi_{j+1}^\varnothing(\pi_\varnothing V_j,\pi_\varnothing K_j).
\end{equation}
On the other hand, the observable components $\Vobs_{j+1}$ and
$\Kobs_{j+1}$ will depend on both the observable and the bulk
components of $(V_j,K_j)$.  The observable component $\Phi^{\obs}_{j+1}$ is
upper-triangular in the sense
that the $a$ component $\pi_{a}\Phi^{\obs}_{j+1}(V_{j},K_{j})$
of $\Phi^{\obs}_{j+1}(V_j,K_j)$ only depends on $(\Vbulk_{j},\Kbulk_{j})$ and
$(V^a_{j},K^a_{j})$ but not on $(V^b_{j},K^b_{j})$ or $(V^{ab}_{j},K^{ab}_{j})$, and similarly
for the $b$ component.  The $ab$ component depends on all components
from the previous scale.
We will use an initial condition $V_{0}\in \cV$ and $K_{0}(X)=1_{X=\varnothing}$. 

We now give the precise definition of the observable component
of the renormalisation group map
$\Phi_{j+1}^\obs\colon (V_j,K_j)\mapsto (\uobs_{j+1}, \Vobs_{j+1}, \Kobs_{j+1})$.
For $j+1<N$, given $(V_j, K_j)$ and $B \in \cB_j$,
define $Q(B)$ and $J(B,X)$ as in
\eqref{e:Q-def}--\eqref{e:J-def} using the extended version of
$\Loc$ from Section~\ref{sec:locobs}. If $j+1=N$ set $Q=J=0$.
We let $Q^\obs(B)=\pi_\obs Q(B) $ and $J^\obs(B,X)=\pi_\obs J(B, X)$ denote
the observable components.  The new detail for the observable renormalisation group map
is that, to define $\Vobs_{j+1}$, we include the second order
contribution from $\Vobs_{j}$ in order to maintain better control on
the renormalisation group flow.  To this end, for $j+1\leq N$
and $B,B'\in\cB_{j}$, let 
\begin{equation} \label{e:Pobs-def}
  \begin{split}
    P^\obs(B,B')
    &=
    \frac12
      \E_{C_{j+1}}\qB{\theta (\Vobs_j(B)-Q^\obs(B)); \theta (\Vobs_j(B')-Q^\obs(B'))},
    \\
    P^\obs(B)
    &= \sum_{B' \in \cB_j} P^{\obs}(B,B')
    .
  \end{split}
\end{equation}
The following observations will be useful later. Since
$V^\obs(B), Q^\obs(B) \in \cV^\obs(B)$, the sum over $B'$ contains at
most two non-zero terms, corresponding to the blocks
containing $a$ and $b$.  Since the covariance matrix $C_{j+1}$ has the
finite range property \eqref{e:C-range}, also $P^\obs(B,B')=0$ for
$B\neq B'$ if $|a-b|_\infty \geq \frac12 L^{j+1}$.  Finally, if $a$
and $b$ are not in the same block, then $P^\obs(B,B)=0$ since
the source fields square to zero.

With these definitions in place, $ \uobs_{j+1}+\Vobs_{j+1}$ is defined
in the same way as $ u_{j+1}+ V_{j+1}$ with the addition of the second
order term $P^\obs$, and $\Kobs_{j+1}$ is then defined in the same way
as $K_{j+1}$:
\begin{definition} \label{def:V+obs}
The map $(V_j,K_j) \mapsto ({ \uobs_{j+1}},\Vobs_{j+1})$ is defined, for $B\in \cB_j$, by
\begin{equation}
  \label{E:Qobs}
  \uobs_{j+1}(B)+
  \Vobs_{j+1}(B) = 
  \E_{C_{j+1}}\qB{\theta (\Vobs_j(B) - Q^\obs(B))} - P^\obs(B)
\end{equation}
where $\uobs_{j+1}$ consists of
all monomials that do not contain factors of
$\psi$ or $\bar\psi$. Explicitly,
  \begin{align}
    \uobs_{j+1}=
    \begin{cases}
      -\sigma_a\bar\sigma_b q_{j+1}, &\text{Case (1)}, \\ 
    -\sigma_a\sigma_b q_{j+1} - \sigma_a \gamma_{a,j+1}
    -\sigma_b\gamma_{b,j+1}, &\text{Case (2)}.
    \end{cases}
  \end{align}
The map $(V_j,K_j)\mapsto \Kbulk_{j+1}+\Kobs_{j+1}$ is defined by the same formula
as in Definition~\ref{def:K+} except that $V^\varnothing$ and $u^\varnothing$
are replaced by 
$V=V^\varnothing+V^\obs$ and $u=u^\varnothing+u^\obs$.
\end{definition}

Propositions~\ref{prop:ZVK}
and~\ref{prop:consistency} also hold for
this extended definition of the renormalisation group map. The proofs
are the same as presented in Section~\ref{sec:step-def}.

\subsection{Estimates for the renormalisation group map with observables}
\label{sec:tyler-renorm-group}

In this section, the $O$-notation refers to scale $j+1$ norms, i.e.,
for $F,G \in \cN(\Lambda)$, we write $F=G+O(t)$ to denote that
$\|F-G\|_{T_{j+1}(\ell_{j+1})} \leq O(t)$. We use
  $\|V_{j}\|_{j}$ and $\|K_{j}\|_{j}$ defined in~\eqref{e:V-norm}--\eqref{e:K-norm},
  with the understanding that the right-hand sides of the definitions
  are the norm~\eqref{e:Tphi-obs} which accounts for source fields.

\begin{theorem} 
  \label{thm:step-obs} 
  Under the assumptions of Theorem~\ref{thm:step}, if also
  $\|\Vobs_j\|_j+\|\Kobs_j\|_j \leq \epsilon$ and $\uobs_j=0$, then for
  $j+1<N$ the observable components of the renormalisation group map
  $\Phi^\obs_{j+1}$ satisfy
  \begin{gather}
    \label{e:V+obs-bd}
    \uobs_{j+1}(\Lambda) +\Vobs_{j+1}(\Lambda)
    =  \E_{C_{j+1}}\qB{\theta \Vobs_j(\Lambda)} - \frac12 \E_{C_{j+1}}\qB{\theta \Vobs_j(\Lambda);\theta\Vobs_j(\Lambda)}
      + O(L^{2(d-2)}\|\Kobs_j\|_j) 
                      \\
    \label{e:K+obs-bd}
    \|\Kobs_{j+1}\|_{j+1} \leq O(L^{-\dplus}+A^{-\eta})\|\Kobs_j\|_j
    + O(A^{\nu}) (\|\Vbulk_j\|_j +\|K_{j}\|_{j})      (\|V_j\|_j+\|K_j\|_j),
  \end{gather}
  provided that $K^{ab}_j(X)=0$ for $X\in \cS_j$ if
  $j<j_{ab}$. Both $\eta=\eta(d)$ and $\nu=\nu(d)$ are
  positive geometric constants.
  For $j+1=N$, $\Phi^\obs_{N}$ is bounded.
\end{theorem}

The first estimate in the theorem expresses that the evolution
of $\uobs+\Vobs$ is given by second-order perturbation theory, plus
a higher order remainder due to $\Kobs$. The second estimate states
that $\Kobs$ is contracting (for $L$ and $A$ large), up to error
terms at most as large as the bulk coupling constants $\Vbulk$ and
$K=\Kbulk+\Kobs$.  The additional factor
$\|V_j\|_j+\|K_j\|_j \geq \|\Vobs_j\|_j$ will be small (but of order
$1$) while all other coordinates will be exponentially small in $j$.
Indeed, as a consequence of the above theorem,
Proposition~\ref{prop:flow-obs} below states that if the bulk flow
$(\Vbulk,\Kbulk)$ is as constructed in Section~\ref{sec:flow} then
$\Vobs$ remains bounded while $\Kobs$ goes to $0$ exponentially fast.

The proof of the theorem follows that of Theorem~\ref{thm:step} closely,
with improvements for the leading terms
that allow for $\Vobs$ to be tracked to second order.
It is given in the remainder of this subsection.
The reader may again wish to skip the details of this proof on a first reading
and proceed to the application of these estimates Section~\ref{sec:flow-obs}.

\subsubsection{Coupling constants}

We first give a bound on
$u^{\obs}_{j+1}(\Lambda)+ \Vobs_{j+1}(\Lambda)$.   By
Proposition~\ref{prop:Loc-obs} (iii), 
\begin{equation}
  Q^{\obs}(\Lambda) =
  \sum_{X\in\cS_{j}}\Loc^{\obs}_{X}
  K_{j}(X).
\end{equation}
Since only small sets $X$ that contain $a$ or $b$ contribute,
Proposition~\ref{prop:Loc-obs} (i) implies
\begin{equation}
  \label{e:Qobsbd}
  \|Q^{\obs}(\Lambda)\|_{T_{j}(\ell_{j})}\leq  O(1)
  \|\Kobs_j\|_{j}. 
\end{equation} 
By algebraic manipulation, the product property, 
that $\E_{C_{j+1}}\theta$ is a contraction,
\eqref{e:Tphi-obs-ub}, and~\eqref{e:Qobsbd}, 
\begin{align}
  P^{\obs}(\Lambda)
  &= \frac{1}{2}\E_{C_{j+1}}\qB{\theta V^{\obs}_{j}(\Lambda);\theta V^{\obs}_{j}(\Lambda)}
    +\E_{C_{j+1}}\qB{\theta Q^{\obs}_{j}(\Lambda);\theta
    (V^{\obs}_{j}(\Lambda)+\frac12 Q^\obs_j(\Lambda))}
    \nnb
  &= \frac{1}{2}\E_{C_{j+1}}\qB{\theta V^{\obs}_{j}(\Lambda);\theta V^{\obs}_{j}(\Lambda)}
    + O(L^{4(d-2)}\|\Kobs_{j}\|_{j} (\|\Vobs_{j}\|_{j}+\|\Kobs_{j}\|_{j})).
\end{align}
Putting these pieces together establishes~\eqref{e:V+obs-bd} as
$L^{2(d-2)}(\|V_{j}\|_{j}+\|K_{j}\|_{j})\leq 1$ if
$\epsilon=\epsilon(L)$ is small enough.
An immediate consequence is
\begin{align}
  \label{e:uobs+}
  \|u^{\obs}_{j+1}(\Lambda)\|_{T_{j+1}(\ell_{j+1})} 
  &\leq O(\|\Vobs_{j}\|_{j} + L^{2(d-2)}\|\Kobs_{j}\|_{j}), \\
  \label{e:vobs+}
  \|\Vobs_{j+1}(\Lambda)\|_{T_{j+1}(\ell_{j+1})} 
  &\leq O(\|\Vobs_{j}\|_{j} + L^{2(d-2)}\|\Kobs_{j}\|_{j}).
\end{align}
The same bounds hold with $\Lambda$ replaced by any $X\in\cP_j$.
These will be used in the following analysis.

\subsubsection{Small sets}

The most significant improvement in the analysis concerns small sets,
which we now analyse to second order. To simplify notation, we write
\begin{equation}
  \hVobs_j = \Vobs_j-Q^{\obs}, \qquad \tVobs_{j+1} = \uobs_{j+1}+\Vobs_{j+1}.
\end{equation}

\begin{lemma}
  \label{lem:Pjj+1}
  For any $B,B' \in \cB_j$,
  \begin{equation} \label{e:Pjj+1}
    P^\obs(B,B') = \frac12
    \E_{C_{j+1}}\qB{\theta \hVobs_j(B)\theta \hVobs_j(B')}
    -\frac12 \tVobs_{j+1}(B)\tVobs_{j+1}(B')
    .
  \end{equation}
\end{lemma}
\begin{proof}
  Note that $P^{\obs}(B,B')=\frac{1}{2}\E_{C_{j+1}}\q{\theta \hVobs_{j}(B);\theta\hVobs_{j}(B')}$.
  Since it is quadratic in $\hVobs_j\in\cV^\obs$, $P^\obs(B,B')$ can
  only contain monomials with a factor of $\sigma_a\bar{\sigma}_b$
  (Case (1)) or $\sigma_a{\sigma_b}$ (Case (2)) because
  $\sigma_a^2=\sigma_b^2=\bar{\sigma}_{b}^{2}=0$.  Similarly, for any
  $W \in \cV^\obs$ and $B,B',B'' \in \cB_j$, it follows that
  $P^\obs(B,B') W (B'')=0$.  The claim follows as this implies that
  $(\E_{C_{j+1}}[\theta \hVobs_j(B)])(\E_{C_{j+1}}[\theta \hVobs_j(B')])$
  is the same as
  \begin{equation}
    \pB{\E_{C_{j+1}}\qB{\theta \hVobs_j(B)} -P^\obs(B)}\pB{\E_{C_{j+1}}\qB{\theta\hVobs_j(B')} -P^\obs(B')}
    = \tVobs_{j+1}(B)\tVobs_{j+1}(B'). \qedhere
  \end{equation}
\end{proof}

The next lemmas are analogues of
Lemmas~\ref{lem:smallset-V}--\ref{lem:smallset-V1} that apply to
the observable components.
We begin with the replacement for Lemma~\ref{lem:smallset-V1}.
For $B\in \cB_j$, recall $\bar{B}$ denotes the scale $j+1$-block containing $B$.

\begin{lemma}
\label{lem:smallset-Vobs1}
Suppose that 
$\|V_j\|_j+\|K_{j}\|_j\leq 1$.
  Then for any $X \in \cP_j$, denoting by $n\in\{0,1,2\}$ the number of $B\in \cB_j(X)$ containing $a$ or $b$,
  \begin{equation} \label{e:smallset-Vobs-1}
    \normB{\pi_\obs \E_{C_{j+1}}\qB{(\delta I)^X}}_{T_{j+1}(\ell_{j+1})} \leq
    O
    (\|V_j\|_j+ L^{2(d-2)}\|K_j\|_j)^n
    (O(\|\Vbulk_j\|_j+\|\Kbulk_j\|_j))^{|\cB_j(X)|-n}
    .
  \end{equation}
  
  For any $B\in \cB_j$ such that $\bar B$ contains at most one of $a$ and $b$,
  \begin{equation}
    \label{E:Blockobs}
    \normB{\pi_\obs \E_{C_{j+1}}\qB{\delta I(B)}}_{T_{j+1}(\ell_{j+1})} \leq
    O(L^{2(d-2)}\|\Kobs_j\|_j)+O(\|V_j\|_j+L^{2(d-2)}\|K_j\|_j)(\|\Vbulk_j\|_j+\|\Kbulk_j\|_j).
  \end{equation}
  Moreover if $|a-b|_\infty \geq \frac12 L^{j+1}$ then for any $X\in\cP_j$ with $|\cB_j(X)|=2$, 
  \begin{equation} \label{e:deltaI-obs-4}
    \normB{\pi_\obs \E_{C_{j+1}}\qB{(\delta I)^X}}_{T_{j+1}(\ell_{j+1})}
    \leq O((\|V_j\|_j + L^{2(d-2)}\|K_j\|_j)(\|\Vbulk_j\|_j + L^{2(d-2)}\|K_j\|_j)).
  \end{equation}
\end{lemma}
\begin{proof}
  Throughout the proof, we will use that 
  for $V$ representing either $V_j$ or $u_{j+1}+V_{j+1}$ one has
  \begin{align} 
    \pi_\obs e^{-V(B)}
    &= \pi_\obs (e^{-V^\varnothing(B)-V^\obs(B)})\nnb
    &= -V^\obs(B) + \frac12 V^\obs(B)^2 + O(\|V^\obs(B)\|_{T_{j+1}(\ell_{j+1})}\|V^\varnothing(B)\|_{T_{j+1}(\ell_{j+1})}),
      \label{e:expVobs2-expansion}
  \end{align}
  where we recall that the $O$-notation refers to terms whose $T_{j+1}(\ell_{j+1})$-norms are bounded by the indicated numbers, up to multiplicative constants.
  For both of the choices for $V$, one has 
$\|V^\varnothing (B)
  \|_{T_{j+1}(\ell_{j+1})}\leq \|V_j^\varnothing\|_j+O(\|K_j^\varnothing\|_j) \leq O(1)$ by  \eqref{E:vj}
  and $\|V^\obs (B)\|_{T_{j+1}(\ell_{j+1})}\leq O(\|V_j^\obs\|_j+L^{2(d-2)}\|K_j^\obs\|_j)$
  by using \eqref{e:uobs+}--\eqref{e:vobs+} (with $B$ instead of $\Lambda$).

  To show \eqref{e:smallset-Vobs-1}, for each $B\in\cB_j$, write
  $\delta I(B) = \pi_{\varnothing}\delta I(B) + \pi_{\obs}\delta I(B)$
  and expand the product defining $(\delta I)^{X}$ 
  using that there are $n$ blocks $B$ for which $\pi_\obs \delta I(B)\neq 0$.
  The claim then follows since 
  $\|\pi_{\varnothing} \delta I(B)\|_{T_{j+1}(\ell_{j+1})} \leq O(\|\Vbulk_j\|_j + \|\Kbulk_j\|_j)$
  by Lemma~\ref{lem:smallset-V1} and
  $\|\pi_{\obs}\delta I(B)\|_{T_{j+1}(\ell_{j+1})}=O(\|\Vobs_{j}\|_{j}+L^{2(d-2)}\|\Kobs_{j}\|_{j})$
  which follows from the previous paragraph (as the doubling
    map commutes with $\pi_{\obs}$).

  For the bound \eqref{E:Blockobs}, using that $B$ can contain only
  $a$ or $b$ by assumption and that source fields square to zero, 
  one has $V^\obs(B)^2=0$  for $V$ either $V_j$ or $u_{j+1}+V_{j+1}$.
  Thus \eqref{e:expVobs2-expansion} simplifies to
  \begin{equation} \label{e:expVobs-expansion}
    \pi_\obs e^{-V(B)} = \pi_\obs (e^{-V^\varnothing(B)-V^\obs(B)})
    = -V^\obs(B) + O(\|V^\obs(B)\|_{T_{j+1}(\ell_{j+1})}\|V^\varnothing(B)\|_{T_{j+1}(\ell_{j+1})}).
  \end{equation}
  Observe that $P^\obs(B)=0$ since $\bar B$ contains only one of $a$
  and $b$, see the remark below \eqref{e:Pobs-def}. As a result,
  \eqref{E:Qobs} and 
  the above show that the term linear in $V_j^\obs(B)$ in $\pi_\obs \E_{C_{j+1}}\delta I(B)$ cancels
  in expectation. The claim \eqref{E:Blockobs} then follows from
  $\|\E_{C_{j+1}}\theta Q^\obs(B)\|_{T_{j+1}(\ell_{j+1})}=O(L^{2(d-2)}\|K_j^\obs\|_j)$
  by \eqref{e:Qobsbd} and \eqref{e:Tphi-obs-ub},
  and bounding the quadratic terms using \eqref{E:vj} and
  \eqref{e:uobs+}--\eqref{e:vobs+} as  below \eqref{e:expVobs2-expansion}.

  For the final assertion \eqref{e:deltaI-obs-4}, we first show that
  $\E_{C_{j+1}}(\pi_\obs\delta I)^X=L^{4(d-2)}O(\|\Vbulk_j\|_j +
  \|K_j\|_j)(\|\Vobs_j\|_j + \|\Kobs_j\|_j)$, where we emphasise that
  $\pi_\obs$ is inside the product over $X$.  To see this bound,
  let $X = B\cup B'$, and note that $V^\obs(B)$ and
  $V^\obs(B')$ are either $0$ or polynomials in $\psi_a,\bar\psi_a$
  and $\psi_b,\bar\psi_b$ respectively.  Since by assumption
  $C_{j+1}(a, b)=0$, 
    $\E_{C_{j+1}}\theta V^{\obs}(B) V^{\obs}(B') =0$. Hence a
  nonvanishing contribution to $\E_{C_{j+1}}(\pi_\obs\delta I)^X$
  involves at least one factor $\Vbulk$ from the expansion of the
  $\delta I$ by \eqref{e:expVobs2-expansion}. The factor of
  $L^{4(d-2)}$ arises from applying \eqref{e:Tphi-obs-ub}.
  The estimate \eqref{e:deltaI-obs-4} now follows similarly to
  the previous cases: 
  \begin{equation}
    \pi_\obs \E_{C_{j+1}}( \delta I)^X
    =\pi_\obs \E_{C_{j+1}}( \pi_\varnothing \delta I + \pi_\obs \delta I)^X
    =O((\|V_j\|_j + L^{2(d-2)}\|K_j\|_j)(\|\Vbulk_j\|_j + L^{2(d-2)}\|K_j\|_j))
  \end{equation}  
  as the cross terms with one factor $\pi_{\obs}$ and one factor
  $\pi_{\varnothing}$ satisfy this bound as above.
\end{proof}

Next we replace Lemma~\ref{lem:smallset-V}. Unlike before we
explicitly consider terms arising from two blocks,
in order to obtain a cancellation up to a third order error in $\Vobs$.
Indeed, note that the right-hand side of   \eqref{e:smallset-Vobs-2} involves
$\|V_j\|_j\|\Vbulk_j\|_j \leq (\|\Vbulk_j\|_j+\|\Vobs_j\|_j)\|\Vbulk_j\|_j$ but no term $\|\Vobs_j\|_j^2$,
and that $\|\Vbulk_j\|_j$ is exponentially small in $j$ along the flow while $\|\Vobs_j\|_j$ is of order $1$. The $K$-terms are higher order.

\begin{lemma} \label{lem:smallset-Vobs}
  Suppose that $\|V_j^\varnothing \|_j+\|K_{j}^\varnothing\|_j\leq\epsilon$
  and $\|\Vobs_{j}\|_{j}+\|\Kobs_{j}\|_{j}\leq \epsilon$.
  Then for $B\in \cB_j$, 
  \begin{multline}
  \label{e:smallset-Vobs-2}
  \norma{\pi_\obs \E_{C_{j+1}}\qa{\delta I(B)+ \frac12
      \sum_{B' \neq  B, B'\subset \bar{B}}
   \delta I(B) \delta I(B') + \theta Q(B)}}_{T_{j+1}(\ell_{j+1})}
  \\
  = O(L^{4d}
  (\|V_{j}\|_{j}+\|K_{j}\|_j)
  (\|\Vbulk_j\|_j+\|K_{j}\|_{j})).
  \end{multline}
\end{lemma}

\begin{proof}
  Recall $\tVobs_{j+1}=\uobs_{j+1}+\Vobs_{j+1}$. Using 
  \eqref{E:Qobs} to re-express $\E_{C_{j+1}}[\theta Q^\obs(B)]$,
  the term inside the norm on the left-hand side of \eqref{e:smallset-Vobs-2} equals
  \begin{equation}
    \label{e:l510a}
    \pi_\obs \E_{C_{j+1}}\qa{ \delta I(B)
      + \frac12  \sum_{ B' \neq  B, B'\subset \bar{B}} \delta I(B) \delta I(B')
    }
    + \E_{C_{j+1}}\qB{\theta \Vobs_j(B)}
    - \sum_{B'} P^\obs(B,B')
    - \tVobs_{j+1}(B).
  \end{equation}
  We start with the one block terms $B'=B$ in~\eqref{e:l510a}. Using Lemma~\ref{lem:Pjj+1} to rewrite $P(B,B)$
  and since $\delta I(B)=\theta
  e^{-V_j(B)}-e^{-(V_{j+1}+u_{j+1})(B)}$, these terms are
  \begin{multline} 
    \pi_\obs \E_{C_{j+1}}\theta  \qa{ e^{-V_j(B)}-1
      + \Vobs_j(B) - \frac12 \hVobs_j(B)^2}
    \\
      -\pi_\obs \qa{ e^{-(V_{j+1}+u_{j+1})(B)}-1+
        \tVobs_{j+1}(B)- \frac12 \tVobs_{j+1}(B)^2}. 
  \end{multline}
 
  To estimate these terms, first note that if 
  $V=V_{j+1}+u_{j+1}$ then~\eqref{e:V+obs-bd} and its 
  consequences~\eqref{e:uobs+}--\eqref{e:vobs+} imply
  $\|\Vbulk\|_{T_{j+1}(\ell_{j+1})}\leq 1,
  \|\Vobs\|_{T_{j+1}(\ell_{j+1})}\leq 1$.
  This bound also holds for $V=V_j$ provided $\epsilon$ is
  sufficiently small by~\eqref{e:Tphi-obs-ub},  we then have for $V=V_j$ or $V=u_{j+1}+V_{j+1}$,
  \begin{align}
    \nonumber
    \pi_{\obs}e^{-V(B)}
    &= \pi_{\obs}(e^{-\Vobs(B)} + (e^{-\Vbulk(B)}-1)e^{-\Vobs(B)}) \\
    \label{e:eVbound}
    &= - \Vobs(B) + \frac12 \Vobs(B)^2 +
    O((\|\Vobs_{j}\|_{j}+L^{2(d-2)}\|\Kobs_{j}\|)
      (\|\Vbulk_{j}\|_j+\|\Kbulk_{j}\|_j )),
  \end{align}
  where we have used $\Vobs(B)^{3}=0$,
  and in the case
  $V=u_{j+1}+V_{j+1}$, 
  \eqref{E:vj} 
  to control
  $\|u^{\varnothing}_{j+1}+\Vbulk_{j+1}\|_{j+1}$ in terms of
  $\|\Vbulk_j\|_j+\|\Kbulk_j\|_j$
  and \eqref{e:vobs+} to control $\|u^\obs_{j+1}+\Vobs_{j+1}\|_{j+1}$ similarly.
  Using also
  \begin{equation}
    \hVobs_j(B)^2
    = (\Vobs_j(B)-Q(B))^2
    = \Vobs_j(B)^2 + O(L^{4(d-2)}\|\Kobs_j\|_j (\|\Vobs_j\|_j +
    \|\Kobs_j\|_j)), 
  \end{equation}
  by the product property, \eqref{e:Qobsbd},
  \eqref{e:Tphi-obs-ub}, and the assumed norm bounds, the estimate for
  the one block terms follow.

  Recall that $P^\obs(B,B')=0$ unless $a, b$ are each in one
  of the two blocks. Thus for $B' \neq B$ the two block terms
  in~\eqref{e:l510a}  are,  by Lemma~\ref{lem:Pjj+1},
  \begin{equation}
    \frac12 \pi_\obs \pbb{\E_{C_{j+1}} \qB{\delta I(B)\delta I(B')}
      -  \E_{C_{j+1}}\qB{\theta \hVobs_j(B)\theta \hVobs_j(B')} + \tVobs_{j+1}(B)  \tVobs_{j+1}(B')}
      .
  \end{equation}
  We start by rewriting this in a more convenient form.
  Let $\delta V_j^{\obs} = \theta \hat V_j^{\obs}-\tilde V_{j+1}^{\obs}$. 
By \eqref{E:Qobs},
$\E_{C_{j+1}}\theta \hat V^\obs_j  
= \tVobs_{j+1} + P^\obs = \tVobs_{j+1} + 
 O(\sigma_a\sigma_b)$,
 where $O(\sigma_a\sigma_b)$ denotes a 
 monomial containing a factor $\sigma_a\bar\sigma_b$ in Case~(1) or a factor $\sigma_a\sigma_b$ in Case~(2).
 Since all terms in $\delta \Vobs_j$ contain a source field (that is, a
 $\sigma$-factor) and source fields square to zero, we obtain
  \begin{align}
    \E_{C_{j+1}}\qB{\delta \Vobs_j(B)\delta \Vobs_j(B')}
    &=
      \E_{C_{j+1}}\qB{\theta \hVobs_j(B)\theta\hVobs_j(B')}
      + \tVobs_{j+1}(B)\tVobs_{j+1}(B')
      \nnb
      &\qquad
        -\tVobs_{j+1}(B) \E_{C_{j+1}}\qB{\theta \hVobs_j(B')} -\tVobs_{j+1}(B') \E_{C_{j+1}}\qB{\theta \hVobs_j(B)} \nnb
          &=
            \E_{C_{j+1}}\qB{\theta \hVobs_j(B)\theta\hVobs_j(B')} - \tVobs_{j+1}(B)\tVobs_{j+1}(B')
            .
  \end{align}
  Therefore we need to estimate
    \begin{equation}
    \frac12 \pi_\obs \E_{C_{j+1}} \qB{\delta I(B)\delta I(B')}
    -      \frac12  \E_{C_{j+1}}\qB{\delta \Vobs_j(B)\delta \Vobs_j(B')}
      .
  \end{equation}
  First write
  \begin{equation}
    \pi_\obs[ \delta I(B)\delta I(B')]
    =
    \pi_\obs \delta I(B)\pi_\obs \delta I(B')
    +
    \pi_\obs \delta I(B)\pi_\varnothing \delta I(B')
    +
    \pi_\varnothing \delta I(B)\pi_\obs \delta I(B') .
  \end{equation}
  The second and third terms on the right-hand side are $O((\|\Vobs_{j}\|_{j}+L^{2(d-2)}\|\Kobs_{j}\|_{j})(\|\Vbulk_j\|_j+\|\Kbulk_j\|_j))$
  using Lemma~\ref{lem:smallset-V1}
  for $\pi_\varnothing \delta I$ and $\|\pi_{\obs}\delta I(B)\|_{T_{j+1}(\ell_{j+1})}
  =O(\|\Vobs_{j}\|_{j}+L^{2(d-2)}\|\Kobs_{j}\|_{j})$ by \eqref{e:vobs+}.
  Using~\eqref{e:eVbound}, the term $\pi_\obs \delta I(B)\pi_\obs \delta I(B')$ can be estimated as
  \begin{multline}
    \pi_\obs
    (\delta V_j(B) -\frac12 (\theta V_j(B)^2-\tilde V_{j+1}(B)^2))
    \pi_\obs (\delta V_j(B') -\frac12 (\theta V_j(B')^2-\tilde V_{j+1}(B')^2))
    \\
    \qquad\qquad\qquad\qquad\qquad +
    O((\|\Vobs_{j}\|_{j}+L^{2(d-2)}\|\Kobs_{j}\|)(\|\Vbulk_j\|_j+
    \|\Kbulk_j\|_j))
    \\ =
    \delta \Vobs_j(B)
    \delta \Vobs_j(B') 
    +
    O((\|\Vobs_{j}\|_{j}+L^{2(d-2)}\|\Kobs_{j}\|)(\|\Vbulk_j\|_j+
    \|\Kbulk_j\|_j)),
  \end{multline}
  since $\sigma_a^2=\sigma_b^2=\bar{\sigma}_b^2=0$. The factor $L^{4d}$ is a convenient  common bound.
\end{proof}

The next lemma replaces Lemma~\ref{lem:smallset-K} on the observable components.
\begin{lemma} \label{lem:smallset-Kobs}
  For any $U \in \cC_{j+1}$,  if $K^{ab}_{j}(Y)=0$ for all $Y\in
    \cS_j$ and all $j<j_{ab}$, then 
  \begin{equation}
    \sum_{X \in \cS_j: \bar X=U} \normB{ \E_{C_{j+1}}\qB{\theta (1-\Loc_X^\obs) \Kobs _{j} (X)}}_{T_{j+1}(\ell_{j+1})} = O(L^{-\dplus}) \|\Kobs\|_j.
  \end{equation}
\end{lemma}

\begin{proof}
  The proof is the same as that of Lemma~\ref{lem:smallset-K} except
  for the following observation.  The sum over $X \in \cS_j$ that
  contributes a factor $O(L^d)$ in the proof of
  Lemma~\ref{lem:smallset-K} only contributes $O(1)$ on the observable
  components because for these only the small sets containing $a$ or
  $b$ contribute.  Thus the bound for $\Loc^\obs$ from
  Proposition~\ref{prop:Loc-obs}, which lacks a factor $L^{-d}$
  compared to the bound for $\Loc^\varnothing$, produces the same
  final bound.
\end{proof}

\begin{proof}[Proof of Theorem~\ref{thm:step-obs}]
  The proof is analogous to that of Theorem~\ref{thm:step}, and we
  proceed in a similar manner, by beginning with the coupling constants and then
  an estimate of
  $\pi_{\obs}\cL_{j+1}(U)$, where $\cL_{j+1}(U)$ is defined by the formula
  \eqref{e:K-linear} but with the extended coordinates introduced in
  Section~\ref{sec:rgmap-ext}.

  For the coupling constants, i.e., the analogue of Section~\ref{sec:K-coupling}, 
  the bound \eqref{E:Q} gets replaced by
  \eqref{e:Qobsbd} which gives
  $\|Q^{\obs}(B)\|_{T_{j+1}(\ell_{j+1})}\leq O(L^{d-2}\|\Kobs_j\|)$,  and
  we also have
  $\|u_{j+1}^{\obs}\|_{j+1}+\|\Vobs_{j+1}\|_{j+1}\leq
  O(\|\Vobs_j\|+L^{d-2}\|K^{\obs}_j\|_j)$ by
  \eqref{e:uobs+}--\eqref{e:vobs+}.

  Note that the terms of $\cL_{j+1}(U)$ are of the form
  $\sum_{F}e^{-V_{j+1}(U\setminus X)+u_{j}|X|}F$. We will use that
  \begin{equation}
    \pi_{\obs}(e^{-V_{j+1}(U\setminus X)+u_{j}|X|}F) =
    (\pi_{\obs}e^{-V_{j+1}(U\setminus X)+u_{j}|X|}) \pi_{\varnothing}F +
    \pi_{\obs}( e^{-V_{j+1}(U\setminus X)+u_{j}|X|}\pi_{\obs}F).
  \end{equation}
  We first explain how to estimate the sum arising from the first
  term, which only contributes to the second term in the
  estimate. The estimation of the terms $\pi_{\varnothing}F$ is
  exactly as in Section~\ref{sec:step}. Estimating
  $\pi_{\obs}e^{-V_{j+1}(U\setminus X)+u_{j}|X|}$, and the resultant
  sum over $F$, requires a replacement of
  Lemma~\ref{lem:pert12}. For this it suffices to note that
  $\| \pi_{\obs}e^{-V_{j+1}(U\setminus
    X)+u_{j+1}|X|}\|_{T_{j+1}(\ell_{j+1})}\leq
  (\|\Vobs_{j}\|_{j}+L^{2(d-2)}\|\Kobs_{j}\|_{j})
  2^{|\cB_{j}(X)|}$. This estimate follows by the product property,
  \eqref{e:uobs+}--\eqref{e:vobs+}, and arguing as in the proof of
  Lemma~\ref{lem:pert12}. Hence this term is bounded by
  $O(\|\Vobs_{j}\|_{j}+L^{2(d-2)}\|\Kobs_{j}\|_{j})(\|\Kbulk_{j}\|_{j} +
  A^{\nu}(\|\Kbulk_{j}\|^{2}_{j}+\|\Vbulk_{j}\|^{2}_{j}))$. 
  
  Next we explain how to estimate
  $e^{-V_{j+1}(U\setminus X)+u_{j}|X|}\pi_{\obs}F$. The prefactor is
  at most $2^{|\cB_{j}(X)|}$, i.e., the analogue of
  Lemma~\ref{lem:pert12} applies when $\Vbulk, u^{\varnothing}$ and $\Kbulk$ are
  replaced by $V, u$ and $K$ if $\epsilon$ is small enough, and it suffices to estimate $\pi_\obs F$.
  
  Consider the small set contributions to   $\cL_{j+1}(U)$,
  i.e., the analogue of Section~\ref{sec:K-smallset}.
  As stated
  previously, Lemma~\ref{lem:smallset-K} is replaced with
  Lemma~\ref{lem:smallset-Kobs} whereas Lemmas~\ref{lem:smallset-Vobs}
  and~\ref{lem:smallset-Vobs1} replace Lemmas~\ref{lem:smallset-V}
  and~\ref{lem:smallset-V1}.
  In detail, in the analogue of \eqref{e:Kblock} we now also include
    quadratic terms in $\delta I$, i.e., we replace \eqref{e:Kblock} by
  \begin{multline} \label{e:Kblock-obs}
    \pi_\obs\E_{C_{j+1}}\qa{\theta K_j(B) 
      + 
      \delta I(B)
      + \frac12   { \sum_{ B' \neq  B, B'\subset \bar{B}}}
      \delta I(B) \delta I(B')
      - 
      \theta J(B,B)}
    \\
    =\pi_\obs\E_{C_{j+1}}\qB{\theta (1-\Loc_B)K_j(B)}
    + \pi_\obs\E_{C_{j+1}}\qa{\delta I(B) + \frac12  {\sum_{ B' \neq  B, B'\subset \bar{B}}}     \delta I(B) \delta I(B')
    + \theta Q(B)},
  \end{multline}
  with the corresponding analogue of \eqref{e:Ksmallset} then being
  (for $X\in \cS_{j}\setminus \cB_{j}$)
  \begin{equation}
    \label{e:Ksmallset-obs}
    \pi_\obs \E_{C_{j+1}}\qB{\theta (1-\Loc_X)  K_j(X)}
    + \pi_\obs \E_{C_{j+1}} \qa{(\delta I)^X- \frac 12 { \sum_{ B' \neq  B, B'\subset \bar{B}}} \delta I(B) \delta I(B')\1_{B\cup B'=X}}. 
\end{equation}
Let us note that since $B \cup B'$ is not necessarily connected (so in that case not a small set),
along with the third term  in \eqref{e:Ksmallset-obs}, there is a corresponding correction for polymers in the large set sum
\eqref{e:K-linear-large}: the terms inside the sum are replaced by
$e^{-V_{j+1}(U \setminus X)+u_{j+1}|X|}$ multiplied by
\begin{equation}
    \label{e:Kquad-obs}
    \pi_\obs \E_{C_{j+1}}\qB{\theta K_j(X)}\1_{X\in \mathcal C_j\backslash \mathcal S_j}
    + \pi_\obs \E_{C_{j+1}} \qa{(\delta I)^X-\frac 12 {\sum_{ B' \neq  B, B'\subset \bar{B}}}  \delta I(B) \delta I(B')\1_{B\cup B'=X}}\1_{X\in \mathcal P_j\backslash \mathcal S_j}.
\end{equation}

Now Lemma~\ref{lem:smallset-Kobs} bounds the sum over $X$ of the
$(1-\Loc_X)$ terms in \eqref{e:Kblock-obs} and
\eqref{e:Ksmallset-obs}.  Lemma~\ref{lem:smallset-Vobs} bounds the
second term on the right-hand side of \eqref{e:Kblock-obs}.  Finally,
\eqref{e:smallset-Vobs-1} of Lemma~\ref{lem:smallset-Vobs1} bounds the
second term in \eqref{e:Ksmallset-obs} by
$O(\|\Vbulk_j\|_j+\|\Kbulk_j\|_j)(\|V_j\|_j+L^{2(d-2)}\|K_j\|_j)$,
after making use of the cancellation between $(\delta I)^X$ and
$\delta I(B)\delta I(B')$ when $X = B \cup B'$ and
$B' \subset \bar B$. Indeed, note that
for all other $X$ at least one $B \in \cB_j(X)$ does not contain $a$ or $b$.
Putting these bounds together
 (as in the proof of Theorem~\ref{thm:step}) 
then gives that the small set contribution to
  $\pi_{\obs}\cL_{j+1}(U)$ is $O(L^{-\dplus}\|\Kobs_j\|_j) + O(\|\Vbulk_j\|_j+\|\Kbulk_j\|_j)(\|V_j\|_j+L^{2(d-2)}\|K_j\|_j)$.

To bound the large set term \eqref{e:Kquad-obs} and the
  non-linear contributions, we will use the principle that for
  $F_i\in\cV$,
  \begin{align} \label{e:norm-obs-nonmon}
    \pi_\obs \prod_{i=1}^k F_i= \sum_{i}  F_{i}^{\obs}\prod_{l\neq i} F_{l}^{\varnothing}+ \sum_{i\neq k} F_{i}^{\obs} F_{k}^{\obs}\prod_{l\neq i,k}F_{l}^{\varnothing}
  \end{align}
  as the product of any three elements of $\cV^{\obs}$ is zero.  The
  bound on the sum over
\begin{equation}
 \pi_\obs \E_{C_{j+1}}\qB{\theta K_j(X)}\1_{X\in \cC_j\setminus \cS_j}
\end{equation}
proceeds exactly as in Section \ref{sec:K-largeset}, bearing in mind
\eqref{e:norm-obs-nonmon} and \eqref{e:Tphi-obs-ub}. The
  resulting estimate is $O(A^{-\eta}\|\Kobs_{j}\|_{j})$.   
For the second term in \eqref{e:Kquad-obs},
observe that if $|\cB_j(X)|=2$ and $\bar{X}\in \mathcal B_{j+1}$, 
the bound is identical to that of the same term in
\eqref{e:Ksmallset-obs} above.
The remaining possibilities are that either $|\cB_j(X)| \geq 3$ or $|\cB_j(X)|=2$ but with constituent $j$-blocks which are in distinct $(j+1)$-blocks.
In the former case, by applying \eqref{e:norm-obs-nonmon},
\eqref{e:smallset-Vobs-1} of Lemma~\ref{lem:smallset-Vobs1} and 
  Lemma~\ref{lem:smallset-V1} and then proceeding as in Section~\ref{sec:K-largeset}, we obtain
\begin{multline}
  A^{(|\cB_{j+1}(U)|-2^d)_+}  \norma{ \pi_\obs \sum_{X\in {\cP_j }\setminus \cS_j: \bar X=U, |\cB_j(X)|\geq 3}e^{-V_{j+1}(U\setminus X)+u_{j+1}|X|} \E_{C_{j+1}}\qB{(\delta I)^X }}_{T_{j+1}(\ell_{j+1})}
  \\
  \leq 
  O (\|V_j\|_j + \|K_j\|_j)(\|\Vbulk_j\|_j + \|\Kbulk_j\|_j).
\end{multline}
The remaining case is $|\cB_{j+1}(U)|=2$ and $|\cB_j(X)|=2$ where $U = \bar X$.
Then the $\delta I(B)\delta I(B')\1_{B \cup B'=X}$ cancellation is absent, but $\pi_\obs\E_{C_{j+1}}(\delta I)^X$
itself satisfies the desired bound by
Lemma~\ref{lem:smallset-Vobs1}. Indeed, either $a$ and $b$ are
in the same $(j+1)$ block or they are not. If they are, we use
\eqref{e:smallset-Vobs-1} with $n=1$, and if not, this follows from
\eqref{e:deltaI-obs-4} 
since
$a$ and $b$ being in 
distinct $(j+1)$-blocks of $U$ implies 
that $|a-b|_\infty \geq \frac12 L^{j+1}$ since $a$ is positioned in the center of all of its blocks.
The bound 
\begin{multline}
\norma{ \pi_\obs  \sum_{X\in {\cP_j }\setminus \cS_j: \bar X=U, |\cB_j(X)|=2} e^{-V_{j+1}(U\setminus X)+u_{j+1}|X|} \E_{C_{j+1}}\qB{(\delta I)^X} }_{T_{j+1}(\ell_{j+1})}
\\  \leq  O(L^{6d}(\|V_j\|_j + \|K_j\|_j)(\|\Vbulk_j\|_j +\|K_j\|_j))
\end{multline}
follows as there are at most $L^{2d}$ summands.

All together, after possibly increasing $A$, we obtain that the large set contribution
to $\cL_{j+1}(U)$ is
\begin{equation}
  O(A^{-\eta}\|\Kobs_{j}\|_{j})
  + O(A^\nu(\|\Vobs_j\|_j + \|\Kobs_j\|_j) (\|\Vbulk_j\|_j + \|K_j\|_j)).
\end{equation}

  The non-linear contribution does not require any
  changes as the bound from Section~\ref{sec:K-nonlinear} already
  gives (after possibly increasing $A$) $A^\nu O(\|K_j\|_j(\|V_j\|_j+\|K_j\|_j))$.
\end{proof}

\subsection{Flow of observable coupling constants}
\label{sec:flow-obs}

With Theorem \ref{thm:step-obs} in place, the evolution of the
observable coupling constants in $u^{\obs}+\Vobs$ is the same as the
free one from 
Section~\ref{sec:obs-free} up to the addition of remainder terms
from the $K$ coordinate. To avoid carrying an unimportant factor of
$L^{2(d-2)}$ through equations, recall that we write $\OL(\cdot)$ to indicate
bounds with constants possibly depending on $L$
(but we reemphasise that implicit constants are always independent of the scale $j$).

\begin{lemma} \label{lem:Vobs-recursion}
Suppose $j<N$, $x\in \{a,b\}$, and that \eqref{e:V+obs-bd} holds. If
$j<j_{ab}$, further suppose that $K^{ab}_{j}(X)=0$ for $X\in\cS_{j}$.
In Case (1), 
\begin{align}
  \label{e:obs-lambda-rec-a}
  \lambda_{x,j+1} &=\lambda_{x,j} + \OL (\ellxj^{-1}\ell_{j}^{-1}\|K_j^x\|_j),
  \\
  q_{j+1} &= q_j + \lambda_{a,j}\lambda_{b,j} C_{j+1}(a,b) {+ r_j C_{j+1}(0,0)}+ \OL (\ellabj^{-1}\|K_j^{ab}\|_j 1_{j\geq j_{ab}}),
  \\
  \label{e:obs-r-rec-K}
  r_{j+1}&= r_j + \OL (\ellabj^{-1}\ell_j^{-2}\|K_j^{ab}\|_j 1_{j \geq j_{ab}}),
                 \intertext{and in Case (2),}
                   \label{e:obs-lambda-rec-b}
  \lambda_{x,j+1} &= \lambda_{x,j} 
                   + \OL (\ellxj^{-1}\ell_j^{-2}\|K_j^x\|_j),
  \\
  \gamma_{x,j+1} &= \gamma_{x,j} + \lambda_{x,j}C_{j+1}(x,x) + \OL (\ellxj^{-1}\|K_j^x\|_j),
  \\
  q_{j+1} &= q_j + \eta_j C_{j+1}(a,b)  - \lambda_{a,j}\lambda_{b,j} C_{j+1}(a,b)^2
            +r_{j}C_{j+1}(0,0)+ \OL (\ellabj^{-1}\|K_j^{ab}\|_j1_{j\geq j_{ab}}),
  \\
  \eta_{j+1} &= \eta_j - 2\lambda_{a,j}\lambda_{b,j} C_{j+1}(a,b), 
  \\
  \label{e:obs-r-rec-b}
  r_{j+1}&= r_{j} + \OL (\ellabj^{-1}\ell_j^{-2} \|K_j^{ab}\|_j 1_{j\geq j_{ab}}).
\end{align}
Moreover, for $j+1<N$, all coupling constants are independent of $N$.
\end{lemma}

Note that there is no error term in the equation for $\eta$, as the
corresponding nonlocal field monomials $\bar\psi_{a}\psi_{b}$
and $\bar\psi_{b}\psi_{a}$ are not contained in the image of $\Loc$. 

\begin{proof}
  For $j<N$, the main contribution in \eqref{e:V+obs-bd} is identical
  to that in Lemma~\ref{lem:Vobs-flow-free}.  The indicator functions
  $1_{j\geq j_{ab}}$ in the error terms are due to the
  assumption $K_{j}^{ab}(X)=0$ for $j<j_{ab}$ and $X\in\cS_{j}$.
  The bounds
  for the error terms follow from the definition of the norms
  as in obtaining \eqref{e:kN2normbd}.  Finally, that the couplings
  are independent of $N$ is a consequence of the consistency of the
  renormalisation group map, i.e., Proposition~\ref{prop:consistency}
  (applied to the renormalisation group map extended by observables).
\end{proof}

The next lemma shows that if we maintain control of $\|K_k^{\obs}\|_k$
up to scale $j$ then we control the coupling constants in $\Vobs$ on scale $j$.
  \begin{lemma}
    \label{lem:Vobs-bd}
    Assume that $\|\Kobs_k\|_k \leq M \lambda_0b_0L^{-\kappa k}$ for
    $k<j$ and that \eqref{e:V+obs-bd} holds for $k<j$. 
    Then, in Case (1) if $q_0=r_0=0$ and $\lambda_0>0$, 
    \begin{align}
      \label{e:Vobs-bd1}
      \lambda_{j} &= \lambda_0 + \OL (M\lambda_0b_0)\\
            \label{e:Vobs-bd2}
      r_j &= \OL (M\lambda_0b_0|a-b|^{-\kappa})1_{j\geq j_{ab}}
     \intertext{and, in Case (2), if $q_0=r_0=\gamma_{x,0}=\eta_0=0$
            and $\lambda_0>0$,}
                  \label{e:Vobs-bd3}
            \lambda_{j} &= \lambda_0 + \OL (M\lambda_0b_0)\\
            \label{e:Vobs-bd4}
      \eta_j &= \OL (\lambda_0^2|a-b|^{-(d-2)})1_{j\geq j_{ab}}
      \\
            \label{e:Vobs-bd5}
      r_{j}&= \OL (M\lambda_0b_0|a-b|^{-(d-2)-\kappa})1_{j\geq j_{ab}},
    \end{align}
    where $\lambda_{j}=\lambda_{x,j}$ for either $x=a$ or  $x=b$.
    In both Cases (1) and (2),
    \begin{equation}
      \label{e:Vobs-bd6}
      \|\Vobs_{j}\|_{j} \leq \lambda_0 + \OL (\lambda_0^2) + \OL (M\lambda_0b_0).
    \end{equation}
\end{lemma}
\begin{proof}
  The bounds on the coupling constants in~\eqref{e:Vobs-bd1}--\eqref{e:Vobs-bd5} follow from
  Lemma~\ref{lem:Vobs-recursion};
  the hypothesis regarding
  $K_{j}(X)=0$ for $j<j_{ab}$ and $X\in \cS_{j}$ holds as
  Definition~\ref{def:V+obs} implies that for an iteration
  $(V_j,K_j)$ of the renormalisation group map, the $\cN^{ab}$
  components of $V_j(B)$ and $K_j(X)$ with $X\in \cS_j$ can only
  be nonzero for $j>j_{ab}$ since we have started the flow with
  $r_{0}=0$ in Case (1), and $q_{0}=\eta_{0}=r_{0}=0$ in Case~(2).
  What remains is to analyse the recurrences to establish~\eqref{e:Vobs-bd6}.
  
  For $\lambda_{x,j}$,  since $\ellxj^{-1}\ell_j^{-p} = 1$ in Case ($p$), using
  \eqref{e:obs-lambda-rec-a}, respectively \eqref{e:obs-lambda-rec-b},
  \begin{equation}
    \lambda_{x,j}  = \lambda_0 + \sum_{k=0}^{j-1} \OL(\|\Kobs_k\|_k)
    = \lambda_0 + \sum_{k=0}^{j-1} \OL(M\lambda_0b_0L^{-\kappa k})
    =\lambda_{0} + \OL(M\lambda_{0}b_{0}).
  \end{equation}
  The bounds on $r_{j}$ follow from the fact that all contributions
  are $0$ for scales $j<j_{ab}$ if $r_{0}=0$. For example, in Case~(2),
  using~\eqref{e:obs-r-rec-b},
  \begin{equation}
    |r_j|=
    \lambda_0b_0\OL(M\sum_{k=j_{ab}}^{j-1} \ell_{ab,j}^{-1} \ell_j^{-2}L^{-\kappa j})
    =
    \lambda_0b_0\ell_{j_{ab}}^2 \OL(M\sum_{k=j_{ab}}^{j-1} L^{-\kappa j})
    =\OL(M\lambda_0b_0|a-b|^{-(d-2)-\kappa}).
  \end{equation}
  Case (1) is similar, except no factor $\ell_{j_{ab}}$ arises (see \eqref{e:ellsigma}).
  The bound on $\eta_j$ in Case~(2) 
  follows from the preceding analysis of $\lambda_{x,j}$,
  the fact
  that $\eta_{j}=0$ for $j<j_{ab}$ if $\eta_{0}=0$ since $C_{j}$
  has finite range ($C_{j}(a,b)=0$ if $|a-b|_\infty \geq \frac12 L^{j}$), and that
  $C_{j+1}(a,b) \leq \OL(L^{-(d-2)j})$:
  \begin{align}
    |\eta_j|
    &=\OL(\lambda_0^2 \sum_{k=j_{ab}}^{j-1} L^{-(d-2)k})
      =\OL(\lambda_0^2 |a-b|^{-(d-2)}).
  \end{align}

  For the bound on the norm of $\|\Vobs_{j}\|_j$ recall that the $q$
  and $\gamma$ terms have been taken out of $\Vobs$. Thus in Case~(1),
  \begin{equation}
    \|\Vobs_j(B)\|
    \lesssim |\lambda_j|  + |r_j|\ellabj\ell_j^2 
    \lesssim |\lambda_j|  + |r_j| 
    = |\lambda_j| + \OL(M\lambda_0b_0)
    .
  \end{equation}
  Similarly, in Case~(2), using that $\ell_j^2\ellabj = \ell_{j\wedge j_{ab}}^{-2} = \OL(|a-b|^{d-2})$ for $j \geq j_{ab}$,
  \begin{align}
    \|\Vobs_j(B)\|
    &\lesssim |\lambda_j| + |\eta_j| \ell_j^2\ellabj 1_{j\geq j_{ab}} +|r_{j}| \ellabj\ell_j^2 1_{j\geq j_{ab}}
      \nnb
    &\lesssim |\lambda_j| + |\eta_j||a-b|^{d-2}1_{j\geq j_{ab}} + |r_{j}| |a-b|^{d-2}1_{j\geq j_{ab}}
      \nnb
    &\lesssim |\lambda_j| + \OL(\lambda_0^2)1_{j\geq j_{ab}} +\OL(M\lambda_0b_0 |a-b|^{-\kappa})1_{j \geq j_{ab}} = |\lambda_j|+\OL(\lambda_0^2)+\OL(M b_0\lambda_0).\qedhere
  \end{align}
\end{proof}
\medskip

We now analyse the observable flow from initial conditions which extend the bulk initial conditions  of Theorem~\ref{thm:flow},
and verify the assumption on $\Kobs$ made in Lemma~\ref{lem:Vobs-bd} along this flow.
As already remarked at the beginning of Section~\ref{sec:rgmap-ext}, the renormalisation group maps with observables are upper triangular,
so that the observable components of the maps do not affect the bulk flow.
Thus from now on, the bulk components $(V_j^\varnothing,K_j^\varnothing)_{j \leq N}$ are identified with the trajectory
given by Theorem~\ref{thm:flow} and we may use the decay rates from that theorem as inputs in obtaining estimates on the remaining components.
In particular, there is an $\alpha>0$ such that
\begin{equation}
  \label{e:flowVK-bis}
  \|\Vbulk_j\|_j = \OL(b_0L^{-\alpha j}), \qquad 
  \|\Kbulk_j\|_j = \OL(b_0L^{-\alpha j}).
\end{equation}
Using this as input,
we iterate the observable flow \eqref{e:V+obs-bd}--\eqref{e:K+obs-bd},
with initial condition $\lambda_{a,0}=\lambda_{b,0}=\lambda_0$ small enough
and all other observable coupling constants equal to $0$.

\begin{proposition}
\label{prop:flow-obs}
  Assume that the bulk renormalisation group flow $(\Vbulk_j,\Kbulk_j)$ obeys
  \begin{equation}
     \label{e:flowVK-bis-1}
     \|\Vbulk_j\|_j+\|\Kbulk_j\|_j = \OL(b_0L^{-\alpha j})
  \end{equation}
  for some $\alpha>0$.
  Then there is $\kappa>0$ such that for
  $\lambda_{a,0}=\lambda_{b,0}=\lambda_0 >0$ and $b_0>0$ sufficiently small 
  and all other observable coupling constants initially $0$,
  \begin{equation} \label{e:flowVKobs}
    \|\Vobs_j\|_j \leq \OL(\lambda_0), \qquad \|\Kobs_j\|_j \leq \OL(\lambda_0b_0L^{-\kappa j}).
  \end{equation}
\end{proposition}
\begin{proof}
    We may assume $\lambda_{0}<1$,
    and that $\kappa$ is less than $\alpha$ and the exponents of $L$ and $A$ in \eqref{e:K+obs-bd}.
    The proof is by induction.
    The inductive assumption is that $\|\Kobs_k\|_{k} \leq M b_0\lambda_0L^{-\kappa k}$ for all $k \leq j$,
    for some  $M=M(L)$ chosen large enough below.  Clearly, this holds for
    $j=0$.  Lemma~\ref{lem:Vobs-bd} then shows that
    $\|\Vobs_{k}\|_k\leq O_L(1)\lambda_0 + O_L(1) M \lambda_0b_0$ for all
    $k\leq j+1$.  We now apply Theorem~\ref{thm:step-obs} to control $\Kobs_{j+1}$.
    Since $A$ is chosen as function of $L$,
    the second term on the right-hand side of \eqref{e:K+obs-bd} is
    \begin{align}
      O(A^{\nu}) (\|\Vbulk_j\|_j +\|K_{j}\|_{j})      (\|V_j\|_j+\|K_j\|_j)
      &\leq O_L(1) (b_0L^{-\alpha j} + \|K_j^\obs\|_j)(b_0L^{-\alpha j} + \lambda_0 + M\lambda_0b_0)
        \nnb
      &\leq O_L(1)
        b_0\lambda_0 L^{-\alpha j} + \frac14 L^{-\kappa} \|K_j^\obs\|_j,
    \end{align}
    as long as $b_0\leq \lambda_0$ and $Mb_0 +
    \lambda_0$ is sufficiently small (depending on $L$).  As
    $A>L$, and using our second assumption on
    $\kappa$, we obtain from~\eqref{e:K+obs-bd} that
    \begin{align}
      \|\Kobs_{j+1}\|_{j+1}
      &\leq \frac12 L^{-\kappa}\|\Kobs_j\|_j + \OL(1)
        b_0\lambda_0 L^{-\alpha j}
      \nnb
      &\leq (\frac12 M + \OL(1)
        ) b_0\lambda_0 L^{-\kappa (j+1)}
      \leq M b_0\lambda_0 L^{-\kappa (j+1)}
    \end{align}
    provided that $M$ is sufficiently large and
    $b_0$ is sufficiently small (depending only on $L$).
    Hence if $\lambda_0$ and $b_0$ are small enough we have advanced the induction,
    completing the proof.
\end{proof}

\section{Computation of pointwise correlation functions}
\label{sec:obs}

In this section we use the results of Section~\ref{sec:obs-flow} to prove the following estimates for the pointwise
correlation functions $\avg{\bar\psi_a\psi_b}$,
$\avg{\bar\psi_a\psi_a}$, and
$\avg{\bar\psi_a\psi_a\bar\psi_b\psi_b}$.
Recall \eqref{e:WN-def} and   \eqref{e:tn}, i.e.,
\begin{equation}
  \label{e:WN-def-bis}
  W_N(x) = W_{N,m^2}(x) = (-\Delta+m^2)^{-1}(0,x) - \frac{t_N}{|\Lambda_N|}, \qquad 0< t_N = \frac{1}{m^2} - O(L^{2N}).
\end{equation}
Thus $W_N(x)$ is essentially the torus Green function
$(-\Delta+m^2)^{-1}(0,x)$ with the zero mode subtracted
and $t_N$ is essentially $m^{-2}$ when $L^{2N} \ll m^{-2}$ and negligible otherwise.

\begin{proposition} \label{prop:1pt}
  For $b_0$ sufficiently small and $m^2 \geq 0$, there exists continuous functions
  \begin{equation}
    \clambda =\clambda(b_0,m^2)= 1+\OL(b_0), \qquad \cgamma = \cgamma(b_0,m^2) = (-\Delta^{\Z^d}+m^2)^{-1}(0,0) + \OL(b_0),
  \end{equation}
  such that
  if $\Vbulk_{0}=V_{0}^{c}(m^{2},b_{0})$ is as in Theorem~\ref{thm:flow},
  $\Vobs_{0}$ is as in Proposition~\ref{prop:flow-obs},
  and $\tilde u_{N,N}^c=\tilde u_{N,N}^c(b_0,m^2)$ is as in Proposition~\ref{prop:ZNN},
  then
    \begin{equation} \label{e:2ptaa-bis}
    \avg{\bar\psi_a\psi_a} =
    \gamma
    + \frac{
      \clambda
      t_N |\Lambda_N|^{-1}
    + \OL(b_0L^{-(d-2+\kappa)N}) + \OL(b_0L^{-\kappa N}(m^{2}|\Lambda_N|)^{-1})
  }{1+\tilde u_{N,N}}
  .
\end{equation}
\end{proposition}

\begin{proposition} \label{prop:2pt}
  Under the same assumptions as in Proposition~\ref{prop:1pt},
    \begin{align}
    \label{e:2pt-a-bis}
    \avg{\bar\psi_a\psi_b}
    &=
      W_N(a-b)
      + \frac{
      t_N|\Lambda_N|^{-1}
      }{1+\tilde u_{N,N}}
      \nnb
      &\qquad
      + \OL(b_0|a-b|^{-(d-2+\kappa)})
      + 
      \frac{
      \OL(b_0|a-b|^{-\kappa} (m^2|\Lambda_N|)^{-1})
      }{1+\tilde u_{N,N}}
    \\
    \label{e:2pt-b-bis}
    \avg{\bar\psi_a\psi_a\bar\psi_b\psi_b}
      &= -\lambda^2 W_N(a-b)^2 + \gamma^2
        +\frac{-2\lambda^2 W_N(a-b)
        + 2\lambda\gamma
        }{1+\tilde u_{N,N}}t_N|\Lambda_N|^{-1}
      \nnb
      &\qquad
        + \OL(b_0 |a-b|^{-2(d-2)-\kappa})
        + \OL(b_0L^{-(d-2+\kappa)N})
        \nnb
      &\qquad
        + (
        \OL(b_0L^{-\kappa N}) + \OL(b_0 |a-b|^{-(d-2+\kappa)}))\frac{(m^{2}|\Lambda_N|)^{-1}}{1+\tilde u_{N,N}}
        .
  \end{align}
\end{proposition}

Throughout this section, we assume that the renormalisation group flow
$(V_j,K_j)_{j \leq N}$ is given as in Corollary~\ref{cor:flow-finvol} (bulk)
and Proposition~\ref{prop:flow-obs} (observables).

\subsection{Integration of the zero mode}

As in the analysis of the susceptibility in Section~\ref{sec:suscept},
we treat the final integration over the zero mode explicitly.
Again we will only require the restriction to constant $\psi,\bar\psi$
(as discussed below \eqref{e:ZNN-def}) of
\begin{equation}
  \E_{C} \theta Z_0
  = \E_{t_N Q_N} \theta Z_N = e^{-u^\varnothing_N|\Lambda_N|} \tilde Z_{N,N},
\end{equation}
where the last equation defines $\tilde Z_{N,N}$. We write
$\tilde Z_{N,N}=\tilde Z_{N,N}^{\varnothing}+\tilde Z_{N,N}^{\obs}$
for its decomposition into bulk and observable parts (see \eqref{e:Fdecomp}).
The bulk term was already computed in Proposition~\ref{prop:ZNN}.  The observable
term $\tilde Z_{N,N}^\obs$ is computed by the next lemma.
In the lemma we only give explicit formulas for the terms that will be used in the proofs of
Propositions~\ref{prop:1pt} and~\ref{prop:2pt}.

\begin{lemma}
  \label{lem:ZNN-obs}
  Restricted to constant $\psi,\bar\psi$, in Case (1),
  \begin{align}
    \tilde Z_{N,N}^\obs
    &=
      \sigma_a\bar\psi \tilde Z_{N,N}^{\sigma_a\bar\psi}  + \bar\sigma_b\psi \tilde Z_{N,N}^{\sigma_b\psi}
      +\sigma_a\bar\sigma_b \tilde Z_{N,N}^{\bar\sigma_b\sigma_a}
      +\sigma_a\bar\sigma_b \psi\bar\psi \tilde Z_{N,N}^{\bar\sigma_b\sigma_a\psi\bar\psi}
      \intertext{where}
      \label{e:ZNN-lambda-a}
      \tilde      Z_{N,N}^{\sigma_a\bar\psi}
    &= \lambda_{a,N} +\OL(\ellxN^{-1}\ell_N^{-1}\|\Kobs_N\|_N)
      \\
      \label{e:ZNN-lambda-b}
    \tilde      Z_{N,N}^{\bar\sigma_b\psi}
    &= \lambda_{b,N} +\OL(\ellxN^{-1}\ell_N^{-1}\|\Kobs_N\|_N)
    \\ 
    \label{e:ZNN-q-a}
    \tilde    Z_{N,N}^{\bar\sigma_b\sigma_a}
    &= q_N (1+\tilde u_{N,N})+ \lambda_{a,N}\lambda_{b,N}
      t_N|\Lambda_N|^{-1} { - r_N t_N |\Lambda_N|^{-1}}
      \nnb
      &\qquad\qquad
      + \OL(m^{-2}|\Lambda_N|^{-1}\ell_N^{-2}\ellabN^{-1}\|\Kobs_N\|_N)
      .
  \end{align}
  In Case (2), 
    \begin{align}
      \tilde Z_{N,N}^\obs
      &=
        \sigma_a\tilde Z_{N,N}^{\sigma_a}+ \sigma_a\bar\psi\psi \tilde Z_{N,N}^{\sigma_a\bar\psi\psi}
        +\sigma_b \tilde Z_{N,N}^{\sigma_b} + \sigma_b\bar\psi\psi \tilde Z_{N,N}^{\sigma_b\bar\psi\psi}
        + \sigma_a\sigma_b \tilde Z_{N,N}^{\sigma_a\sigma_b}
        + \sigma_a\sigma_b \psi\bar\psi \tilde Z_{N,N}^{\sigma_a\sigma_b\psi\bar\psi}
  \intertext{where, setting $\tilde \lambda_{x,N,N} = \lambda_{x,N}
        +\OL(\ellxN^{-1} \ell_N^{-2} \|\Kobs_N\|_N)$,}
    \label{e:ZNN-eta-b}
        \tilde Z_{N,N}^{\sigma_x}
      &=
        \gamma_{x,N} (1+\tilde u_{N,N}) +
        \tilde\lambda_{x,N,N}t_N|\Lambda_N|^{-1} 
        + \OL(\ellxN^{-1}\|\Kobs_N\|_N).
    \\
    \label{e:ZNN-q-b}
      \tilde   Z_{N,N}^{\sigma_a\sigma_b}
      &= (q_N+\gamma_{a,N}\gamma_{b,N}) (1+\tilde u_{N,N})
        + (\eta_N - r_N +\tilde\lambda_{a,N,N}\gamma_{b,N}+\tilde\lambda_{b,N,N}\gamma_{a,N})t_N|\Lambda_N|^{-1}
        \nnb
      &\qquad\qquad
        + \OL((|\gamma_{a,N}|+|\gamma_{b,N}|)\ellxN^{-1}
        + \ellabN^{-1}
        + m^{-2}|\Lambda_N|^{-1}\ell_N^{-2}\ellabN^{-1})\|\Kobs_N\|_N.
  \end{align}
\end{lemma}

The error bounds above reveal the tension in the explicit choices of
$\ellxj^{-1}$ and $\ellabj^{-1}$.  To obtain effective error
estimates, we want $\ellxN^{-1}$ and $\ellabN^{-1}$ to be as small as
possible.  On the other hand, to control the iterative estimates of
Theorem~\ref{thm:step-obs} over the entire trajectory, i.e., to prove
Proposition~\ref{prop:flow-obs}, we needed that $\ell_{x,j}$ and
$\ell_{ab,j}$ were not too large.  In particular, either of the more
naive choices $\ell_{ab,j}=\ell_{j}^{-4}$ and
  $\ell_{ab,j} = \ell_{j_{ab}}^{-4}$ in Case~(2) would have lead to
difficulties, either in terms of forcing us to track additional terms
in the flow and in terms of controlling norms inductively, or
by leading to error bounds that are not strong enough to capture
the zero mode sufficiently accurately.

\begin{proof}
  Throughout the proof, we restrict to constant $\psi,\bar\psi$.
  Since  
  \begin{align}
    \pa{e^{+u_N^\varnothing(\Lambda)}Z_N}^\obs
    &= \pa{e^{-u_N^\obs(\Lambda)}(e^{-V_N(\Lambda)}+K_N(\Lambda))}^\obs
      \nnb
    &=
    (e^{-u_N^\obs(\Lambda)}-1)(e^{-\Vbulk_N(\Lambda)}+\Kbulk_N(\Lambda))
      + e^{-u_N^\obs(\Lambda)}(e^{-V_N(\Lambda)}+K_N(\Lambda))^\obs,
  \end{align}
  by applying $\E_{t_NQ_N}\theta$ we obtain
  \begin{multline}
    \tilde Z_{N,N}^\obs = \underbrace{(e^{-u^\obs_N(\Lambda)}-1)(1+\tilde u_{N,N}-|\Lambda_N|\tilde a_{N,N}\psi\bar\psi)}_{A}
    + \underbrace{e^{-u_N^\obs(\Lambda)}\E_{t_NQ_N}\qB{\theta(e^{-V_N(\Lambda)}+K_N(\Lambda))^\obs}}_{B} 
    .
  \end{multline}
  In obtaining $A$
  we used \eqref{e:ZNN} which gives $\E_{t_NQ_N}[\theta (e^{-V_N^\varnothing(\Lambda)}+K_N^\varnothing(\Lambda))]= 1+\tilde u_{N,N}-|\Lambda_N|\tilde a_{N,N}\psi\bar\psi$.
  Since each term in $V_N^\varnothing(\Lambda)$ contains a factor $\psi\bar\psi$
  and each term in $V_N^\obs(\Lambda)$ either $\psi$ or $\bar\psi$,
  we have $V_N^\varnothing(\Lambda)V_N^\obs(\Lambda)=0$.
  Thus
  \begin{equation}
    B = e^{-u_N^\obs(\Lambda)}\E_{t_NQ_N}\qB{\theta(-V_N^\obs(\Lambda)+\frac12 V_N^\obs(\Lambda)^2+K_N^\obs(\Lambda))}
    .
  \end{equation}

  \noindent\emph{Case (1).}
  Since $\sigma_a^2=\bar\sigma_b^2=0$, 
  \begin{equation}
    e^{-u_N^\obs(\Lambda)}-1=
    -u^\obs_N(\Lambda) = \sigma_a\bar\sigma_b q_N,
  \end{equation}
  we get
  \begin{align}
    A&=\sigma_a\bar\sigma_b q_N (1+\tilde u_{N,N}-|\Lambda_N|\tilde a_{N,N}\psi\bar\psi) 
       \\
    B &=
    \sigma_a\bar\psi(\lambda_{a,N}+k^{\sigma_a\bar\psi}_{N})+{\psi\bar\sigma_{b}}
        (\lambda_{b,N}+k^{\bar\sigma_b\psi}_{N})
        +\sigma_a\bar\sigma_b \E_{t_NQ_N}\qB{\theta \bar\psi\psi (\lambda_{a,N}\lambda_{b,N} - r_N + k^{\sigma_a\bar\sigma_b \bar\psi\psi}_{N})}
        .
  \end{align}
  The constants $k_N^\#$ are given in terms of derivatives of
  $K_N(\Lambda)$ and bounded analogously as in \eqref{e:kN2normbd}.
  For example,
  $k_N^{\sigma_a\bar\psi} = \OL(\ellaN^{-1}\ell_N^{-1}\|\Kobs_N\|_N)$,
  and similarly for the other $k_N^\#$ terms, the rule being that we
  have a factor $\ellxN^{-1}$ if there is a superscript $\sigma_a$ or
  $\bar\sigma_b$ but not both, a factor $\ellabN^{-1}$ for
  $\sigma_a\bar\sigma_b$ and a factor $\ell_N^{-1}$ for each
  superscript $\psi$ or $\bar\psi$.  These bounds follow from the
  definition of the $T_j(\ell_j)$ norm.

  Since $\E_{t_NQ_N}\theta \psi\bar\psi = -t_N|\Lambda_N|^{-1}+\psi\bar\psi$ the claim
  follows by collecting terms and using~\eqref{e:tN}.
  
  \medskip
  \noindent\emph{Case (2).}
  Using again that $\sigma_a^2=\sigma_b^2=0$, but now taking in account that $u^\obs(\Lambda)$ has additional terms compared to Case~(1),
  \begin{equation}
    e^{-u_N^\obs(\Lambda)}-1=
    -u^\obs_N(\Lambda)+\frac12 u^\obs_N(\Lambda)^2 = \sigma_a\sigma_b (q_N+\gamma_{a,N}\gamma_{b,N}) + \sigma_a \gamma_{a,N} +\sigma_b\gamma_{b,N},
  \end{equation}
  and therefore
  \begin{equation}
    A =(\sigma_a\sigma_b(q_N+\gamma_{a,N}\gamma_{b,N})+\sigma_a \gamma_{a,N}+\sigma_b\gamma_{b,N})
    (1 + \tilde u_{N,N} - |\Lambda_N|\tilde a_{N,N} \psi\bar\psi)
    .
  \end{equation}
  Since in Case (2) each term in $\Vobs_N(\Lambda)$ contains a factor of $\psi\bar\psi$,
  we have $\Vobs_N(\Lambda)^2=0$ 
  and thus
  \begin{equation}
    B = (1-u_N^\obs(\Lambda))\E_{t_NQ_N}\qB{\theta(-\Vobs_N(\Lambda)+\Kobs_N(\Lambda))}.
  \end{equation}
  Therefore
   \begin{align}
     B&=         \sigma_a k_N^{\sigma_a}
     +
     \sigma_b k_N^{\sigma_b}
     +\sigma_a\sigma_b (\gamma_{a,N} k_N^{\sigma_b} + \gamma_{b,N} k_N^{\sigma_a} + k_N^{\sigma_a\sigma_b})
     \nnb
      &\qquad+
      \E_{t_NQ_N}\qB{ \theta\p{\bar\psi\psi
        (\sigma_{a}\tilde\lambda_{a,N} +
        \sigma_{b}\tilde\lambda_{b,N}) 
     + \sigma_a \sigma_b\bar\psi\psi (\eta_N - r_N    +
     \gamma_{a,N}\tilde\lambda_{b,N}+\gamma_{b,N}\tilde\lambda_{a,N} + k^{\sigma_a\sigma_b\bar\psi\psi})}}
  \end{align}
  where we have set
  $\tilde\lambda_{x,N} = \lambda_{x,N} + k_N^{\sigma_x\bar\psi\psi}$.
  Taking the expectation and collecting all terms gives
  \begin{align}
    \tilde Z_{N,N}^{\sigma_a\sigma_b}
    &= (q_N+\gamma_{a,N}\gamma_{b,N}) (1+\tilde u_{N,N})
      \nnb
    &\qquad\qquad + (\eta_N - r_N +\tilde\lambda_{a,N}\gamma_{b,N}+\tilde\lambda_{b,N}\gamma_{a,N} + k^{\sigma_a\sigma_b\bar\psi\psi})t_N|\Lambda_N|^{-1}
      \nnb
      &\qquad\qquad
      +\gamma_{a,N} k_N^{\sigma_b} + \gamma_{b,N} k_N^{\sigma_a} + k_N^{\sigma_a\sigma_b}
    \\
    \tilde Z_{N,N}^{\sigma_a}
    &=
      \gamma_{a,N} (1+\tilde u_{N,N}) + \tilde\lambda_{a,N}t_N|\Lambda_N|^{-1}
      +    k_N^{\sigma_a}
      .
  \end{align}
  The bounds on the constants $k_N^\#$ are analogous to those in Case~(1).
\end{proof}

\subsection{Analysis of one-point functions}

We now analyse the observable flow given by Lemma~\ref{lem:Vobs-recursion} to derive the asymptotics of the correlation functions.
Note that the coupling constants $\lambda_{x,j}$ and $\gamma_{x,j}$ can possibly
depend on $x=a,b$
as the contributions from $K$ can depend on the relative position of the points in the division of $\Lambda_N$
into blocks.
The following lemma shows that in the limit $j\to\infty$ they become
independent of $x$;
an analogous argument was used in \cite[Lemma~4.6]{MR3345374}.

\begin{lemma} \label{lem:1pt}
Under the hypotheses of Proposition~\ref{prop:1pt} there are
$\lambda_\infty^{(p)} = \lambda_0 +
\OL(\lambda_0b_0)$ and $\gamma_\infty= \OL(\lambda_0)$, all
continuous in $m^{2}\geq 0$ and $b_{0}$ small,
such that for $x\in\{a,b\}$,
\begin{equation} \label{e:gammainfbd}
  \lambda_{x,j}^{(p)} = \lambda_\infty^{(p)} + \OL(\lambda_0 b_0 L^{-\kappa j}),\qquad
  \gamma_{x,j}=\gamma_\infty + \OL(\lambda_0 b_0 L^{-(d-2+\kappa)j}).
\end{equation}
In Case (1), $\lambda_\infty^{(1)} = \lambda_0$.  In Case (2),
$\lambda_\infty^{(2)} = \lambda_0 + \OL(\lambda_0b_0)$ and
$\gamma_\infty^{(2)} =
\lambda_\infty^{(2)}(-\Delta^{\Z^d}+m^2)^{-1}(0,0)+
\OL(\lambda_0b_0)$, and with the abbreviations
$\lambda\bydef \lambda^{(2)}$ and $\gamma\bydef \gamma^{(2)}$,
\begin{equation} \label{e:2ptaa}
    \avg{\bar\psi_a\psi_a} =
    \frac{\gamma_\infty}{\lambda_0}
    + \frac{
      \frac{\lambda_{\infty}}{\lambda_0} t_N |\Lambda_N|^{-1}
    + \OL(b_0L^{-(d-2+\kappa)N}) + \OL(b_0L^{-\kappa N}(m^{2}|\Lambda_N|)^{-1})
  }{1+\tilde u_{N,N}}
  .
\end{equation}
\end{lemma}

\begin{proof}
  We will typically drop the superscript $(p)$.
  In both cases, we have already seen that
  \begin{equation}
    \lambda_{x,j}  = \lambda_0 + \sum_{k=0}^{j-1} \OL(\|\Kobs_k\|_k)  = \lambda_0 + \sum_{k=0}^{j-1} \OL(\lambda_0b_0L^{-\kappa k}) .
  \end{equation}
  Since the $\Kobs_k$ are independent of $N$ for $k<N$ (by
  Proposition~\ref{prop:consistency} for the extended
  renormalisation group map, see
    Section~\ref{sec:rgmap-ext}), the limit $\lambda_{x,\infty}$ makes sense, exists, and
    $|\lambda_{x,j}-\lambda_{x,\infty}| = \OL(\lambda_0b_0L^{-\kappa j})$.
  Similarly, in Case (2),
  by Lemma~\ref{lem:Vobs-recursion} and $\ellxj^{-1} = \ell_{j}^2 =\OL(L^{-(d-2)j})$,
  \begin{align}
    \gamma_{x,j}
    &= \sum_{k=0}^{j-1} \qB{ \lambda_{x,k} C_{k+1}(x,x) + \OL(L^{-(d-2)k}\|\Kobs_k\|_k)} .
  \end{align}
  In particular, by the above estimate for $|\lambda_{x,j}-\lambda_{x,\infty}|$,
  we have
    \begin{equation}
      \gamma_{x,\infty}
      = \lambda_{x,\infty}\sum_{k=0}^\infty C_{k+1}(0,0) + \OL(\lambda_0b_0)
      = \lambda_{x,\infty}(-\Delta^{\Z^d}+m^2)^{-1}(0,0) + \OL(\lambda_0b_0).
    \end{equation}
  The continuity claims follow from the continuity of the covariances
  $C_{j}$ in $m^{2}\geq 0$, of the renormalisation group coordinates
  $K_j$, and that both $\lambda_\infty$ and $\gamma_\infty$ are uniformly
  convergent sums of terms continuous in $b_{0}$ and $m^{2}\geq 0$. 
  
  To show that  $\lambda_{x,\infty}^{(1)}=\lambda_0$ in Case~(1),
  which is in particular independent of $x$, 
  we argue as in the proof of \cite[Lemma~4.6]{MR3345374}.
  On the one hand, Lemma~\ref{lem:ZNN-obs} implies as $N\to\infty$ with $m^2>0$ fixed,
  \begin{equation}
    \label{e:lambda1-pf1}
    \partial_{\bar\psi} \partial_{\sigma_a} \tilde Z_{N,N}|_{0}
    = \lambda_{a,N} + \OL(\ell_N^{-1}\ellxN^{-1} \|\Kobs_N\|_N)
    = \lambda_{a,N} + \OL(\|\Kobs_N\|_N)
    \xrightarrow[N\to\infty]{}
    \lambda_{a,\infty},
  \end{equation}
  where $|_0$ denotes projection onto the degree $0$ part, i.e., 
  $\psi=\bar\psi=\sigma=\bar\sigma=0$, and we have dropped the superscript (1) from $\lambda_{a,j}$.
  On the other hand, we claim
  \begin{equation}
    \label{e:lambda1-pf2}
      \partial_{\bar\psi} \partial_{\sigma_a} \tilde Z_{N,N}|_{0}
      = \lambda_{0}m^2(1+\tilde u_{N,N})\sum_{x\in\Lambda_N}\avg{\bar\psi_0\psi_x}
      = \lambda_{0}\pa{1+\tilde u_{N,N}-\frac{\tilde a_{N,N}}{m^2}}.
  \end{equation}
  Indeed, the first equality in \eqref{e:lambda1-pf2} follows analogously to \cite[(4.51)--(4.53)]{MR3345374}:
  let $\Gamma(\rho,\bar\rho)$ be as in \eqref{e:Gammadef}, except that $Z_0$
  now includes the observable terms $\sigma_a$ and $\bar\sigma_b$
  and we write $\rho$ and $\bar\rho$ for the constant
  external field to distinguish them from $\sigma_a$ and $\bar\sigma_b$. Then
  as in \eqref{e:Gammasquare},
  \begin{equation}
    -\sum_{x\in\Lambda_N} \avg{\psi_x}_{\sigma_a,\bar\sigma_b} =
    \partial_{\bar\rho} \Gamma(\rho,\bar\rho)|_{\rho=\bar\rho=0} = m^{-2} \frac{\partial_{\bar\psi}\tilde Z_{N,N}|_{\psi=\bar\psi=0}}{\tilde Z_{N,N}|_{\psi=\bar\psi=0}},
  \end{equation}
  and $\avg{\cdot}_{\sigma_a,\bar\sigma_b}$ denotes the expectation that still depends on
  the source
  fields $\sigma_a$ and $\bar\sigma_b$.
  Differentiating with respect to $\sigma_a$ and setting $\bar\sigma_b=0$ gives
  \begin{equation}
    \lambda_{0}
    \sum_{x\in\Lambda_N} \avg{\bar\psi_a\psi_x}
    = -m^{-2} \frac{\partial_{\sigma_a}\partial_{\bar\psi}\tilde Z_{N,N}|_{0}}{\tilde Z_{N,N}|_0}
    = m^{-2} \frac{\partial_{\bar\psi}\partial_{\sigma_a}\tilde Z_{N,N}|_{0}}{1+\tilde u_{N,N}},
  \end{equation}
  which is the first equality of \eqref{e:lambda1-pf2} upon
  rearranging.  The second equality in \eqref{e:lambda1-pf2} follows
  from Proposition~\ref{prop:suscept}.

  The right-hand side of \eqref{e:lambda1-pf2} converges to
  $\lambda_0$ in the limit $N\to\infty$ with $m^2>0$ fixed since
  $\tilde a_{N,N}= a_N - k_N^2/|\Lambda_N| =
  \OL(L^{-2N}\|V_N\|_N)+\OL(L^{-2N}\|K_N\|_N) \to 0$ and
  $\tilde u_{N,N}= k_N^0 + \tilde a_{N,N}t_N = O(\|K_N\|_N) + \tilde
  a_{N,N}t_N \to 0$ when $m^2>0$ is fixed.  Since the left-hand sides
  of \eqref{e:lambda1-pf1}--\eqref{e:lambda1-pf2} are equal, we
  conclude that $\lambda_{a,\infty}=\lambda_0$ when $m^2>0$. By
  continuity this identity then extends to $m^2=0$.  
  
  In Case~(2), to show~\eqref{e:2ptaa}, we use \eqref{e:ZNN-eta-b},
  that Proposition~\ref{prop:flow-obs} implies
  $\|\Kobs_N\|_N = \OL(\lambda_0b_0L^{-\kappa N})$, and
  $\ellxN^{-1}\ell_N^{-2} = 1$ and
  $\ellxN^{-1}= \ell_N^{2} = \OL(L^{-(d-2)N})$ to obtain
  \begin{equation}
    \frac{\tilde Z_{N,N}^{\sigma_a}}{1+\tilde u_{N,N}}
    = \gamma_{a,N} +
    \frac{\lambda_{a,\infty} t_N |\Lambda_N|^{-1}  + \OL(\lambda_0 b_0L^{-(d-2+\kappa)N}) + \OL( \lambda_0b_0 L^{-\kappa N} (m^{2}|\Lambda_N|)^{-1})
      }{1+\tilde u_{N,N}}.
    \end{equation}
    Since
    \begin{equation}
      \lambda_0 \avg{\bar\psi_a\psi_a} = \frac{\partial_{\sigma_a} \tilde Z^{N,N}|_0}{\tilde Z_{N,N}|_0} = \frac{\tilde  Z_{N,N}^{\sigma_a}}{1+\tilde  u_{N,N}},
    \end{equation}
    this gives \eqref{e:2ptaa}.  In particular, by the
    translation invariance of $\avg{\bar\psi_a\psi_a}$, taking $N\to\infty$
    with $m^2>0$ fixed implies $\gamma_{a,\infty}$ is independent of $a$.
    Similarly, taking $m^2\downarrow 0$ first and then $N\to\infty$ we see that $\lambda_{a,\infty}$
    is independent of $a$. Indeed, using \eqref{e:ZNN}, as $N\to\infty$,
    \begin{equation}
      \lambda_0 \avg{\bar\psi_a\psi_a}
      \sim \gamma_{a,\infty}+ \lambda_{a,\infty} \lim_{m^2\downarrow 0}\frac{t_N|\Lambda_N|^{-1}}{\tilde u_{N,N}}
      \sim
      \gamma_{a,\infty} + \lambda_{a,\infty}  \frac{1}{|\Lambda_N|\tilde a_{N,N}}
      ,
    \end{equation}
    where $\avg{\bar\psi_a\psi_a}$ and all scale-dependent coupling constants are evaluated at $m^2=0$.
    Thus $\lambda_{a,\infty} =\lim_{N\to\infty} |\Lambda_N|\tilde a_{N,N} (\lambda_0 \avg{\bar\psi_a\psi_a}-\gamma_{a,\infty})$
    and the right-hand side is independent of $a$.
\end{proof}

\begin{proof}[Proof of Proposition~\ref{prop:1pt}]
  Taking $\lambda_0>0$ small enough,  the proposition follows immediately from Lemma~\ref{lem:1pt} with $\clambda = \lambda_\infty^{(2)}/\lambda_0$
  and $\cgamma = \gamma_\infty^{(2)}/\lambda_0$.
\end{proof}

\subsection{Analysis of two-point functions}

Next we derive estimates for the two-point functions.

\begin{lemma} \label{lem:2pt}
  Under the hypotheses of Proposition~\ref{prop:1pt},
  \begin{align}
    \label{e:2pt-a}
    \avg{\bar\psi_a\psi_b}
    &=
      W_N(a-b)
      + \frac{
      t_N|\Lambda_N|^{-1}
      }{1+\tilde u_{N,N}}
      \nnb
    &
      \qquad
      + \OL(\frac{b_0}{\lambda_0}|a-b|^{-(d-2+\kappa)})
      + 
      \OL(\frac{b_0}{\lambda_0} |a-b|^{-\kappa})
      \frac{(m^2|\Lambda_N|)^{-1}}{1+\tilde u_{N,N}}, 
    \\
    \intertext{and, setting
    $\lambda_\infty = \lambda^{(2)}_\infty$
    and $\gamma_\infty = \gamma^{(2)}_\infty$ as in Lemma~\ref{lem:1pt},}
    \label{e:2pt-b}
    \avg{\bar\psi_a\psi_a\bar\psi_b\psi_b}
      &= -\frac{\lambda_\infty^2}{\lambda_0^2} W_N(a-b)^2 +
        \frac{\gamma_\infty^2}{\lambda_0^2}
        +\frac{-2\lambda_\infty^2 W_N(a-b)
        + 2\lambda_{\infty}\gamma_{\infty} 
        }{\lambda_0^2(1+\tilde u_{N,N})}t_N|\Lambda_N|^{-1}
      \nnb
      &\qquad
        + \OL(\frac{b_0}{\lambda_0} |a-b|^{-2(d-2)-\kappa})
        + \OL(\frac{b_0}{\lambda_0}L^{-(d-2+\kappa)N})
        \nnb
        &\qquad
        + 
        (\OL(\frac{b_0}{\lambda_0}|a-b|^{-(d-2+\kappa)})+\OL(\frac{b_0}{\lambda_0}L^{-\kappa N}))\frac{(m^{2}|\Lambda_N|)^{-1}}{1+\tilde u_{N,N}}
        .
  \end{align}
\end{lemma}

\begin{proof}
  The proofs of \eqref{e:2pt-a} and \eqref{e:2pt-b} corresponding to Cases (1) and (2)
  are again analogous. 

  \smallskip\noindent\emph{Case (1).}
  By Lemma~\ref{lem:Vobs-bd} (whose hypotheses are verified by Proposition~\ref{prop:flow-obs}) and
  Lemma~\ref{lem:1pt},
  \begin{equation} \label{e:lambdajab}
    \lambda_{x,j}
    = \lambda_\infty + \OL( \lambda_0b_0 L^{-\kappa}) 
    = \lambda_0 + \OL( \lambda_0b_0 L^{-\kappa}), 
    \qquad
    r_j=\OL(\lambda_0b_0 |a-b|^{-\kappa})
    1_{j\geq j_{ab}}.
  \end{equation}
  Using that $\ellabj^{-1}\|K^{ab}_j\|_j \leq
  \OL(\lambda_0b_0L^{-(d-2+\kappa)j})1_{j\geq j_{ab}}$ and $|C_{j+1}(a,b)| \leq C_{j+1}(0,0) \leq \OL(L^{-(d-2)j})$ it then follows
  from Lemma~\ref{lem:Vobs-recursion} that
  \begin{align}
    q_{N}
    &= \sum_{j=j_{ab}-1}^{N-1} \qB{ \lambda_{a,j}\lambda_{b,j} C_{j+1}(a,b) + r_j C_{j+1}(0,0) + \OL(\lambda_0 b_0 L^{-(d-2+\kappa)j}) }
      \nnb
    &= \lambda_0^2 \sum_{j=1}^{N-1} C_j(a,b) 
      + \OL(\lambda_0 b_0 |a-b|^{-(d-2)-\kappa})
      \nnb
    &= \lambda_0^2 W_N(a-b)  + \OL(\lambda_0 b_0 |a-b|^{-(d-2)-\kappa})
      ,
  \end{align}
  where we have used
  \eqref{e:lambdajab},
  $|a-b|\leq L$, that $C_j(a,b)=0$ for $j<j_{ab}$,
  and that $W_N(x-y)=C_1(x,y) + \cdots + C_N(x,y)$.
  By \eqref{e:ZNN-q-a},
  using that $\ellabN^{-1}\ell_N^{-2}=1$ and again  \eqref{e:lambdajab}, therefore
  \begin{align}
    \frac{\tilde Z_{N,N}^{\bar\sigma_b\sigma_a}}{1+\tilde u_{N,N}}
    &= \lambda_0^2 W_N(a-b)+ \OL(\lambda_0b_0 |a-b|^{-(d-2)-\kappa}) 
      \nnb
    &\qquad
      + \frac{\lambda_0^2 t_N |\Lambda_N|^{-1}
      + \OL(\lambda_0b_0 |a-b|^{-\kappa} m^{-2}|\Lambda_N|^{-1})}{1+\tilde u_{N,N}}
      .
  \end{align}
  Since
  $\avg{\bar\psi_a\psi_b} =
  \tilde Z_{N,N}^{\bar\sigma_b\sigma_a}/(\lambda_0^2 (1+\tilde u_{N,N}))$ and
  $|\lambda_0| \leq 1$, the claim for the two-point function follows.

  \medskip\noindent\emph{Case (2).}
  Again, the analogue of \eqref{e:lambdajab} holds: 
  \begin{equation}
  \label{e:lastone}
    \lambda_{x,j} = \lambda_\infty + \OL(b_0\lambda_0 L^{-\kappa j})
    ,\qquad
    r_{j} = \OL(b_0\lambda_0|a-b|^{-(d-2+\kappa)})1_{j\geq j_{ab}}
    .
  \end{equation}
  The first estimate is by Lemma~\ref{lem:1pt},
  the second by Lemma~\ref{lem:Vobs-bd}.
  Since $C_{k}(a,b)=0$ for $k<j_{ab}$ and $|C_{k+1}(a,b)| \leq \OL(L^{-(d-2)k})$,
  then by Lemma~\ref{lem:Vobs-recursion} and as $|a-b|\leq L$,
  \begin{equation}
    \eta_{j}
    =
    -2 \sum_{k=j_{ab}-1}^{j-1} \lambda_{a,k}\lambda_{b,k} C_{k+1}(a,b)
    =
    -2 \lambda_\infty^2 \sum_{k=1}^{j} C_k(a,b) + \OL(b_0\lambda_0|a-b|^{-(d-2)-\kappa}).
  \end{equation}
  Note that
  \begin{equation}
    \sum_{{k=j_{ab}-1}}^{{N-1}}
    |r_{k}| C_{k+1}(0,0)
    \leq \OL(b_0\lambda_0|a-b|^{-(d-2+\kappa)})\sum_{k\geq j_{ab}}
    L^{-(d-2)j}
    \leq \OL(b_0\lambda_0|a-b|^{-2(d-2)-\kappa}).
  \end{equation}
  As a result, again by Lemma~\ref{lem:Vobs-recursion}, these bounds together then give
  \begin{align}
    q_{N}
    &= \sum_{k\leq N} [\eta_{k-1} C_{k}(a,b) - \lambda_\infty^2 C_k(a,b)^2]
      + \OL(b_0\lambda_0|a-b|^{-2(d-2)-\kappa})
      \nnb
    &= - \lambda_\infty^2\sum_{k\leq N} [2\sum_{l < k} C_{l}(a,b)C_k(a,b) + C_k(a,b)^2] + \OL(b_0\lambda_0|a-b|^{-2(d-2)-\kappa})
      \nnb
    &
      = - \lambda_\infty^2 \pa{ \sum_{k\leq N} C_k(a,b)}^2 + \OL(b_0\lambda_0|a-b|^{-2(d-2)-\kappa}).
      \nnb
    &
      = - \lambda_\infty^2 W_N(a-b)^2 + \OL(b_0\lambda_0|a-b|^{-2(d-2)-\kappa}).
  \end{align}
  We finally substitute these estimates into \eqref{e:ZNN-q-b}.
  Using also that $\ellabN^{-1}\ell_N^{-2} = \ell_{j_{ab}}^{2} = \OL(|a-b|^{-(d-2)})$,
  that $\ellxN^{-1}\ell_N^{-2} = \OL(1)$,
  that $\gamma_{x,N}=\gamma_\infty+\OL(b_0\lambda_0L^{-(d-2+\kappa)N})$ by \eqref{e:gammainfbd},
  and $\|\Kobs_N\|_N \leq \OL(b_0\lambda_0L^{-\kappa})$,
  we obtain
  \begin{align}
      \frac{\tilde Z^{\sigma_a\sigma_b}_{N,N}}{1+\tilde u_{N,N}}
      &= -\lambda_\infty^2 W_N(a-b)^2 + \gamma_\infty^2
        +\OL(b_0\lambda_0|a-b|^{-2(d-2)-\kappa}) +  \OL(b_0\lambda_0L^{-(d-2+\kappa)N})
      \nnb
      &\qquad
        +\frac{-2\lambda_\infty^2 W_N(a-b)
        + 2\lambda_{\infty}\gamma_{\infty} 
        }{1+\tilde u_{N,N}}t_N|\Lambda_N|^{-1}
      \nnb
      &\qquad
        + \frac{
        \OL(b_0\lambda_0L^{-\kappa N}m^{-2}|\Lambda_N|^{-1}) + \OL(b_0\lambda_0|a-b|^{-(d-2+\kappa)}m^{-2}|\Lambda_N|^{-1})}{1+\tilde u_{N,N}}
    \end{align}
  which gives   \eqref{e:2pt-b} since $\avg{\bar\psi_a\psi_a\bar\psi_b\psi_b} = \tilde Z_{N,N}^{\sigma_a\sigma_b}/(\lambda_0^2(1+\tilde u_{N,N}))$.
\end{proof}

\begin{proof}[Proof of Proposition~\ref{prop:2pt}]
  The proposition follows immediately from Lemma~\ref{lem:2pt} with the same $\clambda$ and $\cgamma$
  as in Proposition~\ref{prop:1pt}.
\end{proof}

\section{Proof of Theorems~\ref{thm:psi4} and \ref{thm:psi4-full}}
\label{sec:proofs}

\begin{proof}[Proof of Theorems~\ref{thm:psi4} and \ref{thm:psi4-full}]
  By summation by parts on the whole torus $\Lambda_N$, we have
  \begin{equation}
    y_0(\nabla \psi,\nabla\bar\psi)     + \frac{z_0}{2}\pB{ (-\Delta\psi,\bar\psi)+(\psi,-\Delta\bar\psi) }
    = (y_0+z_0) (\nabla \psi,\nabla\bar\psi).
  \end{equation}
  Given $m^2\geq 0$ and $b_0$ small, we choose $V_0^c(b_0,m^2)$ as in Theorem~\ref{thm:flow}.
  This defines the functions $s_0^c = y_0^c+z_0^c$ and $a_0^c$ in \eqref{e:thm-psi4-y0a0} with the required regularity properties.
  The claims for the correlation functions and the partition function then follow from
  Propositions~\ref{prop:ZNN}--\ref{prop:suscept} and
  \ref{prop:1pt}--\ref{prop:2pt}. The continuity of $u_{N}^c$
  follows from the continuity of $V_{0}^{c}$ and the continuity of
  the renormalisation group maps.
  
  For Theorem~\ref{thm:psi4}, note that the statements simplify by the
  assumption $m^2 \geq L^{-2N}$. Indeed, using that
  $(m^{2}|\Lambda_N|)^{-1} \leq L^{-(d-2)N}$ and
  $|a_N| \leq \OL(b_0L^{-(2+\kappa)N})$, by
  Proposition~\ref{prop:ZNN}, we have that
  $|\tilde a_{N,N}|\leq \OL(b_0L^{-(2+\kappa)N})$ and
  $|\tilde u_{N,N}| \leq \OL(b_0L^{-\kappa N})$.
\end{proof}

\appendix
\section{Random forests and the $\HH^{0|2}$ model}
\label{app:h02forest}

\subsection{Proof of Proposition~\ref{prop:h02-forest}}

For any graph $G=(\Lambda,E)$ with edge weights $(\beta_{xy})$ and
vertex weights $(h_x)$, the partition function appearing in
\eqref{e:P-forest} can be generalised to
\begin{equation} \label{e:h02forest1}
  Z_{\beta,h}
  = \sum_{F\in\cF} \prod_{xy\in F}\beta_{xy}\prod_{T\in F} (1+\sum_{x\in T}h_{x}),
\end{equation}
where $\cF$ is the set of forest subgraphs of $G$.
Recall from
the discussion above \eqref{eq:order} that expanding the product over $T$
in \eqref{e:h02forest1} can be interpreted as choosing, for each $T$, either (i) a
root vertex $x\in T$ with weight $h_{x}$ or (ii) leaving $T$
unrooted. This interpretation will be used in Lemma~\ref{lem:DecayH}.

By \cite[Theorem~2.1]{MR4218682} (which follows \cite{MR2110547}),
\begin{equation} \label{e:h02forest2}
  Z_{\beta,h}
  =
  \int \prod_{x\in\Lambda} \partial_{\eta_x}\partial_{\xi_x} \frac{1}{z_x} e^{\sum_{xy}\beta_{xy}(u_x\cdot u_y+1) - \sum_{x} h_x(z_x-1)}.
\end{equation}
Moreover, by \cite[Corollary~2.2]{MR4218682}, if $h=0$ then
\begin{equation} \label{e:Pxconny}
  \P_{\beta,0}[x\conn y] = -\avg{u_0\cdot u_x}_{\beta,0}
  = -\avg{z_0z_x}_{\beta,0} = \avg{\xi_x\eta_y}_{\beta,0}
  = 1-\avg{\xi_x\eta_x\xi_y\eta_y}_{\beta,0}.
\end{equation}
Proposition~\ref{prop:h02-forest} follows easily from this.
For convenience, we restate the proposition as follows.
In the statement and throughout this appendix, inequalities like $\beta\geq 0$
are to be interpreted pointwise, i.e., $\beta_{xy}\geq 0$ for all edges $xy$.

\begin{proposition} \label{prop:h02-forest-app}
  For any finite graph $G$, 
  any $\beta \geq 0$ and $h \geq 0$,
  \begin{align} \label{e:zforest-app}
    \P_{\beta,h}[0\conn \ghost]
    &= \avg{z_0}_{\beta,h},
    \\
    \label{e:xietaforest-app}
    \P_{\beta,h}[0\conn x, 0 \nconn \ghost]
    &= \avg{\xi_0\eta_x}_{\beta,h},
    \\
    \label{e:u0uxforest-app}
    \P_{\beta,h}[0\conn x]+\P_{\beta,h}[0\nconn x, 0\conn \ghost, x\conn\ghost]
    &= -\avg{u_0\cdot u_x}_{\beta,h}
      ,
  \end{align}
  and the normalising constants in \eqref{e:P-forest} and \eqref{e:h02-def} are equal.
  In particular,
  \begin{equation} \label{e:P0x4pt-app}
    \P_{\beta,0}[0\conn x] = -\avg{u_0\cdot u_x}_{\beta,0} = -\avg{z_0z_x}_{\beta,0} = \avg{\xi_0\eta_x}_{\beta,0} = 1-\avg{\xi_0\eta_0\xi_x\eta_x}_{\beta,0}.
  \end{equation}
\end{proposition}

\begin{proof}[Proof of Proposition~\ref{prop:h02-forest-app}]
  For notational ease, we write the proof for constant $h$.
  The equality of the normalising constants is a special case of \eqref{e:h02forest2}.
  To see \eqref{e:zforest-app},
  we use that $(z_{0}-1)^2=0$ so that $z_{0}=1-(1-z_{0})=e^{-(1-z_{0})}$. As a result
  $\avg{z_{0}}_{\beta,h} = Z_{\beta,h-1_{0}} / Z_{\beta,h}$, and \eqref{e:h02forest1} gives
  \begin{equation} \label{e:zforest-pf}
    \avg{z_{0}}_{\beta,h} = \E_{\beta,h} \frac{h|T_0|}{1+h|T_0|} = \P_{\beta,h}[0\conn\ghost].
  \end{equation}
  Similarly, $\avg{z_0z_x} = Z_{\beta,h-1_0-1_x}/Z_{\beta,h}$ and thus \eqref{e:h02forest1} shows that
  \begin{align} \label{e:zzforest-pf}
    \avg{z_0z_x}_{\beta,h}
    &=\E_{\beta,h} \frac{-1+h|T_{0}|}{1+h|T_{0}|}1_{0\conn x} + \E_{\beta,h} \frac{h|T_{0}|}{1+h|T_{0}|} \frac{h|T_{x}|}{1+h|T_{x}|}1_{0\not\conn x}
      \nnb
  &= \P_{\beta,h}[0\conn x] -2\P_{\beta,h}[0\conn x,
    0\not\conn \ghost] + \P_{\beta,h}[0\not\conn x,
    0\conn \ghost,     x\conn \ghost].
\end{align}

To see \eqref{e:u0uxforest-app}, we note that the left-hand side is
the connection probability in the extended
graph $G^{\ghost}$.  From
\eqref{e:Pxconny} with $\beta_{xy} = \beta$ for $x,y\in\Lambda$ and
$\beta_{x\ghost} = h$ for $x\in\Lambda$ we thus obtain the claim:
  \begin{equation} \label{e:u0uxforest-pf}
    -\avg{u_0\cdot u_x}_{\beta,h}  =
    \P_{\beta,h}[0\conn x]+\P_{\beta,h}[0\nconn x, 0\conn \ghost, x\conn\ghost]
    .
  \end{equation}

  To see \eqref{e:xietaforest-app}, we combine \eqref{e:zzforest-pf} and \eqref{e:u0uxforest-pf} to get
  \begin{equation}
    2\avg{\xi_0\eta_x}_{\beta,h} = -\avg{u_{0}\cdot u_{x}}_{\beta,h} - \avg{z_{0}z_{x}}_{\beta,h} = 2\P_{\beta,h}[0\conn x,  0\not\conn \ghost].
  \end{equation}

  Finally, \eqref{e:P0x4pt-app} is \eqref{e:Pxconny}.
  This completes the proof.
\end{proof}

The extended 
graph $G^{\ghost}$ allows $z$-observables
to be interpreted in terms of edges connecting vertices in the base
graph $G$ to $\ghost$. To state this, we denote by $\{x\ghost\}$ the event
the edge between $x$ and $\ghost$ is present.
The next lemma will be used in Appendix~\ref{app:infvol}.

\begin{proposition}
  \begin{align} \label{e:z0edge}
    h_{0}\avg{z_0-1}_{\beta,h} &= \P_{\beta,h}[0\ghost]
                          \\
    h_{0}h_{x}\avg{z_0-1;z_x-1}_{\beta,h} &=
    \P_{\beta,h}[0\ghost,x\ghost]-\P_{\beta,h}[0\ghost]\P_{\beta,h}[x\ghost]
  \end{align}
\end{proposition}

\begin{proof}
  As discussed above, after expanding the product in \eqref{e:h02forest1} the external fields $h_{x}$ can be viewed as
  edge weights for edges from $x$ to $\ghost$.  With this in mind the formulas follow by
  differentiating~\eqref{e:h02forest2}.
\end{proof}

\subsection{High-temperature phase and positive external field}
\label{sec:th0}

\begin{proposition}
  \label{prop:theta0}
  If $\beta<p_{c}(d)/(1-p_{c}(d))$, then $\theta_{d}(\beta)=0$.
    Moreover, there is a $c=c(\beta)>0$ such that
    $\P^{\Lambda_{N}}_{\beta,0}[0\leftrightarrow x] \leq e^{-c|x|}$.
\end{proposition}

\begin{proof}
  In finite volume, Holley's inequality implies the stochastic
  domination
  $\P^{\Lambda_{N}}_{\beta,h}\preceq \P^{\Lambda_{N}}_{p,r}$, where
  the latter measure is Bernoulli bond percolation on the  extended
  graph $G^{\ghost}$ with $p=\beta/(1+\beta)$ and $r=h/(1+h)$, see
    \cite[Appendix~A]{MR4218682}.  In particular,
  \begin{equation}
    \P^{\Lambda_{N}}_{\beta,h}[0\conn \ghost] \leq \P^{\Lambda_{N}}_{p,r}[0\conn\ghost].
  \end{equation}
  Since each edge to the ghost is chosen independently with
  probability $r$, this latter quantity is
  \begin{equation}
    \P^{\Lambda_{N}}_{p,r}[0\conn\ghost] = \sum_{n=1}^{|\Lambda_{N}|}
    \P^{\Lambda_{N}}_{p,r}[\abs{C_0}=n](1-(1-r)^{n}) \leq r
    \E^{\Lambda_{N}}_{p,r}\abs{C_0}
  \end{equation}
  since $1-(1-r)^{n}\leq rn$ for $0\leq r\leq 1$. Here $C_0$ is the cluster
  of the origin on the torus without the ghost site, so
  $\E^{\Lambda_{N}}_{p,r}\abs{C_0}=\E^{\Lambda_{N}}_{p,0}\abs{C_0}$.
  Now suppose $\beta$ is such that $p<p_{c}(d)$. Then
  the right-hand side
  is finite and uniformly bounded in $N$. Hence
  \begin{equation}
    \theta_d(\beta) = 
    \lim_{h\to 0}\lim_{N\to\infty}\P^{\Lambda_{N}}_{\beta,h}[0\conn \ghost] 
    \leq 
    \lim_{r\to 0}r \sup_{N}\E^{\Lambda_{N}}_{p,0}\abs{C_0} = 0.
  \end{equation}
  The second claim follows from stochastic domination, as when
  $p<p_{c}(d)$ bond percolation has exponentially decaying
  connection probabilities~\cite{MR1707339}.
\end{proof}

\begin{lemma}
  \label{lem:DecayH}
  Let $h>0$ and suppose that for all $x$, $h_x= h$.  Then there are $c,C>0$ depending on $d,\beta,h$ such that
  \begin{equation}
    \label{eq:DecayH1}
    \P^{\Lambda_{N}}_{\beta,h}[0\conn x,0\ghost] \leq Ce^{-c|x|},
    \qquad 
    \P^{\Lambda_{N}}_{\beta,h}[0\conn x,0\nconn \ghost] \leq Ce^{-c|x|}
    .
  \end{equation}
\end{lemma}
\begin{proof}
  We begin with the inequality on the left
  of~\eqref{eq:DecayH1}. Define $\cF(0\conn x)$ to be the set of
  forests in which both $0$ is connected to $x$ 
  and $T_{0}$ is rooted at $0$, and $\cF$ the set of all forests. In
  this argument we treat $\cF$ as being a set of (possibly) rooted
  forests, i.e., we identify edges to $\ghost$ with roots.  Without
  loss of generality 
  we may assume $x\cdot e_{1} \geq \alpha |x|$ for a
  fixed $\alpha>0$. Note that if $F\in \cF(0\conn x)$
  there is a unique path $\gamma_{F}$ from $0$ to $x$ in $F$, and there are
  at least $\alpha |x|$ edges of the form $\{u,u+e_{1}\}$ in $\gamma_{F}$.

  We define a map $S\colon \cF(0\conn x) \to 2^{\cF}$ by, for
  $F\in \cF(0\conn x)$,
  \begin{enumerate}
  \item choosing a subset $\{u_{i},v_{i}\}$ of the edges
    $\{ \{u,v\} \in \gamma_{F} \mid v=u+e_{1}\}$, and
  \item removing each $\{u_{i},v_{i}\}$ and rooting the tree containing
    $v_{i}$ at $v_{i}$.
  \end{enumerate}
  Thus $S(F)$ is the set of forests that results from all possible
  choices in the first step. The second step does yield an element of
  $2^{\cF}$ since $T_{0}$ is rooted at $0$, so it cannot be the case
  that the tree containing $v_{i}$ is already rooted (connected to
  $\ghost$).

  The map $S$ is injective, meaning that given
  $\bar F\in \bigcup_{F\in \cF(0\conn x)}S(F)$ 
  there is a unique $F$ such that $\bar F \in S(F)$.
  Indeed, given
  $\bar F\in S(F)$, $F$ can be 
  reconstructed as follows. In $\bar F$, either the tree containing
  $x$ contains $0$, or else it is
  rooted at a unique vertex $v'$ and it is not connected to $u'=v'-e_{1}$. Set
  $\bar F'=\bar F\cup \{u',v'\}$. The previous sentence applies to
  $\bar F'$ as well, and continuing until a connection to $0$ is
  formed we recover $F$. This reconstruction was independent of
  $F$, and hence if $\bar F_{1}=\bar F_{2}$,
  $\bar F_{i}\in S(F_{i})$, we have $F_{1}=F_{2}$.

  Let $w(F)= h \beta^{F}\prod_{T\neq T_0}(1+h|V(T)|)$.
  Then for $\bar F \in S(F)$, $w(\bar F) = w(F) (\frac{h}{\beta})^{k}$ if
  $\bar F$ had $k$ edges removed. Hence if the connection from $0$ to
  $x$ in $F$ has $k$ edges of the form $\{u,v\}$, $v=u+e_{1}$, 
  \begin{equation}
    \label{eq:w}
    \sum_{\bar F\in S(F)}w(\bar F) = (1+\frac{h}{\beta})^{k}w(F).
  \end{equation}
  Let $\cF_{k}(x) \subset \cF(0\conn x)$ be the set of forests where the connection
  from $0$ to $x$ contains $k$ edges of the form $\{u,v\}$,
  $v=u+e_{1}$. We have the lower bound
  \begin{equation}
    \label{eq:pfbound}
    Z_{\beta,h}^{\Lambda_{N}} = \sum_{F\in\cF} \beta^{F}\prod_{T\in
      F}(1+h|V(T)|) \geq \sum_{k\geq 0} \sum_{F\in \cF_{k}(x)} \sum_{\bar
      F\in S(F)} w(\bar F) 
  \end{equation}
  since $S$ is injective and all of the summands are
  non-negative. Hence we obtain, using \eqref{eq:w},
  \begin{align}
    \P^{\Lambda_{N}}_{\beta,h}[0\conn x,0\ghost] 
    &\leq \frac{\sum_{k\geq \alpha |x|} \sum_{F\in \cF_{k}(x)}
      w(F)}{\sum_{k\geq 0} \sum_{F\in \cF_{k}(x)} \sum_{\bar
      F\in S(F)} w(\bar F)} \nnb
    &= \frac{\sum_{k\geq \alpha |x|} \sum_{F\in \cF_{k}(x)}
      (1+\frac{h}{\beta})^{-k}\sum_{\bar F\in S(F)}w(\bar
      F)}{\sum_{k\geq 0} \sum_{F\in \cF_{k}(x)} \sum_{\bar 
      F\in S(F)} w(\bar F)}
    \leq (1+\frac{h}{\beta})^{-\alpha |x|}.
    \label{e:x0g}
  \end{align}
  
  A similar argument applies when $0\nconn\ghost$; this condition is
  used in the second step defining $S$ to ensure the trees containing
  the vertices $v_{i}$ are not already connected to $\ghost$. In this
  case the weight $w(F)$ does not have the factor $h$, but the
  remainder of the argument is identical.
\end{proof}
\begin{proposition}
  \label{prop:DecayH}
  Let $h>0$ and suppose that for all $x$, $h_x= h$.
  Then there are $c,C>0$ depending on $d,\beta,h$ such that 
  \begin{equation}
    \label{eq:DecayH2}
    \P^{\Lambda_{N}}_{\beta,h}[0\conn x] \leq Ce^{-c |x |}.
  \end{equation}
\end{proposition}
\begin{proof}
  Since
  \begin{equation}
    \P^{\Lambda_{N}}_{\beta,h}[0\conn x] =
    \P^{\Lambda_{N}}_{\beta,h}[0\conn x, 0\conn\ghost] +
    \P^{\Lambda_{N}}_{\beta,h}[0\conn x, 0\nconn\ghost],
  \end{equation}
  it is enough to estimate the first term, as the second is covered by
  Lemma~\ref{lem:DecayH}. Note
 \begin{align}
    \P^{\Lambda_{N}}_{\beta,h}[0\conn x, 0\conn\ghost] 
    =
      \sum_{y} \P^{\Lambda_{N}}_{\beta,h} [1_{0\conn x}1_{0\conn y}1_{y\ghost}]    = \sum_{y} \P^{\Lambda_{N}}_{\beta,h} [1_{0\conn x}1_{0\conn y}1_{0\ghost}]. 
  \end{align}
  where the first equality follows from the fact that the only one vertex per component may connect to $\mathfrak g$, and the second follows from exchangeability of the choice of root.
  Examining the rightmost expression, there are most $c_{d} |x |^{d}$ summands in which $|y |\leq  |x|$;
  for these terms we drop the condition $0\conn y$. For the rest we
  drop $0\conn x$. This gives,  by Lemma~\ref{lem:DecayH},
  \begin{equation}
    \P^{\Lambda_{N}}_{\beta,h}[0\conn x, 0\conn\ghost] 
    \leq C |x|^{d}e^{-c |x |} + \sum_{ |y |> |x |} Ce^{-c |y |} \leq
    Ce^{-c |x |},
  \end{equation}
  where $c,C$ are changing from location to location but depend on
  $d,\beta,h$ only.
\end{proof}

\subsection{Infinite volume limit}
\label{app:infvol}

We now discuss weak limits $\P^{\Z^d}_{\beta}$ obtained
by (i) first taking a (possibly subsequential) infinite-volume weak
limit $\P^{\Z^d}_{\beta,h} = \lim_{N}\P^{\Lambda_N}_{\beta,h}$ and
(ii) subsequently taking a (possibly subsequential) limit
$\P^{\Z^d}_{\beta}=\lim_{h\downarrow 0 }\P^{\Z^d}_{\beta,h}$. We
do not explicitly indicate the convergent subsequence chosen as what
follows applies to any fixed choice.
Define 
\begin{equation}
  \label{eq:thetaNh}
  \theta_{d,N}(\beta,h) \bydef
  \P^{\Lambda_N}_{\beta,h}[0\conn\ghost]=1-h^{-1}\P^{\Lambda_N}_{\beta,h}[0\ghost]
\end{equation}
where the second equality is due to \eqref{e:z0edge}.
Since this last display only involves cylinder events,
\begin{equation}
  \lim_{N\to\infty}\theta_{d,N}(\beta,h) = 1-h^{-1}\P^{\Z^d}_{\beta,h}[0\ghost] \bydef \theta_{d}(\beta,h),
\end{equation}
where the last equality defines $\theta_{d}(\beta,h)$.

\begin{proposition}
  \label{prop:theta}
  Assume $\lim_{h\downarrow 0}\theta_{d}(\beta,h)=\theta_{d}(\beta)$
  exists. Then
  \begin{equation}
    \label{eq:theta}
    \P^{\Z^d}_{\beta}[|T_{0}|=\infty] = \theta_{d}(\beta).
  \end{equation}
\end{proposition}
\begin{proof}
  Write $\P_{\beta,h}=\P^{\Z^d}_{\beta,h}$. We claim that
  \begin{equation}
    \P_{\beta,h}[0\ghost] = \sum_{n\geq 1} \P_{\beta,h}[|T_{0}|=n] \frac{h}{1+nh},
  \end{equation}
  and hence, since $\theta_{d}(\beta,h) = 1-h^{-1}\P_{\beta,h}[0\ghost]$,
  \begin{equation}
    \theta_{d}(\beta,h) = 1-\sum_{n\geq 1} \P_{\beta,h}[|T_{0}|=n] \frac{1}{1+nh}.
  \end{equation}
  Granting the claim, by dominated convergence we obtain
  \begin{equation}
    \P_{\beta,0}[|T_{0}|<\infty] = \sum_{n\geq 1}\P_{\beta,0}[|T_{0}|=n]=1-\theta_{d}(\beta),
  \end{equation}
  as desired. To prove the claim, rewrite it as
  \begin{equation}
    \P_{\beta,h}[|T_{0}|=\infty, 0\ghost]=\lim_{r\to\infty}\P_{\beta,h}[|T_{0}|\geq r, 0\ghost]=0.
  \end{equation}
  The probabilities inside the limit are probabilities of cylinder
  events, and hence are limits of finite volume probabilities. For a
  fixed $r$ the probability is at most $h/(1+rh)$ in finite volume,
  which vanishes as $r\to\infty$.
\end{proof}

\section{Finite range decomposition}
\label{app:decomp}

In this appendix, we give the precise references for the construction of the finite range decomposition~\eqref{e:C-def}.
The general method we use was introduced in \cite{MR3129804}, and presented in the special case we use in \cite[Chapter~3]{MR3969983}
and we will use this reference.
For $t>0$, first recall the polynomials $P_t$ from \cite[Chapter~3]{MR3969983} (these polynomials are called $W_t^*$ in \cite{MR3129804}).
These are polynomials of degree bounded by $t$ satisfying
\begin{equation} \label{e:decomp-Pt}
  \frac{1}{\lambda} = \int_{0}^\infty t^2 P_t(\lambda) \, \frac{dt}{t},
  \qquad
  0 \leq P_t(u) \leq O_s(1+t^2u)^{-s}
\end{equation}
for any $s>0$ and $u \in [0,2]$.
Our decomposition \eqref{e:C-def} is defined by
\begin{align}
  C_1(x,y) &= \frac{1}{(2d+m^2)|\Lambda_N|}\sum_{k \in \Lambda_N^*}  e^{ik\cdot(x-y)} \int_{0}^{\frac12 L} t^2 P_t(\frac{\lambda(k)+m^2}{2d+m^2}) \, \frac{dt}{t} 
  \\
  C_j(x,y) &= \frac{1}{(2d+m^2)|\Lambda_N|}\sum_{k \in \Lambda_N^*}  e^{ik\cdot(x-y)} \int_{\frac12 L^{j-1}}^{\frac12 L^j} t^2 P_t(\frac{\lambda(k)+m^2}{2d+m^2}) \, \frac{dt}{t}
  \\
  C_{N,N}(x,y) &= \frac{1}{(2d+m^2)|\Lambda_N|}\sum_{k \in \Lambda_N^*} e^{ik\cdot(x-y)} \int_{\frac12 L^N}^\infty t^2 P_t(\frac{\lambda(k)+m^2}{2d+m^2}) \, \frac{dt}{t},
\end{align}
where $\lambda(k) = 4\sum_{j=1}^{d}\sin^{2}(k_{j}/2)$ and
$\Lambda_N^* \subset [-\pi,\pi)^d$ is the dual torus.  The estimates
for $C_1, \dots, C_{N-1}$ are straightforward from these Fourier
representations and can be found in
\cite[Chapter~3]{MR3969983}. We remark that in
\cite[Section~3.4]{MR3969983}, the torus covariances are defined by
periodisation of the finite range covariances on $\Z^d$; by Poisson
summation this is equivalent to the above definition.

The decomposition of $C_{N,N}$ in~\eqref{e:CNN-def} is defined by removing
the zero mode from $C_{N,N}$:
\begin{align}
  C_{N}(x,y) &= \frac{1}{(2d+m^2)|\Lambda_N|}\sum_{k \in \Lambda_N^*:k \neq 0} e^{ik\cdot(x-y)} \int_{\frac12 L^N}^\infty t^2 P_t(\frac{\lambda(k)+m^2}{2d+m^2}) \, \frac{dt}{t}
  \\
  t_N &= \frac{1}{2d+m^2} \int_{\frac12 L^N}^\infty t^2 P_t(\frac{m^2}{2d+m^2}) \, \frac{dt}{t},
\end{align}
from which \eqref{e:CNN-def} is immediate.
For $C_{N}$ estimates follows as in~\cite[Chapter~3]{MR3969983}:
\begin{align}
  |C_N(x,y)| &\leq \frac{1}{|\Lambda_N|} \sum_{k \in \Lambda_N^*: k \neq 0} \pa{\int_{\frac12 L^N}^\infty t^2 P_t(\frac{\lambda(k)+m^2}{2d+m^2}) \, \frac{dt}{t}}
    \nnb
  &\lesssim \frac{1}{|\Lambda_N|} \sum_{k \in \Lambda_N^*: k \neq 0} \pa{\int_{\frac12 L^N}^\infty t^{2}t^{-2s}|k|^{-2s} \, \frac{dt}{t}} 
  \nnb
  &\lesssim
    \frac{L^{2N}}{|\Lambda_N|} \sum_{k \in \Lambda_N^*: k \neq 0}L^{-2sN}|k|^{-2s}
    \lesssim L^{-(d-2)N} \int_1^\infty r^{-2s+d-1} dr \lesssim L^{-(d-2)N}
\end{align}
and analogously for the discrete gradients.
Finally, by \eqref{e:decomp-Pt}, 
\begin{align}
  t_N
  &= \frac{1}{(2d+m^2)} \int_{\frac12 L^N}^\infty t^2P_t(\frac{m^2}{2d+m^2}) \, \frac{dt}{t}
    \nnb
  &= \frac{1}{m^2}-
    \frac{1}{(2d+m^2)} \int_0^{\frac12 L^N} t^2P_t(\frac{m^2}{2d+m^2}) \, \frac{dt}{t}
      =\frac{1}{m^2} - O(L^{2N}).
\end{align}

\section*{Acknowledgements}

We thank David Brydges and Gordon Slade. This article would not have been possible in this form without their previous contributions
to the renormalisation group method.
We also thank them for their permission to include
Figure~\ref{fig:blocks} from \cite{MR3969983}. We thank the referees for their helpful comments.

R.B.\ was supported by the European Research Council under the European Union's Horizon 2020 research and innovation programme
(grant agreement No.~851682 SPINRG).
N.C.\ was supported by Israel Science Foundation grant number 1692/17.

\bibliographystyle{plain}
\bibliography{all}

\end{document}